\documentclass{article}
\usepackage{xr}
\usepackage{geometry,amsmath,amssymb,enumerate,bbm,color,theorem}
\usepackage[american]{babel}
\newcommand{\assign}{:=}
\newcommand{\longhookrightarrow}{{\lhook\joinrel\relbar\joinrel\rightarrow}}
\newcommand{\nin}{\not\in}
\newcommand{\tmdummy}{$\mbox{}$}
\newcommand{\tmmathbf}[1]{\ensuremath{\boldsymbol{#1}}}
\newcommand{\tmop}[1]{\ensuremath{\operatorname{#1}}}
\newcommand{\tmrsub}[1]{\ensuremath{_{\textrm{#1}}}}
\newcommand{\tmstrong}[1]{\textbf{#1}}
\newcommand{\tmtextit}[1]{{\itshape{#1}}}
\newenvironment{enumeratealpha}{\begin{enumerate}[a{\textup{)}}] }{\end{enumerate}}
\newenvironment{enumeratenumeric}{\begin{enumerate}[1.] }{\end{enumerate}}
\newenvironment{enumerateroman}{\begin{enumerate}[i.] }{\end{enumerate}}
\newenvironment{enumerateromancap}{\begin{enumerate}[I.] }{\end{enumerate}}
\newenvironment{itemizedot}{\begin{itemize} }{\end{itemize}}
\newenvironment{proof}{\noindent\textbf{Proof\ }}{\hspace*{\fill}$\Box$\medskip}
\definecolor{grey}{rgb}{0.75,0.75,0.75}
\definecolor{orange}{rgb}{1.0,0.5,0.5}
\definecolor{brown}{rgb}{0.5,0.25,0.0}
\definecolor{pink}{rgb}{1.0,0.5,0.5}
\newtheorem{corollary}{Corollary}
\newtheorem{definition}{Definition}
{\theorembodyfont{\rmfamily\small}\newtheorem{exercise}{Exercise}}
\newtheorem{lemma}{Lemma}
{\theorembodyfont{\rmfamily}\newtheorem{note}{Note}}
\newtheorem{proposition}{Proposition}
{\theorembodyfont{\rmfamily}\newtheorem{remark}{Remark}}
\newtheorem{theorem}{Theorem}
\begin{document}

\title{Carlson's $<_1$-relation on the class of epsilon
numbers}\author{Parm{\'e}nides Garc{\'i}a Cornejo}\maketitle

\begin{abstract}
  This article is a continuation of the ideas in {\cite{GarciaCornejo0}}.
  Based on the class of epsilon numbers, another binary relational $<^1$ in
  the ordinals is introduced. We will see that we can easily describe all the
  isomorphisms that are witnesses of $<^1$. Afterwards we will show that the
  isomorphisms witnesses of $<^1$ are ``essentially the same'' than the
  isomorphisms of the $<_1$-relation for finite subsets of ordinals that are
  closed under ``the cover construction''. Finally, through our understanding
  of $<^1$ we will see how it is that $<_1$ induces thinner $\kappa$-club
  classes of ordinals: Thinner than the thinnest class of ordinals seen in
  {\cite{GarciaCornejo0}}; indeed, as seen in {\cite{Wilken1}}, the thinnest
  class of ordinals induced by $<_1$ that we will obtain in this article will
  have as it's first element the Bachmann's ordinal $| \tmop{ID}_1 |$. An
  independent proof of this last fact will be provided in a coming article.
\end{abstract}

\section{The ordinals $\alpha$ satisfying $\alpha <_1 \alpha + \xi$, for some
$\xi \in [1, \alpha]$.}

We will first show a theorem and a corollary appearing in {\cite{Wilken1}}.
The proof of the theorem we present here is slightly different than the one
given by Wilken. Primarily, let's state the next proposition.

\begin{proposition}
  \label{alpha<less>_1alpha+l+1_implies_cofinal_sequence}Let $\alpha \in
  \tmop{OR}$ and $t \in [\alpha, \alpha 2)$. Let $l \in [0, \alpha)$ be such
  that $t = \alpha + l$. Then \\
  $\alpha <_1 t + 1 = (\alpha + l) + 1 \Longleftrightarrow$ there exists a
  strictly increasing sequence $(\xi_i)_{i \in I} \subset \alpha \cap
  \mathbbm{P}$ such that $\xi_i \underset{\tmop{cof}}{\longhookrightarrow}
  \alpha$ and $\forall i \in I.l < \xi_i \leqslant_1 \xi_i + l$.
\end{proposition}

\begin{proof}
  Direct from {\cite{GarciaCornejo0}} corollary
  \ref{cor_<less>^0_equivalences}.
\end{proof}

\begin{remark}
  Let $\alpha \in \mathbbm{P}$, $t \in [\alpha, \alpha 2)$ and $l \in [0,
  \alpha)$ be such that $t = \alpha + l$. For an additive principal number
  $\beta \in \mathbbm{P}$ with $\beta > l$, let's denote $t / \alpha \assign
  \beta /$ to the ordinal $\beta + l$; that is, $t / \alpha \assign \beta /$
  is simply the replacement of $\alpha$ by $\beta$ in $t = \alpha + l$. With
  this convention we can enunciate previous proposition
  \ref{alpha<less>_1alpha+l+1_implies_cofinal_sequence} as:
  $\alpha <_1 t + 1 \Longleftrightarrow \alpha \in \tmop{Lim} \{\beta \in
  \mathbbm{P} | l < \beta \wedge \beta \leqslant_1 t / \alpha \assign \beta
  /\}$.
\end{remark}

Now we present the theorem of {\cite{Wilken1}} that we mentioned before.

\begin{theorem}
  (Wilken).\label{Wilken_Theorem1}
  
  {\noindent}$\forall \alpha \in \tmop{OR} \forall \xi \in [1, \alpha) .
  \alpha <_1 \alpha + \xi \Longleftrightarrow \alpha = \omega^A$ with $A =
  \omega^{\xi} \cdot s$ for some $s \in \tmop{OR}$, $s \neq 0$.
  
  {\noindent}Remark: In the previous line, we are NOT saying that the Cantor
  Normal Form of A is $\omega^{\xi} \cdot s$.
\end{theorem}

\begin{proof}
  We prove the theorem by induction on $(\tmop{OR}, <)$.
  
  Let $\alpha \in \tmop{OR}$ and suppose the claim of the theorem holds for
  any $\beta < \alpha$. \ \ \ \ \ \ \ {\tmstrong{(IH)}}
  
  We continue the proof by a side induction on $[1, \alpha)$.
  
  Let $\xi \in [1, \alpha)$ and assume the claim holds for any $z \in [1,
  \alpha) \cap \xi$. \ \ \ \ \ \ \ {\tmstrong{(SIH)}}
  
  Case $\xi = 1$. Then the claim holds by {\cite{GarciaCornejo0}} proposition
  \ref{characterization_of_alpha<less>_1alpha+1}.
  
  Case $\xi \in (1, \alpha) \cap \tmop{Lim}$.
  
  $\Longrightarrow)$ Assume $\alpha <_1 \alpha + \xi$. Then $\alpha <_1 \alpha
  + y$ for any $y \in [1, \xi)$ by $\leqslant_1$-connectedness. Then, by our
  side induction hypothesis, for any $y \in [1, \xi)$, $\alpha = \omega^A$,
  where $A = \omega^y \cdot s_y$ for some $s_y \neq 0$. From this follows that
  $A = \sup \{\omega^y \cdot s_y |y \in [1, \xi)\} \geqslant \omega^{\xi}$.
  Now, by Euclid's division algorithm for ordinals, there exist $s, \rho \in
  \tmop{OR}$ with $\rho < \omega^{\xi}$ such that $A = \omega^{\xi} \cdot s +
  \rho$. But since $\rho < \omega^{\xi}$ and \\
  $\xi \in (1, \alpha) \cap \tmop{Lim}$, then there exists $\delta \in [1,
  \xi)$ such that $\omega^{\delta} > \rho$. All this means that $A
  \underset{\text{\tmop{by} \tmop{our} \tmop{SIH}}}{=} \omega^{\delta} \cdot
  s_{\delta} = \omega^{\xi} \cdot s + \rho$, which implies that
  $\omega^{\delta}$ divides $\rho < \omega^{\delta}$; therefore $\rho = 0$ and
  $A = \omega^{\xi} \cdot s$.
  
  $\Longleftarrow)$. Assume $\alpha = \omega^A$ with $A = \omega^{\xi} \cdot
  s$ for some $s \in \tmop{OR}$, $s \neq 0$. Then, for any $y \in [1, \xi)$,
  we can write $A = \omega^y \cdot s_y$ for some $s_y \neq 0$. Then by our SIH
  we get $\forall y \in [1, \xi) . \alpha <_1 \alpha + y$. But the sequence
  $(\alpha + y)_{y \in [1, \xi)}$ is cofinal in $\alpha + \xi$, so by
  $\leqslant_1$-continuity we conclude that $\alpha <_1 \alpha + \xi$.
  
  Case $\xi \in (1, \alpha)$, $\xi = l + 1$ for some $l \in [1, \alpha)$.
  
  $\Longrightarrow)$ Assume $\alpha <_1 \alpha + l + 1$. Then by \ proposition
  \ref{alpha<less>_1alpha+l+1_implies_cofinal_sequence}, there is a strictly
  increasing sequence $(\xi_i)_{i \in I} \subset \tmop{OR}$ such that $\xi_i
  \underset{\tmop{cof}}{\longhookrightarrow} \alpha$ and $\forall i \in I.l <
  \xi_i <_1 \xi_i + l$. This implies, by our IH, that \\
  $\forall i \in I. \xi_i = \omega^{\omega^l \cdot s_i}$ for some $s_i \neq
  0$. Moreover, note that since $(\xi_i)_{i \in I}$ is strictly increasing,
  then the sequence $(s_i)_{i \in I}$ has to be also strictly increasing, and
  therefore $\sigma \assign \sup \{s_i |i \in I\} \in \tmop{Lim}$.
  
  On the other hand, consider $\sigma =_{\tmop{CNF}} \omega^{S_1} a_1 + \ldots
  + \omega^{S_m} a_m$. Then $S_m \neq 0$ (because $\sigma \in \tmop{Lim}$) and
  so we can write $\sigma = \omega \cdot s$ for some $s \in \tmop{OR}$, $s
  \neq 0$. From this and the previous paragraph we get that $\alpha = \sup
  \{\xi_i |i \in I\} = \omega^{\sup \{\omega^l \cdot s_i |i \in I\}} =
  \omega^{\omega^l \cdot \sup \{s_i |i \in I\}} = \omega^{\omega^l \cdot
  \sigma} = \omega^{\omega^l \cdot \omega \cdot s} = \omega^{\omega^{l + 1}
  \cdot s}$ for some $s \neq 0$.
  
  $\Longleftarrow)$. Assume $\alpha = \omega^A$ with $A = \omega^{l + 1} \cdot
  s = \omega^l \cdot (\omega s)$ for some $s \neq 0$.
  
  Take an arbitrary finite $B \subset_{\tmop{fin}} \alpha + l + 1$. Then $B =
  \{a_1 < \ldots < a_n < \alpha + b_1 < \ldots < \alpha + b_p \}$ for some
  $a_i, b_j$. Without loss of generality (see {\cite{GarciaCornejo0}}
  proposition \ref{iso.restriction} in the appendices section) we can assume
  that $\{a_n = \alpha, b_p = l, \alpha + l, b_1, \ldots, b_p \} \subset B$.
  
  We want to define an ($<, <_1, +$)-isomorphism $h : B \longrightarrow h [B]
  \subset \alpha$ with $h|_{\alpha} = \tmop{Id}_{\alpha}$. In order to achieve
  this, we need to do first the following observation: Let $(s_j)_{i \in J}
  \subset [1, \omega s)$ be a sequence such that $s_j
  \underset{\tmop{cof}}{\longhookrightarrow} \omega s$ (this is possible
  because $\omega s \in \tmop{Lim}$). Consider the sequence $(\gamma_j)_{i \in
  J}$, where $\gamma_j \assign \omega^{\omega^l \cdot s_j} < \omega^{\omega^l
  \cdot (\omega s)} = \alpha$. Then $\gamma_j
  \underset{\tmop{cof}}{\longhookrightarrow} \alpha$ and by our (IH), $\forall
  j \in J. \gamma_j <_1 \gamma_j + l + 1$. This shows that $\alpha \in
  \tmop{Lim} \{\gamma \in \tmop{OR} | \gamma <_1 \gamma + l\}$.
  
  On the other hand, let $C \assign \{m (a) |a \in (B \cap \alpha) \wedge m
  (a) < \alpha\}$. Since $C$ is finite and we know that $\alpha \in \tmop{Lim}
  \{\gamma \in \tmop{OR} | \gamma <_1 \gamma + l\}$, then $\emptyset \neq
  (a_{n - 1}, \alpha) \cap (\max C, \alpha) \cap \{\gamma \in \tmop{OR} |
  \gamma <_1 \gamma + l\}$. Take $\rho \in (a_{n - 1}, \alpha) \cap (\max C,
  \alpha) \cap \{\gamma \in \tmop{OR} | \gamma <_1 \gamma + l\}$. Note that
  $\rho <_1 \rho + 1$ by $\leqslant_1$-connectedness and so $\rho \in
  \tmop{Lim} \mathbbm{P} \subset \mathbbm{P}$ by {\cite{GarciaCornejo0}}
  proposition \ref{characterization_of_alpha<less>_1alpha+1}.

  We define the function $h : B \longrightarrow h [B]$ as
  
  $h (a_k) \assign a_k$ for any $k \in [1, n - 1]$,
  
  $h (b_k) \assign b_k$ for any $k \in [1, p]$,
  
  $h (\alpha) \assign \rho$ and
  
  $h (\alpha + b_k) \assign \rho + b_k$ for any $k \in [1, p]$.

  By the definition of $h$, it is clear that $h|_{\alpha} =
  \tmop{Id}_{\alpha}$ and that $h : B \longrightarrow h [B]$ is bijective. Now
  we show it is an ($<, <_1, +$)-isomorphism: The proof that $h$ is an ($<,
  +$)-isomorphism is essentially the same done within the proof of
  {\cite{GarciaCornejo0}} proposition
  \ref{characterization_of_alpha<less>_1alpha+1} (one has to check just some
  few more subcases).
  
  So let's prove that $h$ is an $<_1$-isomorphism. We need to see several
  cases:
  
  1. Since $\alpha = \omega^A$ and $A = \omega^{l + 1} \cdot s = \omega^l
  \cdot (\omega s)$, then by our SIH $\alpha <_1 \alpha + l$. On the other
  hand, note that $h (\alpha) <_1 h (\alpha + l) = h (\alpha) + l$ indeed
  holds, because $\rho \in \{\gamma \in \tmop{OR} | \gamma <_1 \gamma + l\}$.
  
  2. By 1. and $\leqslant_1$-connectedness, we have that $\alpha <_1 \alpha +
  b_k$ for any $k \in [1, p]$. But $\rho <_1 \rho + l$ and $\forall k \in [1,
  p] . \rho < \rho + b_k \leqslant \rho + l$ imply, by
  $\leqslant_1$-connectedness, that $h (\alpha) <_1 h (\alpha + b_k) = h
  (\alpha) + b_k$ also holds for any $k \in [1, p]$.
  
  3. Let $\alpha + b_i, \alpha + b_j \in (B \backslash \{\alpha\})$ with
  $\alpha + b_i < \alpha + b_j$ be arbitrary. Then $\alpha + b_i \nless_1
  \alpha + b_j$ by {\cite{GarciaCornejo0}} corollary
  \ref{corollary1.02.11.2008}, and because of the same reason $h (\alpha +
  b_i) = h (\alpha) + b_i \nless_1 h (\alpha) + b_j = h (\alpha + b_j)$.
  
  4. Let $a_i, a_j \in B \cap \alpha$ with $a_i < a_j$ be arbitrary. Then
  clearly \\
  $a_i <_1 a_j \Longleftrightarrow h (a_i) = a_i <_1 a_j = h (a_j)$.
  
  5. Let $a_i \in B \cap \alpha$. If $a_i <_1 \alpha$, then $a_i <_1 h
  (\alpha)$ holds by $\leqslant_1$-connectedness (because $a_i < \rho <
  \alpha$). If $a_i \nless_1 \alpha$, this means $m (a_i) \in C$ and therefore
  $m (a_i) < \rho$, that is, $h (a_i) = a_i \nless_1 \rho = h (\alpha)$.
  
  6. Let $a_i \in B \cap \alpha$ and $\alpha + b_j \in (B \backslash
  \{\alpha\})$. If $a_i <_1 \alpha + b_j$, then, using \\
  $a_i < \rho + b_j \underset{\tmop{because} \alpha \in \mathbbm{P}}{<}
  \alpha < \alpha + b_j$, we conclude that $h (a_i) = a_i <_1 \rho + b_j = h
  (\alpha + b_j)$ by $\leqslant_1$-connectedness. If $a_i \nless_1 \alpha +
  b_j$, then $a_i \nless_1 \alpha$ (because, by 2., we already know that
  $\alpha <_1 \alpha + b_j$; so $a_j <_1 \alpha$ would imply $a_i <_1 \alpha +
  b_j$ by $\leqslant_1$-transitivity). This means $m (a_i) \in C$ and
  therefore $m (a_i) < \rho$. Hence $h (a_i) = a_i \nless_1 \rho + b_j = h
  (\alpha + b_j)$.

  1., 2., 3., 4., 5. and 6. show that $h$ is also an $<_1$-isomorphism.
\end{proof}

\begin{corollary}
  \label{Wilken_Corollary1}(Wilken). $\forall \alpha \in \tmop{OR} . \alpha
  <_1 \alpha 2 \Longleftrightarrow \alpha \in \mathbbm{E}$
\end{corollary}

\begin{proof}
  Not hard now. Left to the reader.
\end{proof}

\begin{corollary}
  \label{m(alpha)_for_alpha_not_epsilon_number}Let $\alpha \in \tmop{OR}$,
  $\alpha \nin \mathbbm{E}$ be such that $\alpha =_{\tmop{CNF}} \omega^{A_1}
  a_1 + \ldots + \omega^{A_n} a_n$.
  \begin{enumeratealpha}
    \item If $n \geqslant 2$ or $a_1 \geqslant 2$, then $m (\alpha) = \alpha$.
    
    \item If $n = 1 = a_1$ and $A_1 =_{\tmop{CNF}} \omega^{B_1} b_1 + \ldots +
    \omega^{B_s} b_s$, then $m (\alpha) = \alpha + B_s$. ($B_s$ could be
    zero).
  \end{enumeratealpha}
\end{corollary}

\begin{proof}
  $a)$. Suppose $n \geqslant 2$ or $a_1 \geqslant 2$. Then, by
  {\cite{GarciaCornejo0}} proposition
  \ref{characterization_of_alpha<less>_1alpha+1}, $\alpha \nless_1 \alpha +
  1$; then $m (\alpha) = \alpha$.
  
  {\noindent}$b)$. Suppose $n = 1 = a_1$ and $A_1 =_{\tmop{CNF}} \omega^{B_1}
  b_1 + \ldots + \omega^{B_s} b_s$.
  
  Case $B_s = 0$. By {\cite{GarciaCornejo0}} proposition
  \ref{characterization_of_alpha<less>_1alpha+1}, $\alpha \nless_1 \alpha +
  1$; then $m (\alpha) = \alpha = \alpha + 0 = \alpha + B_s$.
  
  Case $B_s \neq 0$. Since by hypothesis $\alpha \nin \mathbbm{E}$ and $\alpha
  = \omega^{\omega^{B_1} b_1 + \ldots + \omega^{B_s} b_s} \in \mathbbm{P}$,
  then $B_s < \alpha > 1$ and therefore $B_s + 1 < \alpha > B_s$. \ \ \ \ \ \
  \ (*)
  
  On the other hand, let $\delta_i$ be such that $B_s + \delta_i = B_i$ for
  any $i \in [1, s]$. Then \\
  $A_1 =_{\tmop{CNF}} \omega^{B_s + \delta_1} b_1 + \ldots + \omega^{B_s +
  \delta_{s - 1}} b_{s - 1} + \omega^{B_s} b_s = \omega^{B_s} \cdot
  (\omega^{\delta_1} b_1 + \ldots + \omega^{\delta_{s - 1}} b_{s - 1} + b_s)$;
  moreover, $A_1 \neq \omega^{B_s + 1} \cdot D$ for any $D \in \tmop{OR}$
  (because of the uniqueness of the Cantor Normal Form: if \\
  $A_1 = \omega^{B_s + 1} \cdot D$ for some $D \in \tmop{OR}$, then for $D
  =_{\tmop{CNF}} \omega^{D_1} d_1 + \ldots + \omega^{D_k} d_k$ one gets $A_1
  =_{\tmop{CNF}} \omega^{B_s + 1 + D_1} d_1 + \ldots + \omega^{B_s + 1 + D_k}
  d_k$, which is different than \\
  $\omega^{B_1} b_1 + \ldots + \omega^{B_{s - 1}} b_{s - 1} + \omega^{B_s}
  b_s$ because $B_s + 1 + D_k > B_s$). The previous and (*) imply, by theorem
  \ref{Wilken_Theorem1}, that $\alpha <_1 \alpha + B_s$ and $\alpha \nless_1
  \alpha + B_s + 1$. Hence $m (\alpha) = \alpha + B_s$.
\end{proof}

\section{$<_1$ in the intervals $[\varepsilon_{\gamma}, \varepsilon_{\gamma +
1})$}

Our interest now is to describe the solutions of $x <_1 \beta$ for $\beta > x
2$. The first thing to note is that, in case we are able to find some ordinal
$x$ such that $x <_1 \beta$ for some $\beta > x 2$, then by
$\leqslant_1$-connectedness $x <_1 x 2$, and therefore, by corollary
\ref{Wilken_Corollary1}, $x \in \mathbbm{E}$. This shows that the solutions of
the inequalities we are now interested, in case they exist, have to be epsilon
numbers. Because of this, {\tmstrong{we define}} $\tmmathbf{\tmop{Class} (1)
\assign \mathbbm{E}}$ and we aim at the description of epsilon numbers $x$
such that they satisfy something of the form $x <_1 \beta$, with $\beta \in
[x, x^+)$ and $x^+ \assign \min \{e \in \mathbbm{E} | e > x\}$; since we
restrict $\beta \in [x, x^+)$, we will informally say that we are studying the
relation $<_1$ in the intervals $[\varepsilon_{\gamma}, \varepsilon_{\gamma +
1})$.

\subsection{Substitutions}

In our previous work, whenever we asserted that for certain ordinals $\alpha,
\xi$ with $\xi \in \alpha$, it holds $\alpha <_1 \alpha + \xi$, we provided,
for every $B \subset_{\tmop{fin}} \alpha + \xi$, an ($<, <_1, +$)-isomorphism
$h : B \longrightarrow h [B] \subset \alpha$ such that $h|_{\alpha} =
\tmop{Id}_{\alpha}$. The important aspect we want to stress is that the
isomorphism we constructed had the following peculiarity: we looked for an
``adequate'' $\rho \in \alpha$, we defined $h (\alpha) \assign \rho$ (we can
always consider that $\alpha \in B$ by {\cite{GarciaCornejo0}} proposition
\ref{iso.restriction} in the appendices section), $h (a) \assign a$ for any $a
\in B \cap \alpha$, and for any $\alpha + l \in B$, we defined $h (\alpha + l)
\assign \rho + l$. So $h$ just ``substituted $\alpha$ by $\rho$ and leave the
rest as it was''. This suggests to study these kind of substitutions as
witnesses of the $<_1$-relation; in particular, this will play an essential
role for our study of $<_1$ in the intervals $[\varepsilon_{\gamma},
\varepsilon_{\gamma + 1})$.

\begin{definition}
  For $x \in \tmop{OR}$, let $\tmop{Ep} (x)$ be the (finite) set of epsilon
  numbers appearing in the Cantor Normal Form of $x$, that is,
  
  {\noindent}$\tmop{Ep} (x) \assign \left\{ \begin{array}{l}
    \{x\} \text{\tmop{if}} x \in \mathbbm{E}\\
    \tmop{Ep} (L_1) \cup \ldots \cup \tmop{Ep} (L_n)  \text{\tmop{if}} x \nin
    \mathbbm{E} \wedge x =_{\tmop{CNF}} L_1 l_1 + \ldots + L_n l_n \wedge (n
    \geqslant 2 \vee l_1 \geqslant 2)\\
    \tmop{Ep} (L) \text{\tmop{if}} x \nin \mathbbm{E} \wedge x =_{\tmop{CNF}}
    \omega^L
  \end{array} \right.$
\end{definition}

\begin{definition}
  Let $\alpha, e \in \mathbbm{E}$ and $x \in \tmop{OR}$. We define the
  substitution of $\alpha$ by $e$ in the Cantor Normal Form of $x$ (and we
  denote it as $x [\alpha \assign e]$) as:

  {\noindent}$x [\alpha \assign e] \assign \left\{ \begin{array}{l}
    x \text{\tmop{if}} x \in \mathbbm{E} \wedge x \neq \alpha\\
    e \text{\tmop{if}} x = \alpha\\
    L_1 [\alpha \assign e] l_1 + \ldots + L_n [\alpha \assign e] l_n
    \text{\tmop{if}} x \nin \mathbbm{E} \wedge x =_{\tmop{CNF}} L_1 l_1 +
    \ldots + L_n l_n \wedge (n \geqslant 2 \vee l_1 \geqslant 2)\\
    \omega^{L [\alpha \assign e]} \text{\tmop{if}} x \nin \mathbbm{E} \wedge x
    =_{\tmop{CNF}} \omega^L
  \end{array} \right.$

\end{definition}

As the reader can see, the substitution $x [\alpha \assign e]$ makes sense for
any $\alpha, e \in \mathbbm{E}$ and $x \in \tmop{OR}$. We will require later
the conditions $x \in \alpha^+$ and $\tmop{Ep} (x) \cap \alpha \subset e$ in
order to guarantee that $x [\alpha \assign e]$ is a Cantor Normal Form
already: the one obtained by simply exchanging in the Cantor Normal Form of
$x$ the epsilon number $\alpha$ by the epsilon number $e$.

\begin{proposition}
  \label{[alpha:=e]_proposition1}Let $\alpha, e \in \mathbbm{E}$.
  \begin{enumeratealpha}
    \item $\tmop{Ep} (x)$ is finite for any $x \in \tmop{OR}$.
    
    \item $0 < x [\alpha \assign e]$ for any $x \in \tmop{OR} \backslash
    \{0\}$.
    
    \item $x [\alpha \assign e] = x$ for any $x \in \alpha$.
  \end{enumeratealpha}
\end{proposition}

\begin{proof}
  Easy.
\end{proof}

\begin{proposition}
  \label{[alpha:=e]_proposition2}Let $\alpha, e \in \mathbbm{E}$ and $q, s \in
  \alpha^+$. Suppose that $\tmop{Ep} (q) \cap \alpha \subset e \supset
  \tmop{Ep} (s) \cap \alpha$. Then \\
  $q < s \Longleftrightarrow q [\alpha \assign e] < s [\alpha \assign e]$.
\end{proposition}

\begin{proof}
  Not hard.
\end{proof}

\begin{proposition}
  \label{[alpha:=e]_proposition3}Let $\alpha, e \in \mathbbm{E}$ and $s \in
  \alpha^+$.
  \begin{enumeratenumeric}
    \item If $\tmop{Ep} (s) \cap \alpha \subset e$ then $s [\alpha \assign e]
    \in e^+$, $\tmop{Ep} (s [\alpha \assign e]) \cap e = \tmop{Ep} (s) \cap
    \alpha$ and the ordinal $s [\alpha \assign e]$ is already in Cantor Normal
    Form.
    
    \item $\tmop{Ep} (s) \cap \alpha \subset e \Longleftrightarrow \tmop{Ep}
    (\omega^s) \cap \alpha \subset e$.
    
    \item If $\tmop{Ep} (s) \cap \alpha \subset e$ then $\omega^s [\alpha
    \assign e] = \omega^{s [\alpha \assign e]}$.
    
    \item If $s =_{\tmop{CNF}} A_1 a_1 + \ldots + A_m a_m$ then $\tmop{Ep} (s)
    \cap \alpha \subset e \Longleftrightarrow ( \bigcup_{1 \leqslant i
    \leqslant m} \tmop{Ep} (A_i) \cap \alpha) \subset e$.
  \end{enumeratenumeric}
\end{proposition}

\begin{proof}
  Not hard.
\end{proof}

\begin{proposition}
  \label{[alpha:=e]_proposition4}Let $\alpha, e \in \mathbbm{E}$ and $q, s \in
  \alpha^+$. Suppose $\tmop{Ep} (q) \cap \alpha \subset e \supset \tmop{Ep}
  (s) \cap \alpha$. Then
  \begin{enumeratealpha}
    \item $\tmop{Ep} (q + s) \cap \alpha \subset e$ and $(q + s) [\alpha
    \assign e] = q [\alpha \assign e] + s [\alpha \assign e]$.
    
    \item $\tmop{Ep} (q \cdot s) \cap \subset e$ and $(q \cdot s) [\alpha
    \assign e] = q [\alpha \assign e] \cdot s [\alpha \assign e]$.
    
    \item $s [\alpha \assign e] [e \assign \alpha] = s$
    
    \item If $s = a + c$ for some $a, c \in \tmop{OR}$, then $\tmop{Ep} (c)
    \cap \alpha \subset e$.
    
    \item If $s = a \cdot b$ for some $a, b \in \tmop{OR}$, then $\tmop{Ep}
    (b) \cap \alpha \subset e$.
  \end{enumeratealpha}
\end{proposition}

\begin{proof}
  Not hard.
\end{proof}

\begin{definition}
  \label{Definition_M(alpha,e)_Set}For $\alpha, e \in \mathbbm{E}$ we define
  $M (\alpha, e) \assign \{q \in \alpha^+ | \tmop{Ep} (q) \cap \alpha \subset
  e\}$.
\end{definition}

We can summarize our previous results in the following two corollaries:

\begin{corollary}
  \label{M(alpha,e)_sum_product_exponentiation_closed}Let $\alpha, e \in
  \mathbbm{E}$. Then:
  \begin{enumeratenumeric}
    \item $M (\alpha, e)$ is closed under the operations $+, \cdot, \lambda x.
    \omega^x$.
    
    \item $M (\alpha, e) \cap [\alpha, \alpha^+)$ is closed under the
    operations $+, \cdot, \lambda x. \omega^x$.
  \end{enumeratenumeric}
  
\end{corollary}

\begin{proof}
  {\color{orange} Left to the reader.}
\end{proof}

\begin{corollary}
  \label{main_[a:=e]_Isomorphism1}Let $\alpha, e \in \mathbbm{E}$. Then
  \begin{enumeratenumeric}
    \item The function \begin{tabular}{l}
      $f : M (\alpha, e) \longrightarrow f [M (\alpha, e)] \subset
      \tmop{OR}$\\
      \ \ \ \ \ \ $q \longmapsto q [\alpha \assign e]$
    \end{tabular} is an $(<, +, \cdot, \lambda x. \omega^x)$-isomorphism.
    
    \item If $e \leqslant \alpha$ then $M (e, \alpha) \cap [e, e^+) = [e,
    e^+)$ and the functions \begin{tabular}{l}
      $h : [\alpha, \alpha^+) \cap M (\alpha, e) \longrightarrow [e, e^+)$\\
      \ \ \ \ \ \ $q \longmapsto q [\alpha \assign e]$
    \end{tabular} and \begin{tabular}{l}
      $k : [e, e^+) \longrightarrow [\alpha, \alpha^+) \cap M (\alpha, e)$\\
      \ \ \ \ \ \ $q \longmapsto q [e \assign \alpha]$
    \end{tabular}are $(<, +, \cdot, \lambda x. \omega^x)$-isomorphisms with
    $h^{- 1} = k$.
  \end{enumeratenumeric}
\end{corollary}

\begin{proof}

  {\noindent}1.\\
  Proposition \ref{[alpha:=e]_proposition2} guarantees that $f$ preserves the
  relation $<$ (this, subsequently, implies that $f$ is injective, and
  therefore $f : M (\alpha, e) \longrightarrow f [M (\alpha, e)]$ is a
  bijection). Moreover, propositions \ref{[alpha:=e]_proposition3} and
  \ref{[alpha:=e]_proposition4} guarantee that $f$ preserves the operations
  $\lambda x. \omega^x, +, \cdot$ too. Finally, since by corollary
  \ref{M(alpha,e)_sum_product_exponentiation_closed} $M (\alpha, e)$ is ($+,
  \cdot, \lambda x. \omega^x$)-closed, then we do not have to worry about $f$
  preserving the \\
  ($+, \cdot, \lambda x. \omega^x$)-closure of it's domain $M (\alpha, e)$
  (that is, $f$ is an ($<, +, \cdot, \lambda x. \omega^x$)-isomorphism in the
  usual sense).
  
  {\noindent}2.\\
  {\color{orange} Left to the reader.}
\end{proof}

\begin{corollary}
  \label{[a:=e]_iso_B_contained_M(a,e)}Let $\alpha, e \in \mathbbm{E}$ and $B
  \subset M (\alpha, e)$. Then the function $h : B \longrightarrow h [B]$, \\
  $h (x) \assign x [\alpha \assign e]$ is an $(<, +, \cdot, \lambda x.
  \omega^x)$- isomorphism.
\end{corollary}

\begin{proof}
  By previous corollary \ref{main_[a:=e]_Isomorphism1}, we already know that
  $h$ preserves $<, +, \cdot$ and $\lambda x. \omega^x$. Moreover, the fact
  that $h$ preserves $<$ implies that $h$ is an injection, and therefore $h$
  is a bijection from it's domain to it's image. So it only remains to show
  that $h$ preserves the $(+, \cdot, \lambda x. \omega^x)$-closure of $B$.
  This is not hard: Let $\beta, \gamma \in B$. Let's denote as $\beta \Box
  \gamma$ to any of $\beta + \gamma$, $\beta \cdot \gamma$ or
  $\omega^{\gamma}$.
  
  Suppose $\beta \Box \gamma \in B$. Then $\beta \Box \gamma = \delta$ for
  some $\delta \in B$. Then $h (\beta) \Box h (\gamma) = h (\beta \Box \gamma)
  = h (\delta) \in h [B]$.
  
  Suppose $h (\beta) \Box h (\gamma) \in h [B]$. Then $\beta [\alpha \assign
  e] \Box \gamma [\alpha \assign e] = h (\beta) \Box h (\gamma) = h (\delta) =
  \delta [\alpha \assign e]$ for some $\delta \in B \subset M (\alpha, e)$. \
  (*). On the other hand, since $\beta, \gamma \in M (\alpha, e)$, then $\beta
  \Box \gamma \in M (\alpha, e)$ and \\
  $\beta [\alpha \assign e] \Box \gamma [\alpha \assign e] = (\beta \Box
  \gamma) [\alpha \assign e]$. \ (**). From (*) and (**) follow $(\beta \Box
  \gamma) [\alpha \assign e] = \delta [\alpha \assign e]$, and since the
  function $x \longmapsto x [\alpha \assign e]$ is a bijection in $M (\alpha,
  e)$, then $\beta \Box \gamma = \delta \in B$.
\end{proof}

\subsubsection{Substitutions and $<_1$ in intervals $(\varepsilon_{\gamma},
\varepsilon_{\gamma + 1})$.}

The next two results are the main reason why we are caring so much about our
substitutions $x \longmapsto x [\alpha \assign e]$.

\begin{proposition}
  \label{M(alpha,e)_Int_(alpha,alpha^+)_closure}Let $\alpha, e \in
  \mathbbm{E}$ and $A \assign (\alpha, \alpha^+) \cap M (\alpha, e)$. Then $A$
  is closed under the operations $+, \cdot, \lambda x. \omega^x$ and $m$.
\end{proposition}

\begin{proof}
  {\color{orange} Left to the reader.}
\end{proof}

\begin{corollary}
  \label{A_[alpha:=e]_isomorphisms}Let $\alpha, e \in \mathbbm{E}$ and $A
  \assign (\alpha, \alpha^+) \cap M (\alpha, e)$. Then
  \begin{enumeratenumeric}
    \item The function \begin{tabular}{l}
      $h : A \longrightarrow h [A] \subset (e, e^+)$\\
      \ \ $q \longmapsto q [\alpha \assign e]$
    \end{tabular} is an $(<, +, \cdot, \lambda x. \omega^x, m)$- isomorphism.
    
    \item If $\alpha \leqslant e$ then $A = (\alpha, \alpha^+)$ and then the
    function \\
    \begin{tabular}{l}
      $h : (\alpha, \alpha^+) \longrightarrow h [(\alpha, \alpha^+)] \subset
      (e, e^+)$\\
      \ \ \ \ \ \ $q \longmapsto q [\alpha \assign e]$
    \end{tabular} is an $(<, +, \cdot, \lambda x. \omega^x, m)$- isomorphism.
  \end{enumeratenumeric}
\end{corollary}

\begin{proof}
  Clearly 2. follows from 1.

  We prove 1. Let $q \in A$ be arbitrary.

  $\alpha < q \underset{\text{\tmop{proposition}
  \ref{[alpha:=e]_proposition2}}}{\Longrightarrow} e = \alpha [\alpha \assign
  e] < q [\alpha \assign e] \underset{\text{\tmop{proposition}
  \ref{[alpha:=e]_proposition3}}}{<} e^+$. This shows that $h [A] \subset (e,
  e^+)$.

  On the other hand, proposition \ref{[alpha:=e]_proposition2} guarantees that
  $h$ preserves the relation $<$. Moreover, propositions
  \ref{[alpha:=e]_proposition3} and \ref{[alpha:=e]_proposition4} guarantee
  that $h$ preserves the operations $\lambda x. \omega^x, +, \cdot$ too.

  Now we prove that $h$ preserves $m$ too.
  
  Since $q \in (\alpha, \alpha^+)$, then $q \nin \mathbbm{E}$. Then we have
  the following cases:
  
  Case $q =_{\tmop{CNF}} B_1 b_1 + \ldots + B_n b_n$ with $n \geqslant 2 \vee
  b_1 \geqslant 2$. Then $m (q) = q$ by corollary
  \ref{m(alpha)_for_alpha_not_epsilon_number}. On the other hand $q [\alpha
  \assign e] =_{\tmop{CNF}} B_1 [\alpha \assign e] b_1 + \ldots + B_n [\alpha
  \assign e] b_n$ with $n \geqslant 2 \vee b_1 \geqslant 2$; so, by corollary
  \ref{m(alpha)_for_alpha_not_epsilon_number}, $m (q [\alpha \assign e]) = q
  [\alpha \assign e] = m (q) [\alpha \assign e]$.
  
  Case $q =_{\tmop{CNF}} \omega^B$ with $B =_{\tmop{CNF}} \omega^{B_1} b_1 +
  \ldots + \omega^{B_{n - 1}} b_{n - 1} + \omega^Z b_n$. Then $m (q)
  = q + Z$ by corollary \ref{m(alpha)_for_alpha_not_epsilon_number};
  moreover, $\tmop{Ep} (Z) \cap \alpha \subset e$ by the proof of previous
  proposition \ref{M(alpha,e)_Int_(alpha,alpha^+)_closure}. \ \ \ \ \ \ \ (*).
  
  On the other hand $q [\alpha \assign e] =_{\tmop{CNF}} \omega^{B [\alpha
  \assign e]}$ with $B [\alpha \assign e] =_{\tmop{CNF}} \omega^{B_1 [\alpha
  \assign e]} b_1 + \ldots + \omega^{Z [\alpha \assign e]} b_n$; thus $m (q
  [\alpha \assign e]) \underset{\text{\tmop{corollary}
  \ref{m(alpha)_for_alpha_not_epsilon_number}}}{=} q [\alpha \assign e] + Z
  [\alpha \assign e] \underset{\text{by (*) and proposition } \ref{[alpha:=e]_proposition4}}{=} (q + Z) [\alpha
  \assign e] = m (q) [\alpha \assign e]$.
  
  All the previous shows that $h$ preserves $m$.

  Finally, since by proposition \ref{M(alpha,e)_Int_(alpha,alpha^+)_closure}
  $A$ is ($+, \cdot, \lambda x. \omega^x, m$)-closed, then we do not have to
  worry about $h$ preserving the ($+, \cdot, \lambda x. \omega^x, m$)-closure
  of it's domain $A$ (that is, $h$ is an \\
  ($<, +, \cdot, \lambda x. \omega^x, m$)-isomorphism in the usual sense).
\end{proof}

\begin{remark}
  \label{remark_m-iso_<less>_1-iso}The function $h : A \longrightarrow h [A]$
  of previous corollary \ref{A_[alpha:=e]_isomorphisms} is an
  $<_1$-isomorphism too: For any $\beta, \gamma \in A$, the ordinals $m
  (\beta), m (\gamma) \in A$ (because by proposition
  \ref{M(alpha,e)_Int_(alpha,alpha^+)_closure}, $A$ is $m$-closed) and we have
  that $\beta \leqslant_1
  \gamma\underset{\leqslant_1-\tmop{connectedness}}{\Longleftrightarrow}\beta
  \leqslant \gamma \leqslant m (\beta)\underset{\tmop{corollary }
  \ref{A_[alpha:=e]_isomorphisms}}{\Longleftrightarrow}$\\
  $h (\beta) \leqslant h (\gamma) \leqslant h (m (\beta)) = m (h
  (\beta))\underset{\leqslant_1-\tmop{connectedness}}{\Longleftrightarrow}h(\beta) \leqslant_1 h (\gamma)$.
\end{remark}

\subsection{The relation $<^1$.}

With the purpose of extending our understanding between the substitutions $x
\longmapsto [\alpha \assign e]$ and the $<_1$-relation, we introduce the
following

\begin{definition}
  For $\alpha, \beta \in \tmop{OR}$, $\alpha <^1 \beta$ means $\alpha < \beta$
  and $\forall Z \subset_{\tmop{fin}} \beta \exists \tilde{Z}
  \subset_{\tmop{fin}} \alpha \exists h$ such that
  
  $(i)$ \ \ $h : (Z, <, <_1, +, \lambda x. \omega^x) \rightarrow ( \tilde{Z},
  <, <_1, +, \lambda x. \omega^x)$ is an isomorphism.
  
  $(i i)$ \ $h|_{Z \cap \alpha} = \tmop{Id} |_{Z \cap \alpha}$, where
  $\tmop{Id} |_{Z \cap \alpha} : Z \cap \alpha \longrightarrow Z \cap \alpha$
  is the identity function.

  By $\alpha \leqslant^1 \beta$ we mean that $\alpha <_1 \beta$ or $\alpha =
  \beta$. We abbreviate $h|_{Z \cap \alpha} = \tmop{Id} |_{Z \cap \alpha}$ as
  $h|_{\alpha} = \tmop{Id} |_{\alpha}$.
\end{definition}

\begin{proposition}
  \label{<less>=^1_implies_<less>=_1}Let $\alpha, \beta, \gamma \in \tmop{OR}$
  and $(\xi_i)_{i \in I} \subset \tmop{OR}$. Then
  \begin{enumeratenumeric}
    \item $\alpha \leqslant^1 \beta \Longrightarrow \alpha \leqslant_1 \beta$.
    
    \item If $\alpha \leqslant \beta \leqslant \gamma \wedge \alpha
    \leqslant^1 \gamma$ then $\alpha \leqslant^1 \beta$. \ \ \ \ \ \ \ \ \ \ \
    \ \ ($\leqslant^1$-connectedness)
    
    \item If $\forall i \in I. \alpha \leqslant^1 \xi_i \wedge \xi_i
    \underset{\tmop{cof}}{\longhookrightarrow} \beta$ then $\alpha \leqslant^1
    \beta$. \ \ \ \ \ \ \ ($\leqslant^1$-continuity)
  \end{enumeratenumeric}
\end{proposition}

\begin{proof}
  1. follows direct from the definition of $\leqslant^1$. The proofs of
  $\leqslant^1$-connectedness and \\
  $\leqslant^1$-continuity are as easy as the proofs of
  $\leqslant_1$-connectedness and $\leqslant_1$-continuity.
\end{proof}

Now we show that the $<^1$-relation is closely related with the substitutions
$x \longmapsto x [\alpha \assign e]$. We first make the following

\begin{definition}
  Let $q \in \tmop{OR}$ with $q =_{\tmop{CNF}} L_1 q_1 + \ldots + L_n q_n$.
  Let \\
  $S_{\tmop{CNF}} (q) \assign \{L_1 q_1, \ldots, L_n q_n \} \cup \{\Sigma_{i =
  1}^j L_i q_i |j \in \{1, \ldots, n\}\} \cup \bigcup_{\{A_i |i \in [1, n]
  \wedge L_i \nin \mathbbm{E} \wedge L_i =_{\tmop{CNF}} \omega^{A_i} \}}
  S_{\tmop{CNF}} (A_i)$.
\end{definition}

\begin{proposition}
  \label{Isomorphisms_are_substitutions}Let $\alpha \in \mathbbm{E}$ be an
  arbitrary epsilon number.
  \begin{enumeratenumeric}
    \item Let $t \in [\alpha, \alpha^+)$ and $B (t) \assign S_{\tmop{CNF}} (t)
    \cup \{L j|L q \in S_{\tmop{CNF}} (t) \wedge L \in \mathbbm{P} \wedge q
    \in [1, \omega) \wedge j \in \{1, \ldots, q\}\}$. Note $t \in B (t)
    \subset_{\tmop{fin}} t + 1$.\\
    Then any $h : B (t) \longrightarrow h [B (t)] \subset \alpha$ that is an
    $(<, <_1, +, \lambda x. \omega^x)$ isomorphism with $h|_{\alpha} =
    \tmop{Id}_{\alpha}$ satisfies $h (\alpha) \in \mathbbm{E} \cap \alpha$ and
    $\forall s \in B (t) . \tmop{Ep} (s) \cap \alpha \subset h (\alpha) \wedge
    h (s) = s [\alpha \assign h (\alpha)]$.
    
    \item Let $t \in (\alpha, \alpha^+)$ and suppose $\alpha <^1 t$. Let $B
    \subset_{\tmop{fin}} t$. Then there exists $\gamma \in \mathbbm{E} \cap
    \alpha$ such that $\forall s \in B. \tmop{Ep} (s) \cap \alpha \subset
    \gamma$ and the function $h : B \longrightarrow h [B] \subset \alpha$, $s
    \longmapsto s [\alpha \assign \gamma]$ is an \\
    $(<, <_1, +, \lambda x. \omega^x)$ isomorphism with $h|_{\alpha} =
    \tmop{Id}_{\alpha}$.
  \end{enumeratenumeric}
\end{proposition}

\begin{proof}
  We prove 2. first.
  
  Suppose $t \in (\alpha, \alpha^+)$, $\alpha <^1 t$ and $B
  \subset_{\tmop{fin}} t$. Consider the set $C \assign \bigcup_{s \in (B \cup
  \{\alpha\}) \cap [\alpha, \alpha^+)} B (s) \subset_{\tmop{fin}} t$, where $B
  (s)$ is the set defined in 1. Now, since $\alpha <^1 t$, then there exists
  an $(<, <_1, +, \lambda x. \omega^x)$-isomorphism $H : C \longrightarrow H
  [C] \subset \alpha$ with $H|_{\alpha} = \tmop{Id}_{\alpha}$. Note that
  $\alpha \in B (\alpha) \subset C$ and therefore, by 1., \\
  $H (\alpha) \in \mathbbm{E} \cap \alpha$. Let $\gamma \assign H (\alpha)$.
  
  Let's show that $\forall s \in B. \tmop{Ep} (s) \cap \alpha \subset \gamma$.
  Let $s \in B$. If $s \in B \cap \alpha$, then $s = H (s) < H (\alpha) =
  \gamma$ because the relation $<$ is preserved by $H$ and so $\tmop{Ep} (s)
  \cap \alpha \subset H (\alpha) = \gamma$. If $s \in B \cap [\alpha, t)$,
  then $s \in B (s) \subset C$ and then, by 1., $\tmop{Ep} (s) \cap \alpha
  \subset H (\alpha) = \gamma$.
  
  Finally, to show that the function $h : B \longrightarrow h [B]$, $h (s)
  \assign s [\alpha \assign \gamma]$ is an ($<, <_1, +, \lambda x. \omega^x$)
  isomorphism with $h|_{\alpha} = \tmop{Id}_{\alpha}$ it is enough to show
  that $h = H|_B$ (since $H|_B : B \longrightarrow H|_B [B] \subset \alpha$ is
  already an ($<, <_1, +, \lambda x. \omega^x$) isomorphism with $(H|_B)
  |_{\alpha} = \tmop{Id}_{\alpha}$ by {\cite{GarciaCornejo0}} proposition
  \ref{iso.restriction} in the appendices). So let $s \in B$. If $s < \alpha$,
  then $s < \gamma = H (\alpha)$ and so $h (s) = s [\alpha \assign \gamma] = s
  = H|_B (s)$. If $s \geqslant \alpha$, then $s \in B (s)$ and then by 1. we
  have that $H|_B (s) = H (s) = s [\alpha \assign H (\alpha)] = s [\alpha
  \assign \gamma] = h (s)$.

  We prove 1.
  
  Let $t \in [\alpha, \alpha)$, $B (t)$ and $h$ as in our hypothesis. Then $h
  (\alpha) = h (\omega^{\alpha}) = \omega^{h (\alpha)}$. So $h (\alpha) \in
  \mathbbm{E}$. Moreover, from the definition of $B (t)$ and using that $h$
  preserves the $<$ relation, it follows that \\
  $\forall s \in B (t) . \tmop{Ep} (s) \cap \alpha \subset B (t)$ and $\forall
  l \in \tmop{Ep} (s) \cap \alpha .l = h (l) < h (\alpha)$; that is, \\
  $\forall s \in B (t) . \tmop{Ep} (s) \cap \alpha \subset h (\alpha)$.
  
  We now show $\forall s \in B (t) .h (s) = s [\alpha \assign h (\alpha)]$ by
  induction on the set $B (t)$ (with the usual order $<$ on the ordinals):
  
  Let $s \in B$ with $s =_{\tmop{CNF}} \omega^{A_1} a_1 + \ldots +
  \omega^{A_u} a_u$.
  
  Suppose $\forall y \in s \cap B (t) .h (y) = y [\alpha \assign h (\alpha)]$.
  \ \ \ \ \ \ \ {\tmstrong{(IH)}}.
  
  If $u \geqslant 2$, then by IH $h (\omega^{A_1} a_1) = \omega^{A_1} a_1
  [\alpha \assign h (\alpha)], \ldots$,$h (\omega^{A_u} a_u) = \omega^{A_1}
  a_u [\alpha \assign h (\alpha)]$ and therefore $h (s) = h (\omega^{A_1} a_1)
  + \ldots + h (\omega^{A_u} a_u) = \omega^{A_1} a_1 [\alpha \assign h
  (\alpha)] + \ldots + \omega^{A_u} a_u [\alpha \assign h (\alpha)] =$\\
  $(\omega^{A_1} a_1 + \ldots + \omega^{A_u} a_u) [\alpha \assign h (\alpha)]
  = s [\alpha \assign h (\alpha)]$.
  
  If $u = 1$ and $a_1 \geqslant 2$, then by IH $h (\omega^{A_1} (a_1 - 1)) =
  \omega^{A_1} (a_1 - 1) [\alpha \assign h (\alpha)]$ and\\
  $h (\omega^{A_1}) = \omega^{A_1} [\alpha \assign h (\alpha)]$. Then $h (s) =
  h (\omega^{A_1} (a_1 - 1)) + h (\omega^{A_1}) =$\\
  \ \ \ \ \ \ \ \ $= \omega^{A_1} (a_1 - 1) [\alpha \assign h (\alpha)] +
  \omega^{A_1} [\alpha \assign h (\alpha)] = \omega^{A_1} a_1 [\alpha \assign
  h (\alpha)] = s [\alpha \assign h (\alpha)]$.
  
  If $u = 1$ and $a_1 = 1$ (that is, $s =_{\tmop{CNF}} \omega^{A_1}$) we have
  two subcases:
  
  $\bullet$ $A_1 < s$. Then by IH $h (A_1) = A_1 [\alpha \assign h (\alpha)]$
  and so $h (s) = \omega^{h (A_1)} = \omega^{A_1 [\alpha \assign h (\alpha)]}
  =$\\
  \ \ \ \ \ \ \ $= \omega^{A_1} [\alpha \assign h (\alpha)] = s [\alpha
  \assign h (\alpha)]$.
  
  $\bullet$ $A_1 = s$. Then $s \in \mathbbm{E}$. If $s < \alpha$, then $h (s)
  = s = s [\alpha \assign h (\alpha)]$ because $h|_{\alpha} =
  \tmop{Id}_{\alpha}$. If $s \geqslant \alpha$, \ \ \ \ \ \ \ \ \ \ then $s =
  \alpha$ (because $\alpha \leqslant s < t < \alpha^+$). So $h (s) = h
  (\alpha) = \alpha [\alpha \assign h (\alpha)] = s [\alpha \assign h
  (\alpha)]$.
\end{proof}

\subsubsection{Cofinality properties of $<^1$.}

What are the ordinals $\alpha$ such that $\alpha <^1 \alpha + 1$?. Well, one
can prove the following:

\begin{exercise}
  \label{a<less>^1a+1} $\forall \alpha \in \tmop{OR} . \alpha <^1 \alpha + 1
  \Longleftrightarrow \alpha \in \tmop{Lim} \mathbbm{E}$.
\end{exercise}

\begin{exercise}
  \label{teo.<less>^1.reaches_alpha2} $\forall \alpha \in \tmop{OR} . \alpha
  \in \tmop{Lim} \mathbbm{E} \Longrightarrow \forall \xi \in (0, \alpha) .
  \alpha <^1 \alpha + \xi$.
\end{exercise}

The consideration of the previous exercises revels two very important
properties of the relation $<^1$ which we prove now:

\begin{proposition}
  \label{<less>^1.implies.cofinal.sequence}(First fundamental cofinality
  property of $<^1$). Let $\alpha \in \mathbbm{E}$ and suppose $\alpha <^1 s$
  for some $s \in (\alpha, \alpha^+)$. Then for any $t \in [\alpha, s)$ there
  exists a sequence $(c_{\xi})_{\xi \in X} \subset \alpha \cap \mathbbm{E}$
  such that $\tmop{Ep} (t) \cap \alpha \subset c_{\xi}$, $c_{\xi}
  \underset{\tmop{cof}}{\longhookrightarrow} \alpha$ and $c_{\xi} <_1 t
  [\alpha \assign c_{\xi}]$.
\end{proposition}

\begin{proof}
  Let $\alpha \in \mathbbm{E}$ and suppose $\alpha <^1 s$ for some $s \in
  (\alpha, \alpha^+)$. Let $t \in [\alpha, s)$. We define \\
  $M \assign \max (\tmop{Ep} (t) \cap \alpha)$. Let $\delta \in [M + 1,
  \alpha)$ be arbitrary.
  
  Consider the set \\
  $B_{\delta} \assign S_{\tmop{CNF}} (t) \cup \{L j|L q \in S_{\tmop{CNF}} (t)
  \wedge L \in \mathbbm{P} \wedge q \in [1, \omega) \wedge j \in \{1, \ldots,
  q\}\} \cup \{\delta\} \subset_{\tmop{fin}} t + 1 \leqslant s$. By
  hypothesis, there exists an $(<, <_1, +, \lambda x. \omega^x)$ isomorphism
  $h_{\delta} : B_{\delta} \longrightarrow h [B_{\delta}] \subset \alpha$ with
  \\
  $h_{\delta} |_{\alpha} = \tmop{Id}_{\alpha}$. Moreover, by proposition
  \ref{Isomorphisms_are_substitutions}, $h_{\delta} (\alpha) \in \mathbbm{E}$,
  $\tmop{Ep} (t) \cap \alpha \subset h_{\delta} (\alpha)$ and \\
  $\delta = h_{\delta} (\delta) < h_{\delta} (\alpha) <_1 h_{\delta} (t) = t
  [\alpha \assign h_{\delta} (\alpha)]$. Therefore, the set $P (t) \assign
  \{h_{\delta} (\alpha) | \delta \in [M + 1, \alpha)\} \subset \alpha$ is
  confinal in $\alpha$ and it satisfies $\forall c \in P (t) . \tmop{Ep} (t)
  \cap \alpha \subset c \wedge c <_1 t [\alpha \assign c]$. From this follows
  the claim of this proposition.
\end{proof}

The next result is a more general version of lemma 3.11 appearing in
{\cite{Wilken1}}. It's proof uses the main argument used in Wilken's proof.
There is, however, one difference that we want to stress (something that may
be overlooked by the reader): For a class of ordinals $\emptyset \neq X
\subset \tmop{OR}$, we have defined $\tmop{Lim} (X) \assign \{\alpha \in
\tmop{OR} | \sup (X \cap \alpha) = \alpha\}$; that is, in general, $\tmop{Lim}
(X) X$. This notion is very important in the whole of our work (and
particularly, in the next proposition).

\begin{proposition}
  \label{2nd_Fund_Cof_Property_<less>^1}(Second fundamental cofinality
  property of $<^1$) \\
  Let $\alpha \in \mathbbm{E}$ and $t \in [\alpha, \alpha^+)$. Assume $\alpha
  \in \tmop{Lim} \{\gamma \in \mathbbm{E} | \tmop{Ep} (t) \cap \alpha \subset
  \gamma \wedge \gamma \leqslant_1 t [\alpha \assign \gamma]\}$. \\
  Then $\forall s \in [\alpha, t + 1] . \alpha \leqslant^1 s$.
\end{proposition}

\begin{proof}
  Let $\alpha \in \mathbbm{E}$, $t \in [\alpha, \alpha^+)$ and assume $\alpha
  \in \tmop{Lim} \{\gamma \in \mathbbm{E} | \tmop{Ep} (t) \cap \alpha \subset
  \gamma \wedge \gamma \leqslant_1 t [\alpha \assign \gamma]\}$.
  
  We prove by induction: $\forall s \in [\alpha, t + 1] . \alpha \leqslant^1
  s$.
  
  For $s = \alpha$ it is clear the claim holds. So, from now on, suppose $s >
  \alpha$.
  
  Case $s \in \tmop{Lim} \cap [\alpha, t + 1]$. Our induction hypothesis is
  $\alpha \leqslant^1 \beta$ for all $\beta \in [\alpha, t + 1] \cap s$. Thus
  $\alpha \leqslant^1 s$ by $\leqslant^1$-continuity.
  
  Suppose $s = l + 1 \in [\alpha, t + 1]$. Our induction hypothesis is $\alpha
  \leqslant^1 l$. \ \ \ \ \ \ \ \ \ \ {\tmstrong{(IH)}}
  
  Let $B \subset_{\tmop{fin}} s = l + 1$. Without loss of generality, suppose
  $\alpha, l \in B$ and write $B = X \cup Y$ where $X \assign B \cap \alpha$,
  $Y \assign B \cap [\alpha, l]$, $Y \assign \{y_1, \ldots, y_m | \alpha = y_1
  < y_2 < \ldots < y_m = l\}$.

  Note $l \in [\alpha, t] \subset [\alpha, \alpha^+) \ni t$ implies that
  $\forall e \in \tmop{Ep} (l) \cup \tmop{Ep} (t) .e \leqslant \alpha$;
  moreover, since $\tmop{Ep} (l) \cup \tmop{Ep} (t)$ is finite and $\alpha \in
  \tmop{Lim} \{\gamma \in \mathbbm{E} | \tmop{Ep} (t) \cap \alpha \subset
  \gamma \wedge \gamma \leqslant_1 t [\alpha \assign \gamma]\}$, then actually
  \\
  $\alpha \in \tmop{Lim} \{\gamma \in \mathbbm{E} | (\tmop{Ep} (l) \cup
  \tmop{Ep} (t)) \cap \alpha \subset \gamma \wedge \gamma \leqslant_1 t
  [\alpha \assign \gamma]\}$. \ \ \ \ \ \ (*). \\
  But for any $\gamma \in \mathbbm{E}$ such that $(\tmop{Ep} (l) \cup
  \tmop{Ep} (t)) \cap \alpha \subset \gamma$ we have $\gamma \leqslant l
  [\alpha \assign \gamma] \leqslant t [\alpha \assign \gamma]$; therefore, by
  $\leqslant_1$-connectedness and (*) we conclude $\alpha \in \tmop{Lim}
  \{\gamma \in \mathbbm{E} | \tmop{Ep} (l) \cap \alpha \subset \gamma \wedge
  \gamma \leqslant_1 l [\alpha \assign \gamma]\}$.

  Let \ $p \assign \max \bigcup_{i \in \{1, \ldots, m\}} (\tmop{Ep} (y_i) \cap
  \alpha)$ and consider the set \\
  $M \assign \{\gamma \in \alpha \cap \mathbbm{E} |p < \gamma \supset X \wedge
  \gamma \leqslant_1 l [\alpha \assign \gamma]\}$. Let $C \assign \{m (a) |a
  \in (B \cap \alpha) \wedge m (a) < \alpha\}$. Since $C \subset_{\tmop{fin}}
  \alpha$ and by our previous observations $M$ is confinal in $\alpha$, then
  $(\max C, \alpha) \cap M \neq \emptyset$. \\
  Let $\gamma \assign \min (M \cap (\max C, \alpha)) \in M$. We define the
  function $h : B \longrightarrow h [B] \subset \alpha$ as \\
  $h (x) \assign x [\alpha \assign \gamma]$ for all $x \in B$.

  Let's see that $h$ is an $(<, <_1, +, \lambda x. \omega^x)$-isomorphism.
  
  That $h$ preserves is an $(<, +, \lambda x. \omega^x)$-isomorphism follows
  directly from the fact that \\
  $X \cup \bigcup_{i \in \{1, \ldots, m\}} (\tmop{Ep} (y_i) \cap \alpha)
  \subset \gamma$ and corollary \ref{[a:=e]_iso_B_contained_M(a,e)}.

  Let's see that $h$ also preserves $<_1$.
  \begin{itemizedot}
    \item First observe that by IH $\alpha \leqslant^1 l$ and so $\alpha
    \leqslant_1 l$; subsequently, by $\leqslant_1$-connectedness it follows
    $\alpha \leqslant_1 y_i$ for any $y_i \in Y$. So we need to show $h
    (\alpha) \leqslant_1 h (y_i)$ for any $y_i \in Y$. But this is easy
    because $h (\alpha) = \gamma \leqslant_1 l [\alpha \assign \gamma]$ by the
    way we took $\gamma$, and since \\
    $\forall y_i \in Y.h (\alpha) \leqslant h (y_i) \leqslant h (l) = l
    [\alpha \assign \gamma]$, then $y_i \in Y.h (\alpha) \leqslant_1 h (y_i)$
    by $\leqslant_1$-connectedness.
    
    \item Clearly $x_1 \leqslant_1 x_2 \Longleftrightarrow h (x_1) = x_1
    \leqslant_1 x_2 = h (x_2)$ for any $x_1, x_2 \in X$.
    
    \item For $x \in X$ and $y_i \in Y$, $x \leqslant_1 y_i \Longrightarrow h
    (x) = x \leqslant_1 h (y_i)$ by $\leqslant_1$-connectedness (because \\
    $x = h (x) \leqslant h (y_i) \leqslant y_i$ for any $i \in \{1, \ldots,
    m\}$).
    
    \item For $x \in X$, and $y_i \in Y$, $x \nleqslant_1 y_i \Longrightarrow
    x \nleqslant_1 \alpha$ (otherwise, using the fact that we know $\alpha
    \leqslant_1 y_i$ for all $i \in \{1, \ldots, m\}$, we would have $x
    \leqslant_1 y_i$ by $\leqslant_1$-transitivity). So $m (x) \in C$ and
    then\\
    $x < m (x) < \gamma \leqslant h (y_i)$; therefore $h (x) = x \nleqslant_1
    h (y_i)$.
    
    \item For $y_i, y_j \in Y \cap (\alpha, \alpha^+)$, $y_i \leqslant_1 y_j
    \Longleftrightarrow y_i \leqslant y_j \leqslant m (y_i)
    \underset{\text{\tmop{corollary}
    \ref{A_[alpha:=e]_isomorphisms}}}{\Longleftrightarrow}$\\
    $h (y_i) = y_i [\alpha \assign \gamma] \leqslant y_j [\alpha \assign
    \gamma] = h (y_j) \leqslant m (y_i) [\alpha \assign \gamma] = m (y_i
    [\alpha \assign \gamma]) = m (h (y_i)) \Longleftrightarrow$\\
    $h (y_i) \leqslant_1 h (y_j)$.
  \end{itemizedot}
  All the previous cases show that $h$ preserves $<_1$ too and from all our
  work we have that $h$ is indeed an $(<, =, <_1, +, \lambda x.
  \omega^x)$-isomorphism. This shows $\alpha \leqslant^1 l + 1$.
\end{proof}

\section{Covering theorem}

From the definition of $<^1$ it is very easy to see that $\alpha <^1 \beta
\Longrightarrow \alpha <_1 \beta$. But, what about the implication $\alpha <^1
\beta \Longleftarrow \alpha <_1 \beta$?. From exercise \ref{a<less>^1a+1} (or
by use of the first fundamental cofinality property of $<^1$) it follows that
this implication does not hold in general. The motivation for the whole of
this section is the study of such implication: The main result is lemma
\ref{CoveringLemma1} (covering lemma), which has two important corollaries:
The proof of the minimality of the substitutions as witnesses of $\alpha <_1
\beta$ for $\beta$ which are closed under the cover construction and the
solution to the question when $\alpha <^1 \beta \Longleftarrow \alpha <_1
\beta$.

We introduce the following definitions as a preparation for the covering
lemma.

\begin{definition}
  The following functions will be used in the main lemma of this section. For
  an ordinal $t =_{\tmop{CNF}} \omega^{T_1} t_1 + \ldots + \omega^{T_n} t_n$
  we define the ordinals\\
  $d q \assign \left\{ \begin{array}{l}
    0 \text{ iff } q \nin \mathbbm{P}\\
    Q_m \text{ iff } q = \omega^Q \text{ with } Q =_{\tmop{CNF}}
    \omega^{Q_1} q_1 + \ldots + \omega^{Q_m} q_m
  \end{array} \right.$,\\
  $\pi t \assign \omega^{T_1}$, and\\
  $\eta t \assign \max \{t, \pi t + d \pi t\}$.
\end{definition}

\begin{proposition}
  \label{pi.eta.properties}Let $\alpha, t, s \in \tmop{OR}$. Then
  \begin{enumeratenumeric}
    \item $\pi (t + 1) = \pi t$; moreover, if $t =_{\tmop{CNF}} \omega^{T_1}
    t_1 + \ldots + \omega^{T_n} t_n$, then $d \pi t \leqslant T_1 \leqslant
    \omega^{T_1} = \pi t$.
    
    \item Suppose $t \leqslant s$. Then $\pi t \leqslant \pi s$, $\pi t + d
    \pi t \leqslant \pi s + d \pi s$ and therefore $\eta t \leqslant \eta s$.
    
    \item If $t \geqslant \alpha \in \mathbbm{E}$ then $\alpha 2 \leqslant
    \eta t$
    
    \item $\pi (\pi t) = \pi t$, $\pi (\pi t + d \pi t) = \pi t$ and so $\eta
    (\eta t) = \eta t$
  \end{enumeratenumeric}
\end{proposition}

\begin{proof}
  {\color{orange} Left to the reader.}
\end{proof}

\begin{proposition}
  \label{eta(t)_m(t)_and_<less>_1}(Properties of $\eta t$ and $<_1$). Let
  $\alpha \in \mathbbm{E}$ and $t \in (\alpha, \alpha^+)$. Then
  \begin{enumeratenumeric}
    \item $t \nin \mathbbm{P} \Longrightarrow m (t) = t$; moreover, $t \in
    \mathbbm{P} \Longrightarrow m (t) = \pi t + d \pi t = \max \{t, \pi t + d
    \pi t\} = \eta t$. \\
    Particularly, $m (t) \leqslant \eta t$.
    
    \item $\forall u \in (\alpha, t] .m (u) \leqslant \eta t$. Therefore,
    $\forall s \in [\alpha, \alpha^+) . \eta s = \left\{ \begin{array}{l}
      \max \{m (u) |u \in (\alpha, s]\} \text{ iff } s > \alpha 2\\
      \alpha 2 \text{ iff } s \leqslant \alpha 2
    \end{array} \right.$.
    
    \item $\alpha <_1 t \Longleftrightarrow \alpha <_1 \eta t$
    
    \item If $l \in [\alpha, t]$, then $\eta l \leqslant \eta t$
    
    \item It indeed happens that $m (t) < \eta t$.
  \end{enumeratenumeric}
\end{proposition}

\begin{proof}
  {\color{orange} Left to the reader.}
\end{proof}

\begin{definition}
  \label{C(delta)_covering}For any $L \in \mathbbm{P}$, let
  
  {\noindent}$F (L) \assign \left\{ \begin{array}{l}
    \begin{array}{l}
      \{\omega^{\omega^{V_1} v_1 + \omega^{V_2} v_2 + \ldots + \omega^{V_g}
      \cdot j} |g \in [1, t], j \in [1, v_g]\} \cup\\
      \{\omega^{\omega^{V_1} v_1 + \omega^{V_2} v_2 + \ldots + \omega^{V_g}
      \cdot j} + V_g |g \in [1, t], j \in [1, v_g]\}
    \end{array} \text{\tmop{if}} \begin{array}{l}
      L = \omega^Z \nin \mathbbm{E} \wedge\\
      Z =_{\tmop{CNF}} \sum_{j = 1}^t \omega^{V_j} v_j
    \end{array}\\
    \\
    \\
    \{L\} \text{\tmop{if}} L \in \mathbbm{E} \cup \{1\}
  \end{array} \right.$

  Now, for any $\delta \in \tmop{OR}$ with $\delta =_{\tmop{CNF}} L_1 l_1 +
  \ldots + L_n l_n$, let

  {\noindent}$C_1 (\delta) \assign \bigcup_{L_i \nin \mathbbm{E}} F (L_i)$ and

  {\noindent}$C_2 (\delta) \assign \{L_i j|i \in [1, n], j \in [1, l_i]\}
  \cup \{\Sigma_{i = 1}^j L_i l_i |j \in [1, n]\}$.

  Finally, for any $\delta \in \tmop{OR}$ with $\delta =_{\tmop{CNF}} L_1 l_1
  + \ldots + L_n l_n$, we define (by recursion on \\
  $(\tmop{OR}, <)$) the set $C (\delta)$ as

  {\noindent}$C (\delta) \assign C_1 (\delta) \cup \bigcup_{\sigma \in C_1
  (\delta)} C_2 (\sigma) \cup C_2 (\delta) \cup \bigcup_{V \in Y (\delta)} C
  (V)$, where \\
  $Y (\delta) \assign \{V_{i j} | \exists L_i \nin \mathbbm{E} .L_i = \omega^Z
  \wedge Z =_{\tmop{CNF}} \sum_{j = 1}^{t (i)} \omega^{V_{i j}} v_{i j} \}$
  (observe $Y \subset \delta$).
\end{definition}

\begin{proposition}
  \label{C(ro)_contained_in_C(delta)}Let $\delta \in \tmop{OR}$. Then $\forall
  \rho \in C (\delta) .C (\rho) \subset C (\delta)$.
\end{proposition}

\begin{proof}
  By induction on the ordinals one shows $\forall \delta \in \tmop{OR} .
  \forall \rho \in C (\delta) .C (\rho) \subset C (\delta)$. It is necessary
  to check the ways how $\rho$ may be in $C (\delta)$. {\color{orange} The
  details are left to the reader.}
\end{proof}

We prove now the covering lemma.

\begin{lemma}
  \label{CoveringLemma1}(Cover for one ordinal). Let $\alpha \in \mathbbm{E}$
  and $\delta \in \alpha^+$ with $\delta =_{\tmop{CNF}} L_1 l_1 + \ldots + L_n
  l_n$. Let \\
  $D (\alpha, \delta) \assign C (\delta) \cup \{\alpha, \alpha 2\}$. Then
  \begin{enumerateroman}
    \item $C (\delta)$ is a finite set.
    
    \item $\bullet$ $\{\delta, L_1 l_1, \ldots, L_n l_n \} \subset C (\delta)
    \subset \max \{\delta + 1, L_1 + d (L_1) + 1\} = \max \{\delta, L_1 + d
    (L_1)\} + 1 = \eta \delta + 1$\\
    $\bullet$ If $\delta \geqslant \alpha$ then $\eta \delta \in D (\alpha,
    \delta) \subset \max \{\delta + 1, L_1 + d (L_1) + 1\} = \max \{\delta,
    L_1 + d (L_1)\} + 1 = \eta \delta + 1$
    
    \item Suppose $\delta \in [\alpha, \alpha^+)$ and $h : D (\alpha, \delta)
    \longrightarrow h [D (\alpha, \delta)]$ is an $(<, <_1, +)$-isomorphism
    such that $h|_{\alpha} = \tmop{Id}_{\alpha}$. Then $h (\alpha) \in
    \mathbbm{E}$ and $\forall x \in D (\alpha, \delta) . (\tmop{Ep} (x) \cap
    \alpha) \subset h (\alpha) \wedge x [\alpha \assign h (\alpha)] \leqslant
    h (x)$.
  \end{enumerateroman}
\end{lemma}

\begin{proof}

  {\noindent}$i$.\\
  By induction on $\delta$. Suppose $\forall r < \delta$. $C (r)$ is finite. \
  \ \ \ \ \ \ {\tmstrong{(IH1)}}\\
  $|C (\delta) | \leqslant |C_2 (\delta) | + |C_1 (\delta) | + | \bigcup_{R
  \in Y (\delta)} C (R) | + | \bigcup_{\sigma \in C_1 (\delta)} C_2 (\sigma) |
  \leqslant$
  
  \ \ \ \ \ $ |\{L_i j|i \in \{1, \ldots, n\}, j \in \{1, \ldots, l_i \}\}| +
  |\{\Sigma_i^j L_i l_i |j \in \{1, \ldots, n\}\}| +$\\
  \ \ \ \ \ \ \ \ \ \ $| \bigcup_{L_i \nin \mathbbm{E}} F (L_i) | + |
  \bigcup_{R \in Y (\delta)} C (R) | + | \bigcup_{\sigma \in C_1 (\delta)} C_2
  (\sigma) | \leqslant$\\
  \ \ \ \ \ \ \ \ $l_1 + l_2 + \ldots + l_n + \sum_{j = 1}^k j + \sum_{L_i
  \nin \mathbbm{E}} |F (L_i) | + \sum_{R \in Y (\delta)} |C (R) | + |
  \bigcup_{\sigma \in C_1 (\delta)} C_2 (\sigma) | < \omega$, where the last
  inequality holds because:
  
  (1). For any $V \in Y (\delta)$, $V < \delta$, and so $C (V)$ is finite by
  our (IH1); moreover, the set\\
  $Y (\delta) = \{V_{i j} | \exists L_i \nin \mathbbm{E} .L_i = \omega^Z
  \wedge Z =_{\tmop{CNF}} \sum_{j = 1}^{t (i)} \omega^{V_{i j}} v_{i j} \}$ is
  finite too, since there are only a finite number of $L_i$'s, and for each
  one of the $L_i \nin \mathbbm{E}$ with $L_i = \omega^Z$ and $Z
  =_{\tmop{CNF}} \sum_{j = 1}^{t (i)} \omega^{V_{i j}} v_{i j}$, there are
  only a finite number of $V_{i j}$. Thus $\sum_{R \in Y (\delta)} |C (R) | <
  \omega$.
  
  (2). For any $L_i \nin \mathbbm{E}$, it is easy to see that $F (L_i)$ is
  finite too; moreover, as we already said, there are only a finite number of
  $L_i$'s. So $|C_1 (\delta) | \leqslant \sum_{L_i \nin \mathbbm{E}} |F (L_i)
  | < \omega$.
  
  (3). $C_2 (\sigma)$ is finite for any $\sigma \in C_1 (\delta)$ (exactly by
  the same reason why $C_2 (\delta)$ is finite) and $C_1 (\delta)$ is finite
  too (as argued in previous subcase (2)); therefore $| \bigcup_{\sigma \in
  C_1 (\delta)} C_2 (\sigma) | < \omega$.

  {\noindent}$i i$.
  
  $\bullet$ We show that $\{\delta, L_1 l_1, \ldots, L_n l_n \} \subset C
  (\delta) \subset \max \{\delta + 1, L_1 + d (L_1) + 1\}$.

  Clearly $\{\delta, L_1 l_1, \ldots, L_n l_n \} \subset C (\delta)$.

  Let's prove by induction $\forall \delta .C (\delta) \subset \max \{\delta
  + 1, L_1 + d (L_1) + 1\}$.

  Suppose $\forall r < \delta .C (r) \subset \max \{r + 1, \pi (r) + d (\pi
  (r)) + 1\}$. \ \ \ \ \ \ \ {\tmstrong{(IH2)}}

  Clearly $\{L_i j|i \in \{1, \ldots, n\}, j \in \{1, \ldots, l_i \}\} \cup
  \{\Sigma_i^j L_i l_i |j \in \{1, \ldots, n\}\} \subset \delta + 1$. \ \ \ \
  \ \ \ {\tmstrong{(ii1*)}}

  Now, take $V \in Y (\delta)$. By definition it means there exist $i, j \in
  \omega$, where $L_i \nin \mathbbm{E}$ is an additive principal number in the
  Cantor normal form of $\delta$, $L_i = \omega^Z$, $Z =_{\tmop{CNF}} \sum_{j
  = 1}^{t (i)} \omega^{V_{i j}} v_{i j}$ and $V = V_{i j}$. Observe that $d
  (\pi (V)) \leqslant \pi (V) \leqslant V < L_i$ ($V < L_i \leqslant \delta$
  holds because equality would imply $L_i \in \mathbbm{E}$ and we know that is
  not the case), and since $L_i \in \mathbbm{P}$, then $\pi (V) + d (\pi (V))
  < L_i$. So both \\
  $\pi (V) + d (\pi (V)) + 1, V + 1 \leqslant L_i \leqslant L_1 \leqslant
  \delta < \delta + 1$. Since the previous holds for any $V \in Y$, then
  $\bigcup_{R \in Y (\delta)} C (V)\underset{\text{by our (IH2)}}{\subset}
  \bigcup_{R \in Y (\delta)} \max \{V + 1, \pi (V) + d (\pi (V)) + 1\}
  \subset \delta + 1$. \ \ \ \ \ \ \ {\tmstrong{(ii2*)}}

  We now check what happens with $C_1 (\delta) = \bigcup_{L_i \nin
  \mathbbm{E}} F (L_i)$. By definition, for any $L_i \nin \mathbbm{E}$ with \\
  $L_i = \omega^Z$ and $Z =_{\tmop{CNF}} \sum_{j = 1}^{t (i)} \omega^{V_{i j}}
  v_{i j}$ \\
  $F (L_i) = \{\omega^{\omega^{V_{i 1}} v_{i 1} + \omega^{V_{i 2}} v_{i 2} +
  \ldots + \omega^{V_{i g}} \cdot j} |g \in \{1, \ldots, t (i)\}, j \in \{1,
  \ldots, v_{i g} \}\} \cup$
  
  \ \ \ \ \ $ \{\omega^{\omega^{V_{i 1}} v_{i 1} + \omega^{V_{i 2}} v_{i 2} +
  \ldots + \omega^{V_{i g}} \cdot j} + V_{i g} |g \in \{1, \ldots, t (i)\}, j
  \in \{1, \ldots, v_{i g} \}\}$.

  Clearly $\{\omega^{\omega^{V_{i 1}} v_{i 1} + \omega^{V_{i 2}} v_{i 2} +
  \ldots + \omega^{V_{i g}} \cdot j} |g \in \{1, \ldots, t (i)\}, j \in \{1,
  \ldots, v_{i g} \}\} \subset$\\
  $L_i + 1 \leqslant L_1 + 1 \leqslant \delta + 1$. \ \ \ \ \ \ \
  {\tmstrong{(ii3*)}}
  
  On the other hand for any $g \in \{1, \ldots, t (i) - 1\}, j \in \{1,
  \ldots, v_{i g} \}$, \\
  $\omega^{\omega^{V_{i 1}} v_{i 1} + \omega^{V_{i 2}} v_{i 2} + \ldots +
  \omega^{V_{i g}} \cdot j} + V_{i g}$ $\leqslant \omega^{\omega^{V_{i 1}} v_{i
  1} + \omega^{V_{i 2}} v_{i 2} + \ldots + \omega^{V_{i g}} \cdot j} +
  \omega^{\omega^{V_{i 1}} v_{i 1} + \omega^{V_{i 2}} v_{i 2} + \ldots +
  \omega^{V_{i g}} \cdot j}=$ \\ $\omega^{\omega^{V_{i 1}} v_{i 1} + \omega^{V_{i
  2}} v_{i 2} + \ldots + \omega^{V_{i g}} \cdot j} 2$ $< \omega^{\omega^{V_{i
  1}} v_{i 1} + \omega^{V_{i 2}} v_{i 2} + \ldots + \omega^{V_{i g}} \cdot j}
  \omega$ $= \omega^{\omega^{V_{i 1}} v_{i 1} + \omega^{V_{i 2}} v_{i 2} +
  \ldots + \omega^{V_{i g}} \cdot j + 1}\leqslant$ \\ $\omega^{\omega^{V_{i 1}}
  v_{i 1} + \omega^{V_{i 2}} v_{i 2} + \ldots + \omega^{V_{i g}} v_{i g} + 1}$
  $\leqslant \omega^{\omega^{V_{i 1}} v_{i 1} + \omega^{V_{i 2}} v_{i 2} +
  \ldots + \omega^{V_{i (t (i) - 1)}} v_{i (t (i) - 1)} + 1} \leqslant$\\
  $\omega^{\omega^{V_{i 1}} v_{i 1} + \omega^{V_{i 2}} v_{i 2} + \ldots +
  \omega^{V_{i (t (i) - 1)}} v_{i (t (i) - 1)} + \omega^{V_{i t (i)}} v_{i t
  (i)}} = L_i \leqslant L_1 \leqslant \delta < \delta + 1$. \ \ \ \ \ \ \
  {\tmstrong{(ii4*)}}
  
  For the case $g = t (i), j \in \{1, \ldots, v_{i g} \}$,
  $\omega^{\omega^{V_{i 1}} v_{i 1} + \omega^{V_{i 2}} v_{i 2} + \ldots +
  \omega^{V_{i g}} \cdot j} + V_{i g} \leqslant$ \\ $L_i + V_{i g} = L_i + d (L_i)
  \leqslant$
  $\left\{ \begin{array}{l}
    L_i 2 < L_{i - 1} < L_1 + d (L_1) + 1 \text{ if } i \geqslant 2\\
    \\
    L_1 + d (L_1) < L_1 + d (L_1) + 1 \text{ if } i = 1\\
    
  \end{array} \right.$. \ \ \ \ \ \ \ {\tmstrong{(ii5*)}}

  So, by (ii3*), (ii4*) and (ii5*), we conclude $C_1 (\delta) \subset \max
  \{\delta + 1, L_1 + d (L_1) + 1\}$. \ \ \ \ \ \ \ {\tmstrong{(ii6*)}}

  We now show that $\bigcup_{\sigma \in C_1 (\delta)} C_2 (\sigma) \subset
  \max \{\delta + 1, L_1 + d (L_1) + 1\}$ too. By the same argument used in
  (ii1*), $\forall \beta \in \bigcup_{\sigma \in C_1 (\delta)} C_2 (\sigma) .
  \beta \leqslant \max \{\sigma | \sigma \in C_1 (\delta)\}\underset{\text{by
  (ii6*)}}{\leqslant}\max \{\delta, L_1 + d (L_1)\}$. \\
  Hence $\bigcup_{\sigma \in C_1 (\delta)} C_2 (\sigma) \subset \max \{\delta
  + 1, L_1 + d (L_1) + 1\}$. \ \ \ \ \ \ \ {\tmstrong{(ii7*)}}

  From (ii1*), (ii2*), (ii6*) and (ii7*) we conclude $C (\delta) \subset \max
  \{\delta + 1, L_1 + d (L_1) + 1\}$.

  $\bullet$ Suppose $\delta \geqslant \alpha$.
  
  Then $D (\alpha, \delta) \subset \max \{\delta + 1, L_1 + d (L_1) + 1\} =
  \max \{\delta, L_1 + d (L_1)\} + 1$ holds because $\alpha 2 \leqslant \max
  \{\delta, L_1 + d (L_1)\}$ by proposition \ref{pi.eta.properties}.
  
  Let's prove that $\eta \delta \in D (\alpha, \delta)$. If $\delta = \eta
  \delta$, then (we just proved that) $\eta \delta = \delta \in C (\delta)
  \subset D (\alpha, \delta)$. So suppose $\delta \neq \eta \delta = \max
  \{\delta, \pi \delta + d \pi \delta\}$. If $\delta \in [\alpha, \alpha 2)$,
  then $\eta \delta = \alpha 2 \in D (\alpha, \delta)$. Suppose $\delta >
  \alpha 2$. Consider $\delta =_{\tmop{CNF}} L_1 l_1 + \ldots + L_n l_n$. Note
  $l_1 = 1$, (otherwise $\pi d + d \pi d \leqslant L_1 + L_1 = L_1 2 \leqslant
  L_1 l_1 \leqslant \delta$ and then we would have that $\delta = \eta
  \delta$); moreover, $L_1 \nin \mathbbm{E}$ (otherwise $L_1 = \alpha$ and
  then $\delta < \alpha 2$). This way, $L_1 \in \mathbbm{P} \backslash
  \mathbbm{E}$ and $L_1 =_{\tmop{CNF}} \omega^Z$ for some $Z \in \tmop{OR}$,
  where $Z =_{\tmop{CNF}} \omega^{R_1} r_1 + \ldots + \omega^{R_u} r_u$ for
  some ordinals $R_i \in \tmop{OR}$ and $r_i \in [1, \omega)$. Therefore,
  $\eta \delta = \pi \delta + d \pi \delta = L_1 + R_u \in F (L_1) \subset C
  (\delta) \subset D (\alpha, \delta)$.

  {\noindent}$i i i$.\\
  Suppose $\delta \in [\alpha, \alpha^+)$ and $h : D (\alpha, \delta)
  \longrightarrow h [D (\alpha, \delta)]$ is an $(<, <_1, +)$-isomorphism with
  $h|_{\alpha} = \tmop{Id}_{\alpha}$. Notice from $\alpha <_1 \alpha 2$
  follows $h (\alpha) <_1 h (\alpha) 2$, which is equivalent to $h (\alpha)
  \in \mathbbm{E}$.
  
  We now prove the claim $\forall x \in D (\alpha, \delta) . \tmop{Ep} (x)
  \cap \alpha \subset h (\alpha) \wedge x [\alpha \assign h (\alpha)]
  \leqslant h (x)$ by induction on the well order $(D (\alpha, \delta), <)$.
  
  Let $x \in D (\alpha, \delta)$. Our induction hypothesis is
  
  $\forall y \in x \cap D (\alpha, \delta) . \tmop{Ep} (y) \cap \alpha
  \subset h (\alpha) \wedge y [\alpha \assign h (\alpha)] \leqslant h (y)$. \
  \ \ \ \ \ \ {\tmstrong{(IH)}}

  If $x =_{\tmop{CNF}} T_1 t_1 + \ldots + T_m t_m$, with $m \geqslant 2$, then
  $x \in C (\delta)$ and then by our (IH) and prop.
  \ref{C(ro)_contained_in_C(delta)} we have that $\tmop{Ep} (T_i) \cap \alpha
  \subset h (\alpha)$ and $h (T_i) \geqslant T_i [\alpha \assign h (\alpha)]$
  for all $i \in \{1, \ldots, m\}$; therefore \\
  $\tmop{Ep} (x) \cap \alpha \subset h (\alpha)$ and $h (x) = h (T_1) t_1 +
  \ldots + h (T_m) t_m \geqslant$\\
  $T_1 [\alpha \assign h (\alpha)] t_1 + \ldots + T_m [\alpha \assign h
  (\alpha)] t_m = x [\alpha \assign h (\alpha)]$.

  If $x =_{\tmop{CNF}} T_1 t_1$ with $t_1 \geqslant 2$ then, $x = \alpha 2$ or
  $x \in C (\delta)$. In any case, proceeding similarly as in the previous
  case, $\tmop{Ep} (x) \cap \alpha \subset h (\alpha)$ and $h (x) \geqslant x
  [\alpha \assign h (\alpha)]$.

  So suppose $x =_{\tmop{CNF}} T_1$.

  If $T_1 \in \mathbbm{E}$ then $T_1 = \alpha$ or $T_1 \in \alpha \cap
  \mathbbm{E}$ (because $x \leqslant \max \{\delta, L_1 + d (L_1)\} + 1 <
  \alpha^+$). If $T_1 = \alpha$, then $\tmop{Ep} (x) \cap \alpha = \emptyset
  \subset h (\alpha) \in \mathbbm{E}$ and $h (x) = h (\alpha) = x [\alpha
  \assign h (\alpha)]$. If $T_1 \in \alpha \cap \mathbbm{E}$, then $x = h (x)
  < h (\alpha)$, and so $\tmop{Ep} (x) \cap \alpha \subset h (\alpha)$;
  moreover $h (x) = x = x [\alpha \assign h (\alpha)]$.
  
  So suppose $T_1 \nin \mathbbm{E}$. Then $T_1 = \omega^Z$, with $Z
  =_{\tmop{CNF}} \omega^{R_1} r_1 + \ldots + \omega^{R_k} r_k$. Notice that
  since \\
  $\forall i \in \{1, \ldots, k\} .R_i \in D (\alpha, \delta) \wedge R_i
  \leqslant Z < T_1 = x$, then by (IH) $\bigcup_{1 \leqslant i \leqslant k}
  \tmop{Ep} (R_i) \cap \alpha \subset h (\alpha)$ and therefore $\tmop{Ep} (Z)
  \cap \alpha \subset h (\alpha)$. Thus $\tmop{Ep} (x) \cap \alpha \subset h
  (\alpha)$. So it only rest to show that the inequality holds. For the case
  $T_1 < \alpha$, we have $h (x) = x = x [\alpha \assign h (\alpha)]$. So the
  interesting case is $\alpha < T_1 = \omega^Z \nin \mathbbm{E}$.

  We have the sets of inequalities (I0) and (I1):

  {\noindent}$\omega^Z > R_1 > R_2 > \ldots > R_k$; \ \ \ \ \ \ \ \ \ \ \ \ \
  \ \ \ \ \ \ \ \ \ \ \ \ \ \ \ \ \ \ \ \ \ \ \ \ \ \ \ \ \ \ \ \ \ \ \ \ \ \
  \ \ (I0) \ {
  
  }$\omega^{R_1} r_1 > \omega^{R_1} (r_1 - 1) \ldots > \omega^{R_1} 3 >
  \omega^{R_1} 2 > \omega^{R_1} \geqslant R_1$ \ \ \ \ \ \ \ \ \ \ \ \ \ \ \ \
  \ \ \ \ \ \ \ \ \ \ \ \ \ \ \ \ \ \ \ \ \ \ \

  {\noindent}$\omega^{R_1} r_1 + \omega^{R_2} r_2 > \omega^{R_2} r_2 >
  \omega^{R_2} (r_2 - 1) \ldots > \omega^{R_2} 3 > \omega^{R_2} 2 >
  \omega^{R_2} \geqslant R_2$ \ \ \ \ \ \ \ \ \ \ \ \ \ \ \ \ \ \ \ \ \ \ \ \
  \ (I1) \ {
  
  }$\ldots$ {
  
  }$\omega^{R_1} r_1 + \omega^{R_2} r_2 + \ldots + \omega^{R_k} r_k >
  \omega^{R_k} r_k > \omega^{R_r} (r_k - 1) \ldots > \omega^{R_k} 3 >
  \omega^{R_k} 2 > \omega^{R_k} \geqslant R_k$

  On the other hand, from the inequalities
  
  {\noindent}$R_1 \leqslant \omega^{\omega^{R_1}} < \omega^{\omega^{R_1} 2} <
  \ldots < \omega^{\omega^{R_1} r_1}$\\
  $R_2 < \omega^{\omega^{R_1} r_1 + \omega^{R_2}} < \ldots <
  \omega^{\omega^{R_1} r_1 + \omega^{R_2} r_2}$\\
  $R_{k - 1} < \omega^{\omega^{R_1} r_1 + \omega^{R_2} r_2 + \ldots +
  \omega^{R_{k - 2}} r_{k - 2} + \omega^{R_{k - 1}}} < \ldots <
  \omega^{\omega^{R_1} r_1 + \omega^{R_2} r_2 + \ldots + \omega^{R_{k - 2}}
  r_{k - 2} + \omega^{R_{k - 1}} r_{k - 1}}$\\
  $R_k < \omega^{\omega^{R_1} r_1 + \omega^{R_2} r_2 + \ldots + \omega^{R_{k -
  1}} r_{k - 1} + \omega^{R_k}}$\\
  and theorem \ref{Wilken_Theorem1} we get the inequalities (I2):

  {\noindent}$\left\{ \begin{array}{l}
    R_1 = \alpha < \omega^{\omega^{R_1} 2} \text{ (because } \alpha < T_1
    = \omega^{\omega^{R_1} r_1 + \ldots + \omega^{R_k} r_k}) \text{ if }
    R_1 = \omega^{\omega^{R_1}}\\
    \\
    \omega^{\omega^{R_1}} <_1 \omega^{\omega^{R_1}} + R_1 \leqslant
    \omega^{\omega^{R_1}} + \omega^{\omega^{R_1}} < \omega^{\omega^{R_1}}
    \omega = \omega^{\omega^{R_1} + 1} \leqslant \omega^{\omega^{R_1} 2} 
    \text{ if } R_1 < \omega^{\omega^{R_1}}
  \end{array} \right.$\\
  $\omega^{\omega^{R_1} 2} <_1 \omega^{\omega^{R_1} 2} + R_1 \leqslant
  \omega^{\omega^{R_1} 2} + \omega^{\omega^{R_1} 2} < \omega^{\omega^{R_1} 2}
  \omega = \omega^{\omega^{R_1} 2 + 1} \leqslant${
  
  }$\omega^{\omega^{R_1} 3} <_1 \omega^{\omega^{R_1} 3} + R_1 \leqslant
  \omega^{\omega^{R_1} 3} + \omega^{\omega^{R_1} 3} < \omega^{\omega^{R_1} 3}
  \omega = \omega^{\omega^{R_1} 3 + 1} \leqslant${
  
  }$\omega^{\omega^{R_1} (r_1 - 1)} <_1 \omega^{\omega^{R_1} (r_1 - 1)} + R_1
  \leqslant \omega^{\omega^{R_1} (r_1 - 1)} + \omega^{\omega^{R_1} (r_1 - 1)}
  < \omega^{\omega^{R_1} (r_1 - 1)} \omega = \omega^{\omega^{R_1} (r_1 - 1) +
  1} \leqslant$ {
  
  }$\omega^{\omega^{R_1} r_1} <_1 \omega^{\omega^{R_1} r_1} + R_1 \leqslant
  \omega^{\omega^{R_1} r_1} + R_1 \leqslant \omega^{\omega^{R_1} r_1} +
  \omega^{\omega^{R_1} r_1} < \omega^{\omega^{R_1} r_1} \omega =
  \omega^{\omega^{R_1} r_1 + 1} \leqslant$\\
  $\omega^{\omega^{R_1} r_1 + \omega^{R_2}} <_1 \omega^{\omega^{R_1} r_1 +
  \omega^{R_2}} + R_2 \leqslant \omega^{\omega^{R_1} r_1 + \omega^{R_2} + 1}
  \leqslant${
  
  }$\omega^{\omega^{R_1} r_1 + \omega^{R_2} 2} <_1 \omega^{\omega^{R_1} r_1 +
  \omega^{R_2} 2} + R_2 \leqslant \omega^{\omega^{R_1} r_1 + \omega^{R_2} 2 +
  1} \leqslant \ldots \leqslant${
  
  }$\omega^{\omega^{R_1} r_1 + \omega^{R_2} r_2} <_1 \omega^{\omega^{R_1} r_1
  + \omega^{R_2} r_2} + R_2 \leqslant \omega^{\omega^{R_1} r_1 + \omega^{R_2}
  r_2 + 1} \leqslant$\\
  $\ldots$\\
  $\omega^{\omega^{R_1} r_1 + \omega^{R_2} r_2 + \ldots + \omega^{R_{k - 1}}}
  <_1 \omega^{\omega^{R_1} r_1 + \omega^{R_2} r_2 + \ldots + \omega^{R_{k -
  1}}} + R_{k - 1} \leqslant \omega^{\omega^{R_1} r_1 + \omega^{R_2} r_2 +
  \ldots + \omega^{R_{k - 1}} + 1} \leqslant$ {
  
  }$\omega^{\omega^{R_1} r_1 + \omega^{R_2} r_2 + \ldots + \omega^{R_{k - 1}}
  2} <_1 \omega^{\omega^{R_1} r_1 + \omega^{R_2} r_2 + \ldots + \omega^{R_{k -
  1}} 2} + R_{k - 1} \leqslant \omega^{\omega^{R_1} r_1 + \omega^{R_2} r_2 +
  \ldots + \omega^{R_{k - 1}} 2 + 1} \leqslant \ldots \leqslant$ {
  
  }$\omega^{\omega^{R_1} r_1 + \omega^{R_2} r_2 + \ldots + \omega^{R_{k - 1}}
  r_{k - 1}} <_1 \omega^{\omega^{R_1} r_1 + \omega^{R_2} r_2 + \ldots +
  \omega^{R_{k - 1}} r_{k - 1}} + R_{k - 1} \leqslant \omega^{\omega^{R_1} r_1
  + \omega^{R_2} r_2 + \ldots + \omega^{R_{k - 1}} r_{k - 1} + 1} \leqslant$\\
  $\left\{ \begin{array}{l}
    \begin{array}{l}
      \omega^{\omega^{R_1} r_1 + \omega^{R_2} r_2 + \ldots + \omega^{R_k}} <_1
      \omega^{\omega^{R_1} r_1 + \omega^{R_2} r_2 + \ldots + \omega^{R_k}} +
      R_k \leqslant \omega^{\omega^{R_1} r_1 + \omega^{R_2} r_2 + \ldots +
      \omega^{R_k} + 1} \leqslant\\
      \omega^{\omega^{R_1} r_1 + \omega^{R_2} r_2 + \ldots + \omega^{R_k} 2}
      <_1 \omega^{\omega^{R_1} r_1 + \omega^{R_2} r_2 + \ldots + \omega^{R_k}
      2} + R_k \leqslant \omega^{\omega^{R_1} r_1 + \omega^{R_2} r_2 + \ldots
      + \omega^{R_k} 2 + 1} \leqslant \ldots \leqslant\\
      \omega^{\omega^{R_1} r_1 + \omega^{R_2} r_2 + \ldots + \omega^{R_k} r_k}
      <_1 \omega^{\omega^{R_1} r_1 + \omega^{R_2} r_2 + \ldots + \omega^{R_k}
      r_k} + R_k = \omega^Z + d (\omega^Z)
    \end{array} \text{\hspace*{8mm} if } R_k \neq 0\\
    \\
    \begin{array}{l}
      \omega^{\omega^{R_1} r_1 + \omega^{R_2} r_2 + \ldots + \omega^{R_{k -
      1}} r_{k - 1}} + \omega^{\omega^{R_1} r_1 + \omega^{R_2} r_2 + \ldots +
      \omega^{R_k}} = \omega^{\omega^{R_1} r_1 + \omega^{R_2} r_2 + \ldots +
      \omega^{R_k}}\\
      \omega^{\omega^{R_1} r_1 + \omega^{R_2} r_2 + \ldots + \omega^{R_k}} +
      \omega^{\omega^{R_1} r_1 + \omega^{R_2} r_2 + \ldots + \omega^{R_k} 2} =
      \omega^{\omega^{R_1} r_1 + \omega^{R_2} r_2 + \ldots + \omega^{R_k} 2}\\
      \omega^{\omega^{R_1} r_1 + \omega^{R_2} r_2 + \ldots + \omega^{R_k} 2} +
      \omega^{\omega^{R_1} r_1 + \omega^{R_2} r_2 + \ldots + \omega^{R_k} 3} =
      \omega^{\omega^{R_1} r_1 + \omega^{R_2} r_2 + \ldots + \omega^{R_k} 3},
      \ldots\\
      \omega^{\omega^{R_1} r_1 + \omega^{R_2} r_2 + \ldots + \omega^{R_{k -
      1}} (r_k - 1)} + \omega^{\omega^{R_1} r_1 + \omega^{R_2} r_2 + \ldots +
      \omega^{R_k} r_k} = \omega^{\omega^{R_1} r_1 + \omega^{R_2} r_2 + \ldots
      + \omega^{R_k} r_k} = \omega^Z + d (\omega^Z)
    \end{array} \text{ if } R_k = 0
  \end{array} \right.$ \ \ \ \

  Therefore, from (I1) and (I2) we get the inequalities:\\
  $\forall i \in \{1, \ldots, k - 1\} \forall j \in \{1, \ldots, r_i \}$.
  
  $h (\omega^{\omega^{R_1} r_1 + \omega^{R_2} r_2 + \ldots + \omega^{R_i}
  \cdot j}) <_1 h (\omega^{\omega^{R_1} r_1 + \omega^{R_2} r_2 + \ldots +
  \omega^{R_i} \cdot j}) + h (R_i)$; \ \ \ \ \ \ \ \ \ \ \ \ \ \ \ \ (J1)
  
  Remark: In (J1), the case $i = 1$, $j = 1$ is $h (\omega^{\omega^{R_1}}) <_1
  h (\omega^{\omega^{R_1}}) + h (R_1)$ and it holds for two different reasons:
  If $R_1 = \omega^{\omega^{R_1}}$, then $R_1 = \alpha$ (because $\alpha < x
  \leqslant \max \{\delta, L_1 + d (L_1)\} + 1 < \alpha^+$) and so $h (\alpha)
  = h (\omega^{\omega^{R_1}}) = \omega^{\omega^{h (R_1)}} <_1
  \omega^{\omega^{h (R_1)}} + h (R_1) = h (\alpha) 2$ holds because we know
  $\alpha <_1 \alpha 2$. If \\
  $R_1 < \omega^{\omega^{R_1}}$, then $h (\omega^{\omega^{R_1}}) <_1 h
  (\omega^{\omega^{R_1}}) + h (R_1)$ holds because $\omega^{\omega^{R_1}} <_1
  \omega^{\omega^{R_1}} + R_1$.

  Moreover, from (I1) and (I2) we get the inequalities and equations (J3):\\
  For $1 \leqslant j \leqslant r_k$,
  
  $\left\{ \begin{array}{l}
    h (\omega^{\omega^{R_1} r_1 + \omega^{R_2} r_2 + \ldots + \omega^{R_k}
    \cdot j}) <_1 h (\omega^{\omega^{R_1} r_1 + \omega^{R_2} r_2 + \ldots +
    \omega^{R_k} \cdot j}) + h (R_k), \text{\hspace*{41mm} if } R_k \neq 0\\
    \text{\tmstrong{(\tmop{Observe} \tmop{here} \tmop{we} \tmop{use} :} } 
    \tmmathbf{\omega^Z + d (\omega^Z) \in \tmop{Dom} h)}\\
    \\
    \begin{array}{l}
      h (\omega^{\omega^{R_1} r_1 + \omega^{R_2} r_2 + \ldots + \omega^{R_{k -
      1}} r_{k - 1}}) + h (\omega^{\omega^{R_1} r_1 + \omega^{R_2} r_2 +
      \ldots + \omega^{R_k}}) = h (\omega^{\omega^{R_1} r_1 + \omega^{R_2} r_2
      + \ldots + \omega^{R_k}})\\
      h (\omega^{\omega^{R_1} r_1 + \omega^{R_2} r_2 + \ldots + \omega^{R_k}})
      + h (\omega^{\omega^{R_1} r_1 + \omega^{R_2} r_2 + \ldots + \omega^{R_k}
      2}) = h (\omega^{\omega^{R_1} r_1 + \omega^{R_2} r_2 + \ldots +
      \omega^{R_k} 2})\\
      h (\omega^{\omega^{R_1} r_1 + \omega^{R_2} r_2 + \ldots + \omega^{R_k}
      2}) + h (\omega^{\omega^{R_1} r_1 + \omega^{R_2} r_2 + \ldots +
      \omega^{R_k} 3}) = h (\omega^{\omega^{R_1} r_1 + \omega^{R_2} r_2 +
      \ldots + \omega^{R_k} 3}), \ldots\\
      h (\omega^{\omega^{R_1} r_1 + \omega^{R_2} r_2 + \ldots + \omega^{R_{k -
      1}} (r_k - 1)}) + h (\omega^{\omega^{R_1} r_1 + \omega^{R_2} r_2 +
      \ldots + \omega^{R_k} r_k}) = h (\omega^{\omega^{R_1} r_1 + \omega^{R_2}
      r_2 + \ldots + \omega^{R_k} r_k})
    \end{array} \text{ if } R_k = 0
  \end{array} \right.$

  Now, from (J1) and theorem \ref{Wilken_Theorem1} we get\\
  For $i \in \{1, \ldots, k - 1\}$, {
  
  }$h (\omega^{\omega^{R_1} r_1 + \ldots + \omega^{R_{i - 1}} r_{i - 1} +
  \omega^{R_i} \cdot j}) = \omega^{\omega^{h (R_i)} S (r_1, \ldots, r_{i - 1},
  j)}$ for some $S (r_1, \ldots, r_{i - 1}, j) \neq 0$.

  On the other hand, from $\omega^{\omega^{R_1}} < \omega^{\omega^{R_1} 2}
  \ldots < \omega^{\omega^{R_1} r_1}$ it follows \\
  $h (\omega^{\omega^{R_1}}) < h (\omega^{\omega^{R_1} 2}) < \ldots < h
  (\omega^{\omega^{R_1} r_1})$, which gives us, using the equality in the
  previous paragraph, $\omega^{\omega^{h (R_1)} S (1)} < \omega^{\omega^{h
  (R_1)} S (2)} < \ldots < \omega^{\omega^{h (R_1)} S (r_1 - 1)} <
  \omega^{\omega^{h (R_1)} S (r_1)}$. This implies\\
  $S (1) < S (2) < \ldots < S (r_1)$, which subsequently implies $r_1
  \leqslant S (r_1)$. Now, notice the following inductive argument: For any $j
  \in \{1, \ldots, r_2 \}$, $h (\omega^{\omega^{R_1} r_1}) = \omega^{\omega^{h
  (R_1)} S (r_1)} < h (\omega^{\omega^{R_1} r_1 + \omega^{R_2} j}) =
  \omega^{\omega^{h (R_2)} S (r_1, j)}$; this way, $\omega^{h (R_1)} S (r_1) <
  \omega^{h (R_2)} S (r_1, j)$ and $h (R_1) > h (R_2)$ (by (I0)). Thus $S
  (r_1, j) > \omega^{- h (R_2) + h (R_1)} S (r_1)$ (otherwise $\omega^{h
  (R_1)} S (r_1) \geqslant \omega^{h (R_2)} S (r_1, j)$) and therefore it
  exists $q (r_1, j) \in \tmop{OR}$ such that $S (r_1, j) = \omega^{- h (R_2)
  + h (R_1)} S (r_1) + q (r_1, j)$. Then $h (\omega^{\omega^{R_1} r_1 +
  \omega^{R_2} j}) = \omega^{\omega^{h (R_2)} S (r_1, j)} = \omega^{\omega^{h
  (R_2)} \omega^{- h (R_2) + h (R_1)} S (r_1) + q (r_1, j)} =
  \omega^{\omega^{h (R_1)} S (r_1) + \omega^{h (R_2)} q (r_1, j)}$. Moreover,
  observe the chain of inequalities $h (\omega^{\omega^{R_1} r_1 +
  \omega^{R_2}}) < h (\omega^{\omega^{R_1} r_1 + \omega^{R_2} 2}) < \ldots < h
  (\omega^{\omega^{R_1} r_1 + \omega^{R_2} r_2})$ implies $q (r_1, r_2)
  \geqslant r_2$. But for any $j \in \{1, \ldots, r_3 \}$, $h
  (\omega^{\omega^{R_1} r_1 + \omega^{R_2} r_2}) < h (\omega^{\omega^{R_1} r_1
  + \omega^{R_2} r_2 + \omega^{R_3} j}) = \omega^{\omega^{h (R_3)} S (r_1,
  r_2, j)}$, and so $\omega^{h (R_1)} S (r_1) + \omega^{h (R_2)} q (r_1, r_2)
  < \omega^{h (R_3)} S (r_1, r_2, j)$; since $h (R_3) < h (R_2) < h (R_1)$,
  then $S (r_1, r_2, j) > \omega^{- h (R_3) + h (R_1)} S (r_1) + \omega^{- h
  (R_3) + h (R_2)} q (r_1, r_2)$ and so it exists $q (r_1, r_2, j) \in
  \tmop{OR}$ such that \\
  $S (r_1, r_2, j) = \omega^{- h (R_3) + h (R_1)} S (r_1) + \omega^{- h (R_3)
  + h (R_2)} q (r_1, r_2) + q (r_1, r_2, j)$. Then $h (\omega^{\omega^{R_1}
  r_1 + \omega^{R_2} r_2 + \omega^{R_3} j}) = \omega^{\omega^{h (R_3)} S (r_1,
  r_2, j)} = \omega^{\omega^{h (R_1)} S (r_1) + \omega^{h (R_2)} q (r_1, r_2)
  + \omega^{h (R_3)} q (r_1, r_2, j)}$. Moreover, the chain of inequalities $h
  (\omega^{\omega^{R_1} r_1 + \omega^{R_2} r_2 + \omega^{R_3}}) < h
  (\omega^{\omega^{R_1} r_1 + \omega^{R_2} r_2 + \omega^{R_3} 2}) < \ldots < h
  (\omega^{\omega^{R_1} r_1 + \omega^{R_2} r_2 + \omega^{R_3} r_3})$ implies
  $q (r_1, r_2, r_3) \geqslant r_3$. Inductively, we obtain \\
  $h (\omega^{\omega^{R_1} r_1 + \ldots + \omega^{R_{k - 1}} r_{k - 1}}) =
  \omega^{\omega^{h (R_1)} S (r_1) + \omega^{h (R_2)} q (r_1, r_2) + \ldots +
  \omega^{h (R_{k - 1})} q (r_1, \ldots, r_{k - 1})}$ with \\
  $S (r_1) \geqslant r_1$, $q (r_1, r_2) \geqslant r_2, q (r_1, r_2, r_3)
  \geqslant r_3, \ldots, q (r_1, \ldots, r_{k - 1}) \geqslant r_{k - 1}$.

  For the case $R_k \neq 0$, doing once more the previous procedure with the
  equalities $h (\omega^{\omega^{R_1} r_1 + \ldots + \omega^{R_k} \cdot j}) =
  \omega^{\omega^{h (R_k)} S (r_1, \ldots, r_{k - 1}, j)}$, we obtain:\\
  $h (\omega^{\omega^{R_1} r_1 + \ldots + \omega^{R_k} r_k}) =
  \omega^{\omega^{h (R_1)} S (r_1) + \omega^{h (R_2)} q (r_1, r_2) + \ldots +
  \omega^{h (R_k)} q (r_1, \ldots, r_k)}$ with \\
  $S (r_1) \geqslant r_1$, $q (r_1, r_2) \geqslant r_2, q (r_1, r_2, r_3)
  \geqslant r_3, \ldots, q (r_1, \ldots, r_k) \geqslant r_k$; \ therefore \\
  $h (\omega^{\omega^{R_1} r_1 + \ldots + \omega^{R_k} r_k}) \geqslant
  \omega^{\omega^{h (R_1)} r_1 + \omega^{h (R_2)} r_2 + \ldots + \omega^{h
  (R_k)} r_k}$. \ \ \ \ \ \ \ {\tmstrong{(**1**)}}

  For the case $R_k = 0$ the additions in (J3) imply:\\
  $h (\omega^{\omega^{R_1} r_1 + \omega^{R_2} r_2 + \ldots + \omega^{R_k}})
  \geqslant \omega^{\omega^{h (R_1)} S (r_1) + \omega^{h (R_2)} q (r_1, r_2) +
  \ldots + \omega^{h (R_{k - 1})} q (r_1, \ldots, r_{k - 1}) + 1}$\\
  $h (\omega^{\omega^{R_1} r_1 + \omega^{R_2} r_2 + \ldots + \omega^{R_k} 2})
  \geqslant \omega^{\omega^{h (R_1)} S (r_1) + \omega^{h (R_2)} q (r_1, r_2) +
  \ldots + \omega^{h (R_{k - 1})} q (r_1, \ldots, r_{k - 1}) + 2}$\\
  $\ldots$\\
  $h (\omega^{\omega^{R_1} r_1 + \omega^{R_2} r_2 + \ldots + \omega^{R_k}
  r_k}) \geqslant \omega^{\omega^{h (R_1)} S (r_1) + \omega^{h (R_2)} q (r_1,
  r_2) + \ldots + \omega^{h (R_{k - 1})} q (r_1, \ldots, r_{k - 1}) + r_k}$\\
  and therefore $h (\omega^{\omega^{R_1} r_1 + \ldots + \omega^{R_k} r_k})
  \geqslant \omega^{\omega^{h (R_1)} r_1 + \omega^{h (R_2)} r_2 + \ldots +
  \omega^{h (R_k)} r_k}$ \ \ \ \ \ \ \ {\tmstrong{(**2**)}}

  Finally, to conclude, in any case $R_k \neq 0$ or $R_k = 0$, from (**1**)
  and (**2**) we have \\
  $h (x) = (\omega^Z) = h (\omega^{\omega^{R_1} r_1 + \ldots + \omega^{R_k}
  r_k}) \geqslant \omega^{\omega^{h (R_1)} r_1 + \omega^{h (R_2)} r_2 + \ldots
  + \omega^{h (R_k)} r_k} \underset{\tmop{IH}}{\geqslant}$\\
  $\omega^{\omega^{R_1 [\alpha \assign h (\alpha)]} r_1 + \omega^{R_2 [\alpha
  \assign h (\alpha)]} r_2 + \ldots + \omega^{R_k [\alpha \assign h (\alpha)]}
  r_k} = \omega^{Z [\alpha \assign h (\alpha)]} = \omega^Z [\alpha \assign h
  (\alpha)] = x [\alpha \assign h (\alpha)]$.

  This finishes the proof of this lemma.
\end{proof}

\subsection{Cover of a finite set $B$.}

Now we extend the construction of the covering for a finite set.

\begin{definition}
  (Cover of a finite set). Let $\alpha \in \mathbbm{E}$ and $B
  \subset_{\tmop{fin}} \alpha^+$. We define \\
  $\Delta (\alpha, B) \assign B \cup\underset{\delta \in B \cap [\alpha,
  \alpha^+)}{\bigcup}D (\alpha, \delta)$, where $D (\alpha, \delta)$ is
  the set defined in previous lemma \ref{CoveringLemma1}.
\end{definition}

\begin{proposition}
  \label{Delta_contained_in_eta(t)}Let $\alpha \in \mathbbm{E}$ and $B
  \subset_{\tmop{fin}} \alpha^+$. If $B \cap [\alpha, \alpha^+) = \emptyset$,
  then $\Delta (\alpha, B) = B \subset \alpha$.\\
  If $B \cap [\alpha, \alpha^+) \neq \emptyset$, then for $t \assign \max B$,
  $\Delta (\alpha, B) \subset_{\tmop{fin}} \eta t + 1 \subset \alpha^+$. In
  any case, $\Delta (\alpha, B)$ is finite.
\end{proposition}

\begin{proof}
  That $B \cap [\alpha, \alpha^+) = \emptyset$ implies $\Delta (\alpha, B) = B
  \subset \alpha$ is clear.
  
  Suppose $B \cap [\alpha, \alpha^+) \neq \emptyset$ and let $t \assign \max
  B$. Let $\delta \in B \cap [\alpha, \alpha^+)$ be arbitrary. If $\delta =
  \alpha$, then $D (\alpha, \delta) = \{\alpha, \alpha 2\}
  \underset{\text{\tmop{by} \tmop{proposition} \ref{pi.eta.properties},
  \tmop{claim} 3.}}{\subset} \eta t + 1$. If $\delta > \alpha$, then $\delta
  \leqslant t$ and so \\
  $D (\alpha, \delta) \underset{\text{\tmop{by} \tmop{prop} .
  \ref{pi.eta.properties} \tmop{claim} 3, \tmop{and} \tmop{by} \tmop{prop} .
  \ref{C(ro)_contained_in_C(delta)}}}{\subset} \eta \delta + 1
  \underset{\text{\tmop{by} \tmop{proposition} \ref{eta(t)_m(t)_and_<less>_1}
  \tmop{claim} 4}}{\leqslant} \eta t + 1 < \alpha^+$.
  
  Finally, $\Delta (\alpha, B)$ is finite because it is finite union of finite
  sets.
\end{proof}

\begin{theorem}
  \label{theorem_cover_simple_substitutions}(Covering theorem). Let $\alpha
  \in \mathbbm{E}$ and $B \subset_{\tmop{fin}} \alpha^+$ be such that $B \cap
  [\alpha, \alpha^+) \neq \emptyset$. Consider $F B \assign \{h : \Delta
  (\alpha, B) \longrightarrow h [\Delta (\alpha, B)] \subset \alpha |h
  \text{\tmop{is} \tmop{an}} ( <, <_1, +) \text{- \tmop{isomorphism}
  \tmop{with}} h|_{\alpha} = \tmop{Id}_{\alpha} \}$. Then for any $h \in F B$
  the ordinal $h (\alpha) \in \alpha \cap \mathbbm{E}$ and
  \begin{enumeratealpha}
    \item $\forall x \in \Delta (\alpha, B) . \tmop{Ep} (x) \cap \alpha
    \subset h (\alpha) \wedge x [\alpha \assign h (\alpha)] \leqslant h (x)$.
    
    \item If $\alpha \leqslant_1 \max \Delta (\alpha, B)$, then the function
    $H : \Delta (\alpha, B) \longrightarrow H [\Delta (\alpha, B)]$, $H (x) :
    = x [\alpha : = h (\alpha)]$ is an $(<, <_1, +, \lambda x.
    \omega^x)$-isomorphism with $H|_{\alpha} = \tmop{Id}_{\alpha}$.
  \end{enumeratealpha}
\end{theorem}

\begin{proof}
  Let $\alpha$ and $B$ as stated. Let $h \in F B$.
  
  First note that $\alpha, \alpha 2 \in \Delta (\alpha, B)$ (because $B \cap
  [\alpha, \alpha^+) \neq \emptyset$) and since $\alpha <_1 \alpha 2$, then $h
  (\alpha) <_1 h (\alpha) 2$. This implies that $h (\alpha) \in \mathbbm{E}$.

  Now we show $a)$.
  
  Let $x \in \Delta (\alpha, B)$.
  
  If $x < \alpha$, then $x = h (x) < h (\alpha)$ because $h$ is an
  $<$-isomorphism such that $h|_{\alpha} = \tmop{id}_{\alpha}$. Therefore
  $\tmop{Ep} (x) \cap \alpha \subset h (\alpha)$ and $x = x [\alpha \assign h
  (\alpha)]$.
  
  If $x = \alpha$, then clearly $\tmop{Ep} (x) \cap \alpha \subset h (\alpha)$
  and $x [\alpha \assign h (\alpha)] = h (x)$.
  
  Case $x > \alpha$. Then $x \in D (\alpha, x) \subset \Delta (\alpha, B)$
  and $h|_{D (\alpha, x)} : D (\alpha, x) \longrightarrow h [D (\alpha, x)]
  \subset \alpha$ is an \\
  $(<, <_1, +)$-isomorphism with $h |_{D (\alpha, x)} |_{\alpha} =
  \tmop{Id}_{\alpha}$ by {\cite{GarciaCornejo0}} proposition
  \ref{iso.restriction} in the appendices section. Therefore, by lemma
  \ref{CoveringLemma1}, $\tmop{Ep} (x) \cap \alpha \subset h|_{D (\alpha, x)}
  (\alpha) = h (\alpha)$ and $x [\alpha \assign h (\alpha)] \leqslant h|_{D
  (\alpha, x)} (x) = h (x)$.
  
  The previous shows $a)$.

  We show $b)$.
  
  Suppose $\alpha \leqslant_1 \max \Delta (\alpha, B)$.
  
  By $a)$ we know $\forall x \in \Delta (\alpha, B) . \tmop{Ep} (x) \cap
  \alpha \subset h (\alpha)$; so, by corollary
  \ref{[a:=e]_iso_B_contained_M(a,e)}, the function $H$ is an \ \\
  $(<, +, \cdot, \lambda x. \omega^x)$- isomorphism. Moreover, it is also
  clear that $H|_{\alpha} = \tmop{Id}_{\alpha}$. So we just need to prove that
  $H$ preserves the relation $<_1$ too. Let $\Delta (\alpha, B) \cap \alpha =
  \{a_1, \ldots, a_N \}$ and \\
  $\Delta (\alpha, B) \cap [\alpha, \alpha^+) = \{\alpha = b_1, \ldots, b_M
  \}$. Then:

  $\bullet$ Note $\alpha <_1 \max \Delta (\alpha, B)$ and
  $\leqslant_1$-connectedness imply that $\alpha <_1 b_j$ for any $b_j \neq
  \alpha$. So we need to show $H (\alpha) <_1 H (b_j)$ for any $b_j \neq
  \alpha$. But by $a)$ we know $H (b_j) = b_j [\alpha \assign h (\alpha)]
  \leqslant h (b_j)$; moreover, we know $h (\alpha) = H (\alpha) < H (b_j)$
  and $h (\alpha) <_1 h (b_j)$ for any $b_j \neq \alpha$. Thus by
  $\leqslant_1$-connectedness, $H (\alpha) <_1 H (b_j)$ for any $b_j \neq
  \alpha$.
  
  $\bullet$ $a_i <_1 a_j \Longleftrightarrow a_i = H (a_i) <_1 H (a_j) = a_j$.
  
  $\bullet$ $a_i <_1 \alpha \Longleftrightarrow H (a_i) = a_i = h (a_i) <_1 H
  (\alpha) = h (\alpha)$ because $h$ is an $<_1$isomorphism.
  
  $\bullet$ If $a_i <_1 b_j$, then $H (a_i) = a_i <_1 H (b_j) < \alpha < b_j$
  by $<_1$-connectedness.
  
  $\bullet$ If $H (a_i) <_1 H (b_j)$, then $H (a_i) <_1 H (\alpha)$ by
  $<_1$-connectedness. But \\
  $h (a_i) = H (a_i) <_1 H (\alpha) = h (\alpha) \Longleftrightarrow a_i <_1
  \alpha$ (because $h$ is an $<_1$isomorphism) and since $\alpha \leqslant_1
  b_j$, then $a_i <_1 b_j$ follows by $<_1$-transitivity.
  
  $\bullet$ For $b_i \neq \alpha \neq b_j$, $b_i <_1 b_j
  \underset{\text{\tmop{corollary} \ref{A_[alpha:=e]_isomorphisms}}
  }{\Longleftrightarrow} H (b_i) <_1 H (b_j)$.
  
  The previous shows $b)$.
\end{proof}

\subsubsection{Consequences of the covering theorem.}

Consider a finite set of ordinals $L \subset_{\tmop{fin}} \tmop{OR}$ and $F L
\subset \{k|k : L \longrightarrow \tmop{OR}\}$ a class of functionals. Then $F
L$ is well ordered under the lexicographic order $<_{F L, \tmop{lex}}$; that
is, for $h, k \in F L$, $h <_{\tmop{FL}, \tmop{lex}} k : \Longleftrightarrow$
$\exists y \in L.h (y) \neq k (y)$ and for $m \assign \min \{x \in L|h (x)
\neq k (x)\}$ it holds $h (m) < k (m)$. Moreover, in case $F L \neq
\emptyset$, we can consider $\min (F L)$, the minimum element in $F L$ with
respect to $<_{F L, \tmop{lex}}$. The next corollary uses this concepts.

\begin{corollary}
  \label{minimal_cover-iso_Class(1)}Let $\alpha \in \mathbbm{E}$ and $\beta
  \in (\alpha, \alpha^+)$ be with $\alpha <_1 \beta$. Suppose $B
  \subset_{\tmop{fin}} \beta$ is such that \\
  $\Delta (\alpha, B) \subset \beta$. Consider \\
  $F B \assign \{h : \Delta (\alpha, B) \longrightarrow h [\Delta (\alpha, B)]
  \subset \alpha |h \text{\tmop{is} \tmop{an}} ( <, <_1, +) \text{-
  \tmop{isomorphism} \tmop{with}} h|_{\alpha} = \tmop{Id}_{\alpha} \}$. Then
  \\
  $\mu \assign \min (F B)$ exists, $\mu (\alpha) \in \alpha \cap \mathbbm{E}$
  and $\mu$ is the substitution $x \longmapsto x [\alpha \assign \mu
  (\alpha)]$.
\end{corollary}

\begin{proof}
  Since $\alpha <_1 \beta$ and $\Delta (\alpha, B) \subset_{\tmop{fin}}
  \beta$, then $F B \neq \emptyset$ and so $\mu \assign \min F B$ exists. Now,
  by previous theorem \ref{theorem_cover_simple_substitutions}, $\mu (\alpha)
  \in \mathbbm{E} \cap \alpha$ and the function $H : \Delta (\alpha, B)
  \longrightarrow H [\Delta (\alpha, B)]$, $H (x) : = x [\alpha : = \mu
  (\alpha)]$ is well defined and satisfies the following two things: $H \in F
  B$ and $\forall x \in \Delta (\alpha, B) .H (x) \leqslant \mu (x)$. Thus,
  from the minimality of the function $\mu$, it follows $H = \mu$.
\end{proof}

The following is the main result that relates $<_1$ with $<^1$.

\begin{corollary}
  \label{<less>_1.iff.<less>^1}Let $\alpha \in \mathbbm{E}$ and $t \in
  [\alpha, \alpha^+)$. Then $\alpha <_1 \eta t + 1 \Longleftrightarrow \alpha
  <^1 \eta t + 1$.
\end{corollary}

\begin{proof}
  The implication $\Longleftarrow)$ is already known.
  
  {\noindent}Let's show $\Longrightarrow)$.
  
  Let $B \subset_{\tmop{fin}} \eta t + 1$. If $B \subset \alpha$, then $I : B
  \longrightarrow B$, $I (x) \assign x$ is an $(<, <_1, +, \lambda x.
  \omega^x)$- isomorphism with $I|_{\alpha} = \tmop{Id}_{\alpha}$. So suppose
  $B \cap [\alpha, \alpha^+) \neq \emptyset$. Let $l \assign \max B \geqslant
  \alpha$. Proposition \ref{Delta_contained_in_eta(t)} guarantees that $\Delta
  (\alpha, B) \subset_{\tmop{fin}} \eta l + 1$; but $\eta l\underset{\text{prop. }
  \ref{eta(t)_m(t)_and_<less>_1}}{\leqslant}\eta \eta t\underset{\text{prop. }
  \ref{pi.eta.properties}}{=}\eta t$, so $\Delta (\alpha, B)
  \subset_{\tmop{fin}} \eta l + 1 \leqslant \eta t + 1$. Moreover, since by
  hypothesis $\alpha <_1 \eta t + 1$, then there exists $h : \Delta (\alpha,
  B) \longrightarrow h [\Delta (\alpha, B)] \subset \alpha$ an $(<, <_1,
  +)$-isomorphism with $h|_{\alpha} = \tmop{Id}_{\alpha}$. Therefore, by
  theorem \ref{theorem_cover_simple_substitutions}, the function $H : \Delta
  (\alpha, B) \longrightarrow H [\Delta (\alpha, B)] \subset \alpha$ defined
  as $H (x) \assign x [\alpha \assign h (\alpha)]$ is an $(<, <_1, +, \lambda
  x. \omega^x)$-isomorphism with $H|_{\alpha} = \tmop{Id}_{\alpha}$. Then, by
  {\cite{GarciaCornejo0}} proposition \ref{iso.restriction} in the appendices
  section, $H|_B : B \longrightarrow H|_B [B]$ is an $(<, <_1, +, \lambda x.
  \omega^x)$-isomorphism with $H |_B |_{\alpha} = \tmop{Id}_{\alpha}$.
\end{proof}

\begin{corollary}
  \label{alpha<less>_1alpha^+.iff.alpha<less>^1alpha^+}$\forall \alpha \in
  \mathbbm{E} . \alpha <_1 \alpha^+ \Longleftrightarrow \alpha <^1 \alpha^+$
\end{corollary}

\begin{proof}
  {\color{orange} Easy. Left to the reader.}
\end{proof}

\begin{corollary}
  $\forall \alpha \in \mathbbm{E} . \alpha <_1 \alpha^+ \Longleftrightarrow$\\
  $\alpha \in \{\beta \in \mathbbm{E} | \forall t \in [\beta, \beta^+) \exists
  (c_{\xi})_{\xi \in X} \subset \mathbbm{E} \cap \beta . \tmop{Ep} (t) \cap
  \beta \subset c_{\xi} \wedge c_{\xi} <_1 t [\beta \assign c_{\xi}] \wedge
  c_{\xi} \underset{\tmop{cof}}{\longhookrightarrow} \beta\}$.
\end{corollary}

\begin{proof}
  {\color{orange} Not hard. Left to the reader.}
\end{proof}

We want to conclude this section with a characterization of the case $\alpha
<^1 t + 1$ for ordinals $\alpha \in \mathbbm{E}$ and $t \in [\alpha,
\alpha^+)$. For this (and also for our work on the next section), it will be
convenient to prove following

\begin{proposition}
  \label{pi.eta.substitutions}Let $\alpha, \beta, t \in \tmop{OR}$ such that
  $\alpha, \beta \in \mathbbm{E}$ and $t \in [\alpha, \alpha^+) \wedge
  \tmop{Ep} (t) \cap \alpha \subset \beta$. Then
  \begin{enumeratealpha}
    \item $(\tmop{Ep} (\pi t) \cap \alpha) \cup (\tmop{Ep} (d \pi t) \cap
    \alpha) \cup (\tmop{Ep} (\eta t) \cap \alpha) \subset \beta$
    
    \item $\pi (t [\alpha \assign \beta]) = (\pi t) [\alpha \assign \beta]$
    
    \item $d \pi (t [\alpha \assign \beta]) = (d \pi t) [\alpha \assign
    \beta]$
    
    \item $\pi (t [\alpha \assign \beta]) + d \pi (t [\alpha \assign \beta]) =
    (\pi t + d \pi t) [\alpha \assign \beta]$
    
    \item $\eta (t [\alpha \assign \beta]) = (\eta t) [\alpha \assign \beta]$.
  \end{enumeratealpha}
\end{proposition}

\begin{proof}
  {\color{orange} Not hard. Left to the reader.}
\end{proof}

\begin{note}
  Because of the previous proposition, whenever we have such hypothesis, we
  will simply write $\pi t [\alpha \assign \beta]$, $d \pi t [\alpha \assign
  \beta]$ and $\eta t [\alpha \assign \beta]$ to the ordinals $\pi (t [\alpha
  \assign \beta])$, $d \pi (t [\alpha \assign \beta])$ and $\eta (t [\alpha
  \assign \beta])$ respectively.
\end{note}

\begin{corollary}
  \label{eta(t)+1<less>_1_Equivalences}Let $\alpha \in \mathbbm{E}$ and $t \in
  [\alpha, \alpha^+)$. The following are equivalent
  \begin{enumeratealpha}
    \item $\alpha <^1 t + 1$
    
    \item $\alpha \in \tmop{Lim} \{\xi \in \mathbbm{E} | \tmop{Ep} (t) \cap
    \alpha \subset \xi \wedge \xi \leqslant_1 t [\alpha \assign \xi]\}$
    
    \item $\alpha <^1 \eta t + 1$
    
    \item $\alpha <_1 \eta t + 1$
  \end{enumeratealpha}
\end{corollary}

\begin{proof}
  Let $\alpha \in \mathbbm{E}$, $t \in [\alpha, \alpha^+)$.
  
  $a) \Longleftrightarrow b)$ holds because of propositions
  \ref{2nd_Fund_Cof_Property_<less>^1} and
  \ref{<less>^1.implies.cofinal.sequence}.
  
  $c) \Longleftrightarrow d)$ is corollary \ref{<less>_1.iff.<less>^1}.
  
  $c) \Longrightarrow a)$ holds because of $<^1$-connectedness, and so $c)
  \Longrightarrow b)$ (because $a) \Longleftrightarrow b$)).
  
  So it suffices to prove $b) \Longrightarrow c)$.

  Suppose $\alpha \in \tmop{Lim} \{\xi \in \mathbbm{E} | \tmop{Ep} (t) \cap
  \alpha \subset \xi \wedge \xi \leqslant_1 t [\alpha \assign \xi]\}$. Let
  $(c_j)_{j \in J} \subset \alpha \cap \mathbbm{E}$ be a sequence such that
  $\tmop{Ep} (t) \cap \alpha \subset c_j
  \underset{\tmop{cof}}{\longhookrightarrow} \alpha$ and $\forall j \in J.c_j
  \leqslant_1 t [\alpha \assign c_j]$. Note $\forall j \in J.t [\alpha \assign
  c_j] \in (c_j, c_j^+)$, and then, by proposition
  \ref{eta(t)_m(t)_and_<less>_1} claim 4., $\forall j \in J.c_j \leqslant_1
  \eta (t [\alpha \assign c_j]) = (\eta t) [\alpha \assign c_j])$, where the
  last equality holds because, by proposition \ref{pi.eta.substitutions},
  $\tmop{Ep} (\eta t) \cap \alpha \subset c_j$ and $\eta (t [\alpha \assign
  c_j]) = (\eta t) [\alpha \assign c_j])$ for any $j \in J$. So, summarizing,
  $(c_j)_{j \in J} \subset \alpha \cap \mathbbm{E}$ is a sequence of epsilon
  numbers such that for $\eta t \in [\alpha, \alpha^+)$, $\tmop{Ep} (\eta t)
  \cap \alpha \subset c_j \underset{\tmop{cof}}{\longhookrightarrow} \alpha$
  and $\forall j \in J. \mathbbm{E} \ni c_j \leqslant_1 (\eta t) [\alpha
  \assign c_j]$; therefore, by proposition
  \ref{2nd_Fund_Cof_Property_<less>^1}, $\alpha <^1 \eta t + 1$.
\end{proof}

\begin{corollary}
  \label{a_limit{e|m(e)=eta(t)[a:=e]}}Let $\alpha \in \mathbbm{E}$.
  \begin{enumeratealpha}
    \item $\forall e \in \alpha \cap \mathbbm{E} .m (e) \in [e, e^+)
    \Longrightarrow \exists t \in [\alpha, \alpha^+) . \eta t = t \wedge m (e)
    = t [\alpha \assign e]$.
    
    \item Suppose $m (\alpha) \in [\alpha, \alpha^+)$. Then\\
    $\forall t \in [\alpha, m (\alpha)) .t = \eta t \Longrightarrow \{\delta
    \in \mathbbm{E} \cap \alpha | \tmop{Ep} (t) \cap \alpha \subset \delta
    \wedge m (\delta) = t [\alpha \assign \delta]\}$ is confinal in $\alpha$.
  \end{enumeratealpha}
\end{corollary}

\begin{proof}
  Let $\alpha \in \mathbbm{E}$.
  
  {\noindent}$a)$. Let $e \in \alpha \cap \mathbbm{E}$ and suppose $m (e) \in
  [e, e^+)$. Then $\eta (m (e)) \ngtr m (e)$ (otherwise by proposition
  \ref{eta(t)_m(t)_and_<less>_1} claim 4, $e <_1 \eta (m (e)) \geqslant m (e)
  + 1$ which is impossible); but by definition $\eta (m (e)) =$\\
  $\max \{m (e), \pi (m (e)) + d \pi (m (e))\} \geqslant m (e)$, thus $\eta (m
  (e)) = m (e)$. This way, for $t \assign m (e) [e \assign \alpha]$, \\
  $\eta t = \eta (m (e) [e \assign \alpha])\underset{\text{proposition }
  \ref{pi.eta.substitutions}}{=}\eta (m (e)) [e \assign \alpha] = m (e) [e
  \assign \alpha] = t$ and clearly $m (e) = t [\alpha \assign e]$.

  {\noindent}$b)$. Suppose $m (\alpha) \in [\alpha, \alpha ( +^1))$ and let $t
  \in [\alpha, m (\alpha))$ be such that $t = \eta t$.
  
  Take $\gamma \in \alpha$ arbitrary.
  
  First note that $\alpha < \eta t + 1 \leqslant m (\alpha)$, implies, by
  $\leqslant_1$-connectedness, $\alpha <_1 \eta t + 1$. Subsequently, by
  previous corollary \ref{eta(t)+1<less>_1_Equivalences}, $\alpha \in
  \tmop{Lim} \{\xi \in \mathbbm{E} | \tmop{Ep} (t) \cap \alpha \subset \xi
  \wedge \xi \leqslant_1 t [\alpha \assign \xi]\} =$\\
  $\tmop{Lim} \{\xi \in \mathbbm{E} | \tmop{Ep} (\eta t) \cap \alpha \subset
  \xi \wedge \xi \leqslant_1 (\eta t) [\alpha \assign \xi]\}$. \\
  Let $e \assign \min (\gamma, \alpha] \cap \{\xi \in \mathbbm{E} | \tmop{Ep}
  (\eta t) \cap \alpha \subset \xi \wedge \xi \leqslant_1 (\eta t) [\alpha
  \assign \xi]\}$. Then \\
  $\gamma < e \leqslant_1 (\eta t) [\alpha \assign e]\underset{\text{proposition }
  \ref{pi.eta.substitutions}}{=}\eta t [\alpha \assign e]$. We assure $e
  \nless_1 \eta t [\alpha \assign e] + 1$. Suppose the opposite $e <_1 \eta t
  [\alpha \assign e] + 1$. Then by previous corollary
  \ref{eta(t)+1<less>_1_Equivalences}, \\
  $e \in \tmop{Lim} \{\xi \in \mathbbm{E} | \tmop{Ep} (t [\alpha \assign e])
  \cap e \subset \xi \wedge \xi \leqslant_1 t [\alpha \assign e] [e \assign
  \xi]\}\underset{\text{proposition } \ref{[alpha:=e]_proposition3}}{=}
  \tmop{Lim} \{\xi \in \mathbbm{E} | \tmop{Ep} (t) \cap \alpha \subset \xi
  \wedge \xi \leqslant_1 t [\alpha \assign \xi]\} =
  \tmop{Lim} \{\xi \in \mathbbm{E} | \tmop{Ep} (\eta t) \cap \alpha \subset
  \xi \wedge \xi \leqslant_1 (\eta t) [\alpha \assign \xi]\}$. \\
  The latter implies that there exist some ordinal $\varphi < e$ with \\
  $\varphi \in (\gamma, \alpha] \cap \{\xi \in \mathbbm{E} | \tmop{Ep} (\eta
  t) \cap \alpha \subset \xi \wedge \xi \leqslant_1 (\eta t) [\alpha \assign
  \xi]\}$. But this is impossible since by definition $e = \min (\gamma,
  \alpha] \cap \{\xi \in \mathbbm{E} | \tmop{Ep} (\eta t) \cap \alpha \subset
  \xi \wedge \xi \leqslant_1 (\eta t) [\alpha \assign \xi]\}$. Contradiction.
  \\
  Thus $e \nless_1 \eta t [\alpha \assign e] + 1$. Thus $m (e) = \eta t
  [\alpha \assign e]$.
  
  Our previous work has provided, given an arbitrary ordinal $\gamma \in
  \alpha$, an ordinal $e \in \mathbbm{E}$ such that $\gamma < e \in
  \mathbbm{E} \wedge \tmop{Ep} (t) \cap \alpha \subset e \wedge m (e) = \eta t
  [\alpha \assign e] = t [\alpha \assign e]$. Hence, we have shown that \\
  $\{\delta \in \mathbbm{E} \cap \alpha | \tmop{Ep} (t) \cap \alpha \subset
  \delta \wedge m (\delta) = t [\alpha \assign \delta]\}$ is confinal in
  $\alpha$.
\end{proof}

\section{A hierarchy induced by $<_1$ and the intervals
$[\varepsilon_{\gamma}, \varepsilon_{\gamma + 1})$}

In this section we will provide our theorem linking ``the solutions of the
$<_1$-inequality $x <_1 t$ with $t \in [x, x^+)$'' with a hierarchy of
ordinals obtained by a thinning procedure.

For the main theorem, we will need the following

\begin{lemma}
  \label{l_j-sequence}Let $\alpha, t \in \tmop{OR}$, $\alpha \in \mathbbm{E}$
  and $t \in (\alpha, \alpha^+) \cap \tmop{Lim}$. Then there exists a sequence
  $(l_j)_{j \in I}$ with $(I \cup \{0\}) \in \tmop{OR}$, $(I \cup \{0\})
  \leqslant \alpha$ such that\\
  (1) For all $j \in I$, $l_j \in (\alpha, \alpha^+)$, $l_j
  \underset{\tmop{cof}}{\longhookrightarrow} t$ and $(l_j)_{j \in I}$ is
  strictly monotonous increasing.\\
  (2) For any $\beta \in \mathbbm{E} \cap \alpha \cup \{\alpha\}$ with
  $\tmop{Ep} (t) \cap \alpha \subset \beta$, $\forall j \in I \cap \beta .
  \tmop{Ep} (l_j) \cap \alpha \subset \beta$; moreover, the \ \ \ \ \ \ \ \
  sequence $(l_j [\alpha \assign \beta])_{j \in I \cap \beta}$ is cofinal in
  $t [\alpha \assign \beta]$.\\
  (3) $\forall \beta \in ( \mathbbm{E} \cap \alpha \cup \{\alpha\}) \forall j
  \in I \cap \beta$.
  
  \ \ \ \ $\bullet$ $\eta l_j [\alpha \assign \beta] \leqslant \eta t [\alpha
  \assign \beta]$
  
  \ \ \ \ $\bullet$ $\eta l_j [\alpha \assign \beta] < \eta t [\alpha
  \assign \beta]$ if $t > \pi t + d \pi t$.
\end{lemma}

\begin{proof}
  Let $\alpha$ and $t$ be as stated. Below we give only the sequence. The
  proof that such sequence satisfies what is stated, is long and boring and
  {\color{orange} it is left to the reader.}

  Consider $t =_{\tmop{CNF}} \omega^{T_1} t_1 + \ldots + \omega^{T_n} t_n$ and
  $T_1 =_{\tmop{CNF}} \omega^{Q_1} q_1 + \ldots + \omega^{Q_m} q_m$. Suppose
  that for any $a \in (\alpha, \alpha^+) \cap \tmop{Lim} \cap t$ we have been
  able to define a sequence $(l'_j)_{i \in I'}$ satisfying what the theorem
  state with respect to $a$.

  Then we have cases:

  {\noindent}So $t = \omega^{T_1} t_1$ and $T_1 =_{\tmop{CNF}} \omega^{Q_1}
  q_1 + \ldots + \omega^{Q_m} q_m$.
  
  {\noindent} \ \ If $t_1 = 1$ then $t = \omega^{T_1}$.
  
  {\noindent} \ \ \ \ \ \ If $Q_m = 0$, then $m \geqslant 2$ and $Q_1
  \geqslant \alpha$ (otherwise $t < \alpha$) and $t = \omega^{T_1} =
  \omega^{\omega^{Q_1} q_1 + \ldots + \omega^0 q_m}$.
  
  {\noindent} \ \ \ \ \ \ \ \ \ \ {\color{blue} Let $l_j \assign
  \omega^{\omega^{Q_1} q_1 + \ldots + \omega^0 (q_m - 1)} j$}, with $j \in I
  \assign \omega \backslash \{0\}$.
  
  {\noindent} \ \ \ \ If $Q_m \neq 0$. Then $\omega^{Q_m} \leqslant
  \omega^{T_1} = t$. Moreover, we assure $\omega^{Q_m} < \omega^{T_1} = t$.
  This is because $\omega^{Q_m} = \omega^{T_1}$ implies $T_1
  \leqslant \omega^{T_1} = \omega^{Q_m} \leqslant T_1$ and then $\alpha = T_1$
  (since $T_1 \in [\alpha, \alpha^+)$); moreover, \ \ \ \ \ \ \ \ \ \ \ since
  $T_1 =_{\tmop{CNF}} \omega^{Q_1} q_1 + \ldots + \omega^{Q_m} q_m$, then $m =
  1$, $q_1 = 1$ and $Q_1 = \alpha$. That is, we have \ \ \ \ \ \ \ \ \ \ \ $t
  = \omega^{\alpha} = \alpha$ which is contradictory because from the
  beginning we picked $t \in (\alpha, \alpha^+)$. The
  previous showed $\omega^{Q_m} < \omega^{T_1}$.
  
  {\noindent} \ \ \ \ \ \ \ \ \ \ If $Q_m < \alpha$, then $m \geqslant 2$
  (otherwise $t = \omega^{T_1} = \omega^{\omega^{Q_1} q_1} < \alpha$).
  
  {\noindent} \ \ \ \ \ \ \ \ \ \ \ \ \ \ If $Q_m$ is a successor ordinal.
  
  {\noindent} \ \ \ \ \ \ \ \ \ \ \ \ \ \ \ \ \ \ {\color{blue} Let $l_j
  \assign \omega^{\omega^{Q_1} q_1 + \ldots + \omega^{Q_{(m - 1)}} q_{(m - 1)}
  + \omega^{Q_m} (q_m - 1) + \omega^{Q_m - 1} j}$} with $j \in I \assign
  \omega \backslash \{0\}$.
  
  {\noindent} \ \ \ \ \ \ \ \ \ \ \ \ \ \ If $Q_m$ is a limit ordinal.
  
  {\noindent} \ \ \ \ \ \ \ \ \ \ \ \ \ \ \ \ \ \ \ {\color{blue} Let $l_j
  \assign \omega^{\omega^{Q_1} q_1 + \ldots + \omega^{Q_{(m - 1)}} q_{(m - 1)}
  + \omega^{Q_m} (q_m - 1) + \omega^j}$} with $j \in I \assign Q_m \backslash
  \{0\}$.
  
  {\noindent} \ \ \ \ \ \ \ \ \ \ If $Q_m = \alpha$. Since \ $T_1
  =_{\tmop{CNF}} \omega^{Q_1} q_1 + \ldots + \omega^{Q_m} q_m \in [\alpha,
  \alpha^+)$, then $m = 1$ and so {
  
  } \ \ \ \ \ \ \ \ \ \ \ \ \ \ $t = \omega^{T_1} = \omega^{\omega^{Q_1} q_1}
  = \omega^{\omega^{\alpha} q_1}$.
  
  {\noindent} \ \ \ \ \ \ \ \ \ \ \ \ \ \ {\color{blue} Let $l_j \assign
  \omega^{\omega^{\alpha} (q_1 - 1) + \omega^j}$} with $j \in I \assign Q_1
  \backslash \{0\} = \alpha \backslash \{0\}$.
  
  {\noindent} \ \ \ \ \ \ \ If $Q_m > \alpha$. Then $\omega^{Q_m} \in (\alpha,
  \alpha^+) \cap \tmop{Lim}$ and moreover, we already know $\omega^{Q_m} <
  \omega^{T_1} = t$. \ \ \ \ \ \ \ \ \ \ \ \ Then by our induction hypothesis
  applied to $\omega^{Q_m}$ there exists a sequence $(\xi_j)_{j \in I}$ 
  with $I \cup \{0\} \in \tmop{OR}$ and $I \cup \{0\}
  \leqslant \alpha$, such that (1), (2), and (3) hold with respect to 
  the sequence $(\xi_j)_{j \in I}$ and $\omega^{Q_m}$.\\
  \hspace*{13mm} {\color{blue} Let $l_j \assign \omega^{\omega^{Q_1}
  q_1 + \ldots + \omega^{Q_{(m - 1)}} q_{(m - 1)} + \omega^{Q_m} (q_m - 1) +
  \xi_j}$} with $j \in I$.
  
  {\noindent} \ \ Case $t_1 \geqslant 2$.
  
  {\noindent} \ \ \ If $T_1 = \alpha$, then $t_1 = \omega^{\alpha} t_1$.
  
  {\noindent} \ \ \ \ \ \ {\color{blue} Let $l_j \assign \omega^{T_1} (t_1 -
  1) + j$} with $j \in I \assign T_1 \backslash \{0\} = \alpha \backslash
  \{0\}$.
  
  {\noindent} \ \ \ If $T_1 > \alpha$, then $t = \omega^{T_1} t_1 >
  \omega^{T_1} \in (\alpha, \alpha^+)$; so by our induction hypothesis applied
  to $\omega^{T_1}$ \ \ \ \ \ \ \ \ \ \ there exists a sequence $(\xi_j)_{j
  \in I}$ with $I \cup \{0\} \in \tmop{OR}$ and $I \cup \{0\} \leqslant
  \alpha$, such that (1), (2) and (3) hold with respect to
  the sequence $(\xi_j)_{j \in I}$ and $\omega^{T_1}$.
  
  {\noindent} \ \ \ \ \ \ {\color{blue} Let $l_j \assign \omega^{T_1} (t_1 -
  1) + \xi_j$} with $j \in I$.
  
  {\noindent}Case $n \geqslant 2$.
  
  {\noindent}So $t = \omega^{T_1} t_1 + \omega^{T_2} t_2 + \ldots +
  \omega^{T_n} t_n$, $T_1 =_{\tmop{CNF}} \omega^{Q_1} q_1 + \ldots +
  \omega^{Q_m} q_m$ and $T_1 > T_n \neq 0$ (because $t \in \tmop{Lim}$). Then
  $\omega^{T_n} < t$.
  
  {\noindent} \ If $T_n < \alpha$.
  
  {\noindent} \ \ \ \ {\color{blue} Let $l_j \assign \omega^{T_1} t_1 +
  \omega^{T_2} t_2 + \ldots + \omega^{T_{(n - 1)}} t_{(n - 1)} + \omega^{T_n}
  (t_n - 1) + j$} with $j \in I \assign \omega^{T_n} \backslash \{0\}$. So
  clearly \ \ \ \ $I \cup \{0\} < \alpha$.
  
  {\noindent} \ If $T_n = \alpha$. Then the argument is almost the same as in
  the previous subcase:
  
  {\noindent} \ \ \ \ {\color{blue} Let $l_j \assign \omega^{T_1} t_1 +
  \omega^{T_2} t_2 + \ldots + \omega^{T_n} (t_n - 1) + j$} with $j \in I
  \assign T_n \backslash \{0\} = \alpha \backslash \{0\}$.
  
  {\noindent} \ If $T_n > \alpha$. Then $\omega^{T_n} \in (\alpha, \alpha^+)
  \cap \tmop{Lim}$ and moreover, we already know $\omega^{T_n} < \omega^{T_1}
  \leqslant t$. Then by our induction hypothesis
  applied to $\omega^{T_1}$ there exists a sequence $(\xi_j)_{j \in I}$ with 
  $I \cup \{0\} \in \tmop{OR}$ and $I \cup \{0\}
  \leqslant \alpha$, such that (1), (2), and (3) hold with respect to the 
  sequence $(\xi_j)_{j \in I}$ and $\omega^{T_n}$.
  
  {\noindent} \ \ \ \ {\color{blue} Let $l_j \assign \omega^{T_1} t_1 + \ldots
  + \omega^{T_n} (t_n - 1) + \xi_j$} with $j \in I$.
\end{proof}

The following will be also needed in the main theorem of this section (theorem
\ref{teo.A(t)=G(t)}).

\begin{proposition}
  \label{alpha<less>^1alpha2+1}Let $\beta \in \tmop{OR}$.$\beta <^1 \beta 2 +
  1 \Longleftrightarrow \beta \in \tmop{Lim} \mathbbm{E}$
\end{proposition}

\begin{proof}
  {\color{orange} Not hard. Left to the reader.}
\end{proof}

\begin{definition}
  Let $A : [\varepsilon_{\omega}, \infty) \longrightarrow \tmop{Subclases}
  (\tmop{OR})$ be defined recursively as:

  {\noindent}For $l + 1 \in [\varepsilon_{\omega}, \infty)$,
  
  $A (l + 1) \assign \left\{ \begin{array}{l}
    A (l) \text{ if } l < \pi l + d \pi l\\
    \\
    \tmop{Lim} A (l) \text{ otherwise}
  \end{array} \right.$

  {\noindent}For $t \in [\varepsilon_{\omega}, \infty) \cap \tmop{Lim}$,

  $A (t) \assign \left\{ \begin{array}{l}
    (\tmop{Lim} \mathbbm{E}) \cap (M, \alpha + 1) \text{ iff } t \in
    [\alpha, \alpha 2]\\
    \\
    \tmop{Lim} \{r \leqslant \alpha |M < r \in \bigcap_{j \in I \cap r} A
    (l_j)\} \text{ iff } t > \pi t + d \pi t \wedge t \in (\alpha 2,
    \alpha^+)\\
    \\
    \tmop{Lim} \{r \leqslant \alpha |M < r \in \bigcap_{j \in S \cap r} A
    (e_j)\} \text{ iff } t \leqslant \pi t + d \pi t \wedge t \in (\alpha
    2, \alpha^+)
  \end{array} \right.$,

  {\noindent}where $\alpha \in \mathbbm{E}$ is such that $t \in [\alpha,
  \alpha^+)$; $(l_j)_{j \in I}$ is obtained by lemma \ref{l_j-sequence}
  applied to $t$ and $\alpha$; $(e_j)_{j \in S}$ is obtained by lemma
  \ref{l_j-sequence} applied to $\pi t$ and $\alpha$; and $M \assign \left\{
  \begin{array}{l}
    \max (\tmop{Ep} (t) \cap \alpha) \text{ iff } \tmop{Ep} (t) \cap
    \alpha \neq \emptyset\\
    0 \text{ otherwise }
  \end{array} \right.$.

  On the other hand, we define $G : [\varepsilon_{\omega}, \infty)
  \longrightarrow \tmop{Subclases} (\tmop{OR})$ in the following way: Consider
  $t \in [\varepsilon_{\omega}, \infty)$ and $\alpha \in \mathbbm{E}$ such
  that $t \in [\alpha, \alpha^+)$. Let

  $G (t) \assign \{\beta \in \tmop{OR} | \tmop{Ep} (t) \cap \alpha \subset
  \beta \leqslant \alpha \text{\tmop{and}} \beta \leqslant^1 (\eta t) [\alpha
  \assign \beta] + 1\}\underset{\text{theorem } \ref{teo.A(t)=G(t)}}{=}$
  
  \ \ \ \ \ $= \{\beta \in \mathbbm{E} | \tmop{Ep} (t) \cap \alpha \subset
  \beta \leqslant \alpha \text{\tmop{and}} \beta \leqslant^1 (\eta t) [\alpha
  \assign \beta] + 1\}\underset{\text{proposition } \ref{pi.eta.substitutions} \text{and
  corollary }\ref{<less>_1.iff.<less>^1}}{=}$
  
  \ \ \ \ \ $= \{\beta \in \mathbbm{E} | \tmop{Ep} (t) \cap \alpha \subset
  \beta \leqslant \alpha \text{\tmop{and}} \beta \leqslant_1 (\eta t) [\alpha
  \assign \beta] + 1\}$.
\end{definition}

\begin{theorem}
  \label{teo.A(t)=G(t)}{\tmdummy}
  
  \begin{enumeratenumeric}
    \item $\forall t \in [\varepsilon_{\omega}, \infty) .G (t) \subset
    \tmop{Lim} \mathbbm{E}$
    
    \item $\forall t \in [\varepsilon_{\omega}, \infty) .G (t) = A (t)$
  \end{enumeratenumeric}
\end{theorem}

\begin{proof}

  {\noindent}1.\\
  First observe the following. Let $t \in [\varepsilon_{\omega}, \infty)$ and
  $\alpha \in \mathbbm{E}$ such that $t \in [\alpha, \alpha^+)$ and take
  $\beta \in G (t)$; so $\beta \leqslant^1 (\eta t) [\alpha \assign \beta] +
  1$. Since by proposition \ref{pi.eta.properties} $\alpha 2 \leqslant \eta
  t$, then $\beta < \beta 2 \leqslant (\eta t) [\alpha \assign \beta]$ and so $\beta
  \leqslant^1 \beta 2 + 1$ by $\leqslant^1$-connectedness. So $\beta \in
  \tmop{Lim} \mathbbm{E}$ by proposition \ref{alpha<less>^1alpha2+1}.

  {\noindent}2.\\
  We first prove the following easy case: Let $\alpha \in
  [\varepsilon_{\omega}, \infty) \cap \mathbbm{E}$. Consider $t \in [\alpha,
  \alpha 2]$. Then \\
  $\pi t + d \pi t = \alpha 2$ and so $\eta t = \max \{t, \alpha 2\} = \alpha
  2$. Then\\
  $G (t) = \{\beta \in \mathbbm{E} | \tmop{Ep} (t) \cap \alpha \subset \beta
  \leqslant \alpha \wedge \beta \leqslant^1 \eta t [\alpha \assign \beta] +
  1\} =$
  
  \ $= \{\beta \in \mathbbm{E} | \tmop{Ep} (t) \cap \alpha \subset \beta
  \leqslant \alpha \wedge \beta \leqslant^1 \alpha 2 [\alpha \assign \beta] +
  1\} =$
  
  \ $= \{\beta \in \mathbbm{E} | \beta \leqslant^1 \beta 2 + 1\} \cap (\max
  (\tmop{Ep} (t) \cap \alpha), \alpha + 1) \underset{\text{\tmop{prop} .
  \ref{alpha<less>^1alpha2+1}}}{=} \tmop{Lim} \mathbbm{E} \cap (\max
  (\tmop{Ep} (t) \cap \alpha), \alpha + 1)$.

  On the other hand, we prove by induction that for $t \in [\alpha, \alpha
  2]$, \\
  $A (t) = \tmop{Lim} \mathbbm{E} \cap (\max (\tmop{Ep} (t) \cap \alpha),
  \alpha + 1)$. For $t \in \tmop{Lim}$ it is clear. So suppose $t = l + 1$ is
  a successor. Then $l < l + 1 < \alpha 2 = \pi l + d \pi l$, and so $A (t) =
  A (l + 1) = A (l) =$\\
  $\tmop{Lim} \mathbbm{E} \cap (\max (\tmop{Ep} (t) \cap \alpha), \alpha +
  1)$, where the last equality holds because of the induction hypothesis.

  Hence we have shown that $G (t) = A (t) = \tmop{Lim} \mathbbm{E} \cap (\max
  (\tmop{Ep} (t) \cap \alpha), \alpha + 1)$ for all $t \in [\alpha, \alpha
  2]$, with $\alpha \in [\varepsilon_{\omega}, \infty) \cap \mathbbm{E}$.

  Now we proceed to prove that $G (t) = A (t)$ for arbitrary $t \in
  [\varepsilon_{\omega}, \infty)$. We proceed by induction on the class
  $[\varepsilon_{\omega}, \infty)$.

  So let $t \in [\varepsilon_{\omega}, \infty]$ and $\alpha \in \mathbbm{E}$
  be such that $t \in [\alpha, \alpha^+)$.

  Suppose $\forall l \in [\varepsilon_{\omega}, \infty) \cap t.A (l) = G (l)$.
  \ \ \ \ \ \ \ \ \ \ \ {\tmstrong{(IH)}}

  {\tmstrong{Successor case.}}
  
  Suppose $t = l + 1$.

  Subcase $l < \pi l + d \pi l$. Then $\eta (l + 1) = \max \{l + 1, \pi (l +
  1) + d \pi (l + 1)\} =$, proposition \ref{pi.eta.properties}, {
  
  } \ \ \ \ \ \ \ \ \ \ \ \ \ \ \ \ \ \ \ \ \ \ \ \ \ \ \ \ \ \ \ \ \ \ \ \ \
  \ \ \ \ \ \ $= \max \{l + 1, \pi l + d \pi l\} = \pi l + d \pi l = \max \{l,
  \pi l + d \pi l\} = \eta l$. Thus $G (t) = G (l + 1) = \{\beta \in
  \mathbbm{E} | \tmop{Ep} (l + 1) \cap \alpha \subset \beta \wedge \beta
  \leqslant^1 (\eta (l + 1)) [\alpha \assign \beta] + 1\} =${
  
  } \ \ \ \ \ \ \ \ $= \{\beta \in \mathbbm{E} | \tmop{Ep} (l) \cap \alpha
  \subset \beta \wedge \beta \leqslant^1 (\eta l) [\alpha \assign \beta] + 1\}
  = G (l) \underset{\tmop{IH}}{=} A (l) = A (l + 1) = A (t)$.

  Subcase $l \geqslant \pi l + d \pi l$.
  
  Let's see $G (t) = G (l + 1) \subset A (l + 1) = A (t)$.
  
  Take $\beta \in G (t)$; so $\alpha \geqslant \beta \leqslant^1 (\eta (l +
  1)) [\alpha \assign \beta] + 1$ and $(\eta (l + 1)) [\alpha \assign \beta]
  \in [\beta, \beta^+)$. So, by proposition
  \ref{<less>^1.implies.cofinal.sequence}, there is a sequence $(c_{\xi})_{\xi
  \in X}$, such that $c_{\xi} \in \mathbbm{E}$, $c_{\xi}
  \underset{\tmop{cof}}{\longhookrightarrow} \beta \leqslant \alpha$, \\
  $\tmop{Ep} ((\eta (l + 1)) [\alpha \assign \beta]) \cap \beta \subset
  c_{\xi}$ and $c_{\xi} \leqslant_1 (\eta (l + 1)) [\alpha \assign \beta]
  [\beta \assign c_{\xi}] = (\eta (l + 1)) [\alpha \assign c_{\xi}] =$\\
  $(\eta l + 1) [\alpha \assign c_{\xi}] = (\eta l) [\alpha \assign c_{\xi}] +
  1$ where the last two equalities hold because \\
  $\tmop{Ep} (\eta (l + 1)) \cap \alpha = \tmop{Ep} ((\eta (l + 1)) [\alpha
  \assign \beta]) \cap \beta \subset c_{\xi}$ (and then $(\tmop{Ep} (l) \cap
  \alpha) \cup (\tmop{Ep} (\eta l) \cap \alpha) \subset c_{\xi}$) and because
  $\eta (l + 1) = \max \{l + 1, \pi (l + 1) + d \pi (l + 1)\} =$, by
  proposition \ref{pi.eta.properties},\\
  \ \ \ \ \ \ \ \ \ \ \ \ \ \ \ \ \ \ \ \ \ \ \ \ \ \ $= \max \{l + 1, \pi l
  + d \pi l\} = l + 1 = \eta l + 1$.
  
  Now, by proposition \ref{pi.eta.substitutions}, $(\eta l) [\alpha \assign
  c_{\xi}] = \eta (l [\alpha \assign c_{\xi}])$ so we have $c_{\xi}
  \leqslant_1 \eta (l [\alpha \assign c_{\xi}]) + 1$; moreover, this holds
  iff, by corollary \ref{<less>_1.iff.<less>^1}, $c_{\xi} \leqslant^1 \eta (l
  [\alpha \assign c_{\xi}]) + 1 = (\eta l) [\alpha \assign c_{\xi}]) + 1$.
  This way, we have actually shown that $c_{\xi} \in G (l)$ for all $\xi \in
  X$; therefore $\beta \in \tmop{Lim} G (l) \underset{\tmop{IH}}{=} \tmop{Lim}
  A (l) = A (l + 1) = A (t)$. This shows $G (t) \subset A (t)$.
  
  Let's see $G (t) = G (l + 1) \supset A (l + 1) = A (t)$.
  
  Let $\beta \in A (t) = A (l + 1) = \tmop{LimA} (l) \underset{\tmop{by}
  \tmop{IH}}{=} \tmop{LimG} (l)$. Then there exists a sequence $(c_{\xi})_{\xi
  \in X}$, with $c_{\xi} \in G (l)$ and $c_{\xi}
  \underset{\tmop{cof}}{\longhookrightarrow} \beta$; i.e., for all $\xi \in X$
  it also holds $\tmop{Ep} (l) \cap \alpha \subset c_{\xi} \leqslant \alpha$,
  $c_{\xi} \in \mathbbm{E}$ and \\
  $c_{\xi} \leqslant^1 (\eta l) [\alpha \assign c_{\xi}] + 1 = (\eta l + 1)
  [\alpha \assign c_{\xi}] = (\eta l + 1) [\alpha \assign \beta] [\beta
  \assign c_{\xi}]$. It is easy to see that \\
  $(\tmop{Ep} (l + 1) \cap \alpha) \cup (\tmop{Ep} (\eta l + 1) \cap \alpha)
  \subset \beta$ and that the last equality hold; the reason to introduce them
  is the following: from all the previous we have $\beta \in \mathbbm{E}$,
  $(\eta l + 1) [\alpha \assign \beta] \in [\beta, \beta^+)$, $c_{\xi}
  \underset{\tmop{cof}}{\longhookrightarrow} \beta \leqslant \alpha$, $\forall
  \xi \in X.c_{\xi} \in \mathbbm{E} \wedge \tmop{Ep} ((\eta l + 1) [\alpha
  \assign \beta]) \cap \beta \subset c_{\xi}$ and $c_{\xi} \leqslant_1 (\eta l
  + 1) [\alpha \assign \beta] [\beta \assign c_{\xi}]$. Therefore, applying
  proposition \ref{2nd_Fund_Cof_Property_<less>^1}, $\beta \leqslant_1 = (\eta
  l + 1) [\alpha \assign \beta] + 1 = (\eta (l + 1)) [\alpha \assign \beta] +
  1 = \eta ((l + 1) [\alpha \assign \beta]) + 1$, where the last equalities
  hold because of proposition \ref{pi.eta.substitutions}. From this, and
  corollary \ref{<less>_1.iff.<less>^1} we get $\beta \leqslant^1 \eta ((l +
  1) [\alpha \assign \beta]) + 1 = (\eta (l + 1)) [\alpha \assign \beta] + 1$.
  So we have shown $\beta \in G (l + 1) = G (t)$.

  {\tmstrong{Limit case.}}
  
  Suppose $t \in \tmop{Lim}$. Moreover, since we have already proved what
  happens for $t \in [\alpha, \alpha 2]$, then suppose $t \in (\alpha 2,
  \alpha^+)$.

  Subcase $t > \pi t + d \pi t$.

  To show $G (t) \subset A (t)$.
  
  Let $\beta \in G (t)$. So $\alpha \geqslant \beta \leqslant^1 (\eta t)
  [\alpha \assign \beta] + 1 = t [\alpha \assign \beta] + 1$. Then by
  proposition \ref{<less>^1.implies.cofinal.sequence} there exists a sequence
  $(c_{\xi})_{\xi \in X}$ such that $\tmop{Ep} (t) \cap \alpha = \tmop{Ep} (t
  [\alpha \assign \beta]) \cap \beta \subset c_{\xi}$ (so $c_{\xi} > 1$),
  $c_{\xi} \underset{\tmop{cof}}{\longhookrightarrow} \beta$ and \\
  $c_{\xi} \leqslant_1 t [\alpha \assign \beta] [\beta \assign c_{\xi}] = t
  [\alpha \assign c_{\xi}]$.
  
  On the other hand, by lemma \ref{l_j-sequence} we know that for the sequence
  $(l_j)_{j \in I}$, it holds:
  
  {\noindent}- $I \cup \{0\} \leqslant \alpha${
  
  }- $(l_j [\alpha \assign c_{\xi}])_{j \in I \cap c_{\xi}}$ is cofinal in $t
  [\alpha \assign c_{\xi}]$ and\\
  - For any $j \in I \cap c_{\xi}$, $\eta l_j [\alpha \assign c_{\xi}] < \eta
  t [\alpha \assign c_{\xi}]$.\\
  Therefore, for any $\xi \in X$ and for any $j \in I \cap c_{\xi}$, $\alpha
  \geqslant c_{\xi} \leqslant l_j [\alpha \assign c_{\xi}] + 1 \leqslant \eta
  l_j [\alpha \assign c_{\xi}] + 1 \leqslant$\\
  $\eta t [\alpha \assign c_{\xi}] = t [\alpha \assign c_{\xi}]$, which
  implies, by $\leqslant_1$-connectedness, $\forall j \in I \cap c_{\xi}$,
  $c_{\xi} \leqslant_1 \eta l_j [\alpha \assign c_{\xi}] + 1$. Then, by
  corollary \ref{<less>_1.iff.<less>^1} we obtain $\forall j \in I \cap
  c_{\xi}$, $c_{\xi} \leqslant^1 \eta l_j [\alpha \assign c_{\xi}] + 1$.
  
  The previous shows $c_{\xi} \underset{\tmop{cof}}{\longhookrightarrow}
  \beta$, and $\tmop{Ep} (t) \cap \alpha \subset c_{\xi} \in \bigcap_{j \in I
  \cap c_{\xi}} G (l_j)  \underset{\tmop{IH}}{=}  \bigcap_{j \in I \cap
  c_{\xi}} A (l_j)$. Thus $\beta \in \tmop{Lim} \{r \leqslant \alpha |M < r
  \in \bigcap_{j \in I \cap r} A (l_j)\} = A (t)$.

  To show $G (t) \supset A (t)$.
  
  Let $\beta \in A (t) = \tmop{Lim} \{r \leqslant \alpha |M < r \in \bigcap_{j
  \in I \cap r} A (l_j)\}$. Since we know $l_j < t$ for any $j \in I$, then $A
  (l_j) = G (l_j)$ for any $j \in I \cap r$ by induction hypothesis. This way
  \\
  $\beta \in \tmop{Lim} \{r \leqslant \alpha |M < r \in \bigcap_{j \in I \cap
  r} G (l_j)\}$, which means there exists a sequence $(c_{\xi})_{\xi \in X}$,
  such that $\alpha \geqslant c_{\xi}
  \underset{\tmop{cof}}{\longhookrightarrow} \beta$ and $\tmop{Ep} (t) \cap
  \alpha \subset c_{\xi} \in \bigcap_{j \in I \cap c_{\xi}} G (l_j)$; i.e.,
  $\forall j \in I \cap c_{\xi}$, $c_{\xi} \leqslant^1 \eta l_j [\alpha
  \assign c_{\xi}] + 1$. This way, for any $\xi \in X$ and any $j \in I \cap
  c_{\xi}$, $c_{\xi} \leqslant^1 l_j [\alpha \assign c_{\xi}]$ by
  $<^1$-connectedness (because from $\alpha < l_j \leqslant \eta l_j$ follows
  $c_{\xi} < l_j [\alpha \assign c_{\xi}] \leqslant \eta l_j [\alpha \assign
  c_{\xi}] < \eta l_j [\alpha \assign c_{\xi}] + 1$). But by lemma
  \ref{l_j-sequence}, $(l_j [\alpha \assign c_{\xi}])_{j \in I \cap c_{\xi}}$
  is cofinal in $t [\alpha \assign c_{\xi}]$; therefore $\forall \xi \in
  X.c_{\xi} \leqslant^1 t [\alpha \assign c_{\xi}]$ by
  $\leqslant^1$-continuity. \ \ \ (*)
  
  On the other hand, $\forall \xi \in X.c_{\xi} \in \mathbbm{E}$ (because
  $c_{\xi} \in G (l_1) \underset{\text{\tmop{by} 1.}}{\subset} \tmop{Lim}
  \mathbbm{E}$), and since\\
  $\tmop{Ep} (t) \cap \alpha \subset c_{\xi}
  \underset{\tmop{cof}}{\longhookrightarrow} \beta$, then $\tmop{Ep} (t) \cap
  \alpha \subset \beta \in \mathbbm{E}$. So $t [\alpha \assign \beta] \in
  [\beta, \beta^+)$ and\\
  $\tmop{Ep} (t) \cap \alpha = \tmop{Ep} (t [\alpha \assign \beta]) \cap
  \beta$. From all this and the fact that (*) implies \\
  $\forall \xi \in X.c_{\xi} \leqslant_1 t [\alpha \assign c_{\xi}] = t
  [\alpha \assign \beta] [\beta \assign c_{\xi}]$, we conclude \\
  $\beta \in \tmop{Lim} \{\gamma \in \mathbbm{E} | \tmop{Ep} (t [\alpha
  \assign \beta]) \cap \beta \subset \gamma \wedge \gamma \leqslant_1 t
  [\alpha \assign \beta] [\beta \assign \gamma]\}$. This implies, by
  proposition \ref{2nd_Fund_Cof_Property_<less>^1}, $\beta \leqslant_1 t
  [\alpha \assign \beta] + 1 = \eta t [\alpha \assign \beta] + 1$ and
  subsequently, by corollary \ref{<less>_1.iff.<less>^1}, $\beta \leqslant^1
  \eta t [\alpha \assign \beta] + 1$. \\
  So $\beta \in G (t)$.

  All the previous shows $G (t) = A (t)$ for the subcase $t > \pi t + d \pi
  t$.

  Subcase $t \leqslant \pi t + d \pi t$.

  Write $t =_{\tmop{CNF}} \omega^{T_1} t_1 + \ldots + \omega^{T_n} t_n$ and
  $T_1 =_{\tmop{CNF}} \omega^{Q_1} q_1 + \ldots + \omega^{Q_m} q_m$. Note $Q_m
  < T_1$ and then $\omega^{Q_m} < \omega^{T_1}$ (otherwise $T_1 = Q_m
  \leqslant \omega^{Q_m} \leqslant T_1$ and then $T_1 = Q_m \in \mathbbm{E}$;
  from this and the fact that \\
  $t \in (\alpha 2, \alpha^+)$ follows that $T_1 = \alpha$, but then $t
  \leqslant \omega^{\alpha} + \alpha = \alpha 2$, which is contradictory with
  our supposition $t \in (\alpha 2, \alpha^+)$). The previous also shows $T_1
  > \alpha$. This way, the inequalities \\
  $t \leqslant \pi t + d \pi t = \omega^{T_1} + Q_m$ and $Q_m < T_1$ imply
  that $t$ looks like $t =_{\tmop{CNF}} \omega^{T_1} + \omega^{T_2} t_2 \ldots
  + \omega^{T_n} t_n$ with $\omega^{T_2} t_2 \ldots + \omega^{T_n} t_n
  \leqslant Q_m$, and $T_1 > \alpha$.

  Lets show now $G (t) \subset A (t)$.
  
  Let $\beta \in G (t)$. So $\tmop{Ep} (t) \cap \alpha \subset \beta
  \leqslant^1 (\eta t) [\alpha \assign \beta] + 1$ and $\alpha \geqslant \beta
  \in \mathbbm{E}$. Then, by proposition
  \ref{<less>^1.implies.cofinal.sequence}, there is a sequence $(c_{\xi})_{\xi
  \in X}$, such that $c_{\xi} \in \mathbbm{E}$, $c_{\xi}
  \underset{\tmop{cof}}{\longhookrightarrow} \beta$, $\tmop{Ep} (\eta t) \cap
  \alpha = \tmop{Ep} (\eta t [\alpha \assign \beta]) \cap \beta \subset
  c_{\xi}$, and $c_{\xi} \leqslant_1 \eta t [\alpha \assign \beta] [\beta
  \assign c_{\xi}] = \eta t [\alpha \assign c_{\xi}] = \eta \pi t [\alpha
  \assign c_{\xi}]$.
  
  We now need to remember how the sequence $ (e_j)_{j \in S}$ is defined.
  Consider the ordinal $\omega^{Q_m}$. If $Q_m > \alpha$, let $(a_j)_{j \in
  K}$, with $K \leqslant \alpha$ be the sequence obtained by lemma
  \ref{l_j-sequence} applied to $\omega^{Q_m}$. \\
  If $0 \neq Q_m \leqslant \alpha$ let $K \assign \omega^{Q_m} \backslash
  \{0\}$ and $a_j \assign j$ for any $j \in K$. Then
  
  $e_j = \left\{ \begin{array}{l}
    \omega^{\omega^{Q_1} q_1 + \ldots + \omega^{Q_{(m - 1)}} q_{(m - 1)} + q_m
    - 1} j, \text{ with } j \in S \assign \omega \backslash \{0\}
    \text{ iff } Q_m = 0\\
    \\
    \omega^{\omega^{Q_1} q_1 + \ldots + \omega^{Q_{(m - 1)}} q_{(m - 1)} +
    \omega^{Q_m} (q_m - 1) + a_j}, \text{ with } j \in S \assign K
    \text{ iff } Q_m \neq 0
  \end{array} \right.$.

  {\noindent}As we know, $(e_j)_{j \in S}$ is cofinal in $\omega^{T_1} = \pi
  t$. Besides, since $\forall \xi \in X. \tmop{Ep} (\eta t) \cap \alpha
  \subset c_{\xi}$, then \\
  $\forall \xi \in X. \tmop{Ep} (\pi t) \cap \alpha \subset c_{\xi}$; this
  way, for any $\xi \in X$,
  
  - for any $j \in S \cap c_{\xi}$, $\tmop{Ep} (e_j) \cap \alpha \subset
  c_{\xi}$ and
  
  - $(e_j [\alpha \assign c_{\xi}])_{j \in S \cap c_{\xi}}$ is cofinal in
  $\omega^{T_1} [\alpha \assign c_{\xi}]$. \\
  Moreover, notice $\forall j \in S. \eta e_j < \eta t = \eta \pi t$; so
  $\forall j \in S \cap c_{\xi} . \eta e_j [\alpha \assign c_{\xi}] < \eta \pi
  t [\alpha \assign c_{\xi}]$.

  From all our previous work we obtain: \ $\forall \xi \in X. \forall j \in S
  \cap c_{\xi}$.$c_{\xi} \leqslant \eta e_j [\alpha \assign c_{\xi}] + 1
  \leqslant \eta \pi t [\alpha \assign c_{\xi}]$, which implies,
  by $\leqslant_1$-connectedness, $\forall \xi \in X. \forall j \in S \cap
  c_{\xi}$.$c_{\xi} \leqslant_1 \eta e_j [\alpha \assign c_{\xi}] + 1$, which
  in turn is equivalent (by corollary \ref{<less>_1.iff.<less>^1}) to $\forall
  \xi \in X. \forall j \in S \cap c_{\xi}$.$c_{\xi} \leqslant^1 \eta e_j
  [\alpha \assign c_{\xi}] + 1$. Finally, since \\
  $c_{\xi} \underset{\tmop{cof}}{\longhookrightarrow} \beta \supset \tmop{Ep}
  (t) \cap \alpha$ and $\tmop{Ep} (t) \cap \alpha$ is a finite set, then there
  exists $y \subset X$ such that $(c_{\xi})_{\xi \in (X \backslash y)}
  \underset{\tmop{cof}}{\longhookrightarrow} \beta$ and $\forall \xi \in (X
  \backslash y) . \forall j \in S \cap c_{\xi}$.$\tmop{Ep} (t) \cap \alpha
  \subset c_{\xi} \leqslant^1 \eta e_j [\alpha \assign c_{\xi}] + 1$.
  
  The previous paragraph shows $\forall \xi \in X \backslash y$.$M < c_{\xi}
  \in \bigcap_{j \in S \cap c_{\xi}} G (e_j)  \underset{\tmop{IH}}{=}
  \bigcap_{j \in S \cap c_{\xi}} A (e_j)$ and $(c_{\xi})_{\xi \in (X
  \backslash y)} \underset{\tmop{cof}}{\longhookrightarrow} \beta$; i.e., it
  shows $\beta \in \tmop{Lim} \{r \leqslant \alpha |M < r \in \bigcap_{j \in S
  \cap r} A (e_j)\} = A (t)$.

  To show $G (t) \supset A (t)$.
  
  Let $\beta \in A (t) = \tmop{Lim} \{r \leqslant \alpha |M < r \in \bigcap_{j
  \in S \cap r} A (e_j)\} \underset{\tmop{IH}}{=} \tmop{Lim} \{r \leqslant
  \alpha |M < r \in \bigcap_{j \in S \cap r} G (e_j)\}$. Then there exists a
  sequence $(c_{\xi})_{\xi \in X}$ such that $\alpha \geqslant c_{\xi}
  \leqslant^1 \eta e_j [\alpha \assign c_{\xi}] + 1$ for all $j \in S \cap
  c_{\xi}$ and $M < c_{\xi} \underset{\tmop{cof}}{\longhookrightarrow} \beta$.
  Of course, the last inequality means $\tmop{Ep} (t) \cap \alpha \subset
  c_{\xi}$, which implies $\tmop{Ep} (\pi t) \cap \alpha \subset c_{\xi}$.
  
  Now, noting that $\forall j \in S \cap c_{\xi} .c_{\xi} \leqslant e_j
  [\alpha \assign c_{\xi}] < e_j [\alpha \assign c_{\xi}] + 1 \leqslant \eta
  e_j [\alpha \assign c_{\xi}] + 1$, we obtain by $\leqslant^1$-connectedness
  $c_{\xi} \leqslant^1 e_j [\alpha \assign c_{\xi}]$ for all $j \in S \cap
  c_{\xi}$. But the fact that $\forall \xi \in X. \tmop{Ep} (\pi t) \cap
  \alpha \subset c_{\xi}$ implies (by lemma \ref{l_j-sequence}) that the
  sequence $(e_j [\alpha \assign c_{\xi}])_{j \in S \cap c_{\xi}}$ is cofinal
  in $\pi t [\alpha \assign c_{\xi}]$ for any $\xi \in X$, and so we conclude
  $\forall \xi \in X.c_{\xi} \leqslant^1 \pi t [\alpha \assign c_{\xi}]$ by
  $\leqslant^1$-continuity.
  
  From the work done in the previous paragraph follows immediately that \\
  $\forall \xi \in X.c_{\xi} \leqslant_1 \pi t [\alpha \assign c_{\xi}]$; but
  $\pi t [\alpha \assign c_{\xi}] \leqslant_1 \pi t [\alpha \assign c_{\xi}] +
  d \pi t [\alpha \assign c_{\xi}]$ by theorem \ref{Wilken_Theorem1}; thus by
  $\leqslant_1$-transitivity we conclude $\forall \xi \in X.c_{\xi}
  \leqslant_1 \pi t [\alpha \assign c_{\xi}] + d \pi t [\alpha \assign
  c_{\xi}] = (\pi t + d \pi t) [\alpha \assign c_{\xi}] = \eta t [\alpha
  \assign c_{\xi}]$, where the last two equalities hold by proposition
  \ref{pi.eta.substitutions}. Finally applying proposition
  \ref{2nd_Fund_Cof_Property_<less>^1} to \\
  $\forall \xi \in X.c_{\xi} \leqslant_1 \eta t [\alpha \assign c_{\xi}]$ and
  to $M < c_{\xi} \underset{\tmop{cof}}{\longhookrightarrow} \beta$, and using
  the fact that $\forall \xi \in X.c_{\xi} \leqslant \alpha$, we conclude
  $\tmop{Ep} (t) \cap \alpha \subset \beta \leqslant \alpha$ and $\beta
  \leqslant_1 \eta t [\alpha \assign \beta] + 1$. Observe the latter is
  equivalent (by corollary \ref{<less>_1.iff.<less>^1}) to $\beta \leqslant^1
  \eta t [\alpha \assign \beta] + 1$. Thus $\beta \in G (t)$.
\end{proof}

\subsection{Uncountable regular ordinals and the $A (t)$ sets}

Up to this moment we have shown that the sets $A (t)$ consists of the ordinals
that are ``solutions of certain $<_1$-inequalities of the form $x <_1 t$ with
$t \in [x, x^+)$'', but we still do not know whether these solutions indeed
exist. We address this problem now: our purpose is to study closer the $A (t)$
sets and, very specifically, by the introduction of an uncountable regular
ordinal $\kappa$, show that for $t \in [\kappa, \kappa^+)$, the $A (t)$ sets
have to have elements.

In the following we will use the next two propositions.

\begin{proposition}
  \label{X_club_implies_LimX_club}Let $\kappa$ be an uncountable regular
  ordinal and let $X$ be a class of ordinals that are club in $\kappa$. Then
  $\tmop{Lim} X$ is club in $\kappa$.
\end{proposition}

\begin{proof}
  Known result about club classes.
\end{proof}

\begin{proposition}
  \label{Intersection_club_classes}Let $\kappa$ be an uncountable regular
  ordinal and let $(X_i)_{i < I}$ be a sequence of classes of ordinals that
  are club in $\kappa$.
  \begin{itemizedot}
    \item If $|I| < \kappa$, then $\underset{i < I}{\bigcap X_i}$ is club in
    $\kappa$.
    
    \item Suppose $I = \kappa$. Then $\{\xi < \kappa | \xi \in \underset{i <
    \xi}{\bigcap X_i} \}$ is club in $\kappa$.
  \end{itemizedot}
\end{proposition}

\begin{proof}
  Known result about club classes.
\end{proof}

\begin{proposition}
  \label{A(t)_club_in_kapa}Let $\kappa$ be an uncountable regular ordinal.
  Then $\forall t \in [\kappa, \kappa^+)$, $A (t)$ is club in $\kappa$.
\end{proposition}

\begin{proof}
  We prove the claim by induction on the interval $[\kappa, \kappa^+)$.
  
  {\tmstrong{Case}} $\tmmathbf{t = \kappa .}$\\
  Then $A (t) = (\tmop{Lim} \mathbbm{E}) \cap (0, \kappa + 1)$ is club in
  $\kappa$ because $\mathbbm{E}$ is club in $\kappa$ and by proposition
  \ref{X_club_implies_LimX_club}.

  Our induction hypothesis is $\forall s < t.A (s)$ is club in $\kappa$. \ \ \
  \ \ \ \ {\tmstrong{(IH)}}

  {\tmstrong{Case}} $\tmmathbf{t = l + 1 \in [\kappa, \kappa^+).}$\\
  Then $A (t) = A (l + 1) = \left\{ \begin{array}{l}
    A (l) \text{ if } l < \pi l + d \pi l\\
    \\
    \tmop{Lim} A (l) \text{ otherwise}
  \end{array} \right.$; this way, by our IH and proposition
  \ref{X_club_implies_LimX_club}, $A (t)$ is club in $\kappa$ in any case.

  {\tmstrong{Case}} $\tmmathbf{t \in [\kappa, \kappa^+) \cap \tmop{Lim} .}$
  
  Let $(l_j)_{j \in I}$ be the sequence obtained by the application of lemma
  \ref{l_j-sequence} to $t$ and $\kappa$ and $(e_j)_{j \in S}$ the sequence
  obtained by the application of lemma \ref{l_j-sequence} to $\pi t$ and
  $\kappa$. Moreover, in case \\
  $\tmop{Ep} (t) \cap \alpha \neq \emptyset$, let $M \assign \max \tmop{Ep}
  (t) \cap \alpha$; in case $\tmop{Ep} (t) \cap \alpha = \emptyset$, let $M
  \assign 0$. Then by definition

  {\noindent}$A (t) = \left\{ \begin{array}{l}
    (\tmop{Lim} \mathbbm{E}) \cap (M, \kappa + 1) \text{ iff } t \in
    [\kappa, \kappa 2]\\
    \\
    \tmop{Lim} \{r \leqslant \kappa |M < r \in \bigcap_{j \in I \cap r} A
    (l_j)\} \text{ iff } t > \pi t + d \pi t \wedge t \in (\kappa 2,
    \kappa^+)\\
    \\
    \tmop{Lim} \{r \leqslant \kappa |M < r \in \bigcap_{j \in S \cap r} A
    (e_j)\} \text{ iff } t \leqslant \pi t + d \pi t \wedge t \in (\kappa
    2, \kappa^+)
  \end{array} \right.$\\
  and we have some subcases:

  If $t \in [\kappa, \kappa 2]$, then $A (t) = (\tmop{Lim} \mathbbm{E}) \cap
  (M, \kappa + 1)$ is club in $\kappa$ because $\mathbbm{E}$ is club in
  $\kappa$ and because of proposition \ref{X_club_implies_LimX_club}.

  Subcase $t > \pi t + d \pi t \wedge t \in (\kappa 2, \kappa^+)$.
  
  Note it is enough to show that $Y \assign \{r \leqslant \kappa |M < r \in
  \bigcap_{j \in I \cap r} A (l_j)\}$ is club in $\kappa$ because knowing this
  we conclude Lim$Y$ is club in $\kappa$ by proposition
  \ref{X_club_implies_LimX_club}. In order to see that $Y$ is club in
  $\kappa$, we define for any $i \in \kappa$, $X_i \assign \left\{
  \begin{array}{l}
    A (l_i) \text{ iff } i \in I\\
    \\
    A (l_1) \text{ iff } i \nin I
  \end{array} \right.$. Since by lemma \ref{l_j-sequence}, $l_j < t$ for any
  $j \in I$, then by our IH we have that $X_i$ is club in $\kappa$ for any $i
  \in \kappa$; consequently, by proposition \ref{Intersection_club_classes},
  the set \\
  $X \assign \{\xi < \kappa | \xi \in \bigcap_{i < \xi} X_i \}$ is club in
  $\kappa$.

  We now show $Y \cap \kappa = X \backslash (M + 1)$.
  
  $'' \supset''$. Let $r \in X \backslash (M + 1)$. Then $M + 1 \leqslant r <
  \kappa$ and $r \in \bigcap_{i < r} X_i \subset \bigcap_{i \in I \cap r} X_i
  = \bigcap_{i \in I \cap r} A (l_i)$. This shows $r \in Y \cap \kappa$.
  
  $'' \subset''$. Let $r \in Y \cap \kappa$. Then $M < r < \kappa$ and $r \in
  \bigcap_{i \in I \cap r} A (l_i) = \bigcap_{i \in I \cap r} X_i = \bigcap_{i
  \in I \cap r} X_i \cap X_1 =$\\
  $= \bigcap_{i \in I \cap r} X_i \cap \bigcap_{i \in (r \backslash I)} X_i =
  \bigcap_{i < r} X_i$. So $r \in X \backslash (M + 1)$.

  Hence, since $Y \cap \kappa = X \backslash (M + 1)$ and $X$ is club in
  $\kappa$, then $Y = \left\{ \begin{array}{l}
    Y \cap \kappa \cup \{\kappa\} \text{ iff } \kappa \in Y\\
    Y \cap \kappa \text{ otherwise}
  \end{array} \right.$ is also club in $\kappa$.

  Subcase $t \leqslant \pi t + d \pi t \wedge t \in (\kappa 2, \kappa^+)$.
  
  It is enough to show that $Z \assign \{r \leqslant \kappa |M < r \in
  \bigcap_{j \in S \cap r} A (e_j)\}$ is club in $\kappa$ because of the same
  reasons of the previous subcase. For any $i \in \kappa$, let $W_i \assign
  \left\{ \begin{array}{l}
    A (e_i) \text{ iff } i \in S\\
    \\
    A (e_1) \text{ iff } i \nin S
  \end{array} \right.$. Since by lemma \ref{l_j-sequence}, $e_j < \pi t
  \leqslant t$ for any $j \in S$, then by our IH we have that $W_i$ is club in
  $\kappa$ for any $i \in \kappa$; consequently, by proposition
  \ref{Intersection_club_classes}, the set $W \assign \{\xi < \kappa | \xi \in
  \bigcap_{i < \xi} W_i \}$ is club in $\kappa$.

  We show $Z \cap \kappa = W \backslash (M + 1)$.
  
  $'' \supset''$. Let $r \in W \backslash (M + 1)$. Then $M + 1 \leqslant r <
  \kappa$ and $r \in \bigcap_{i < r} W_i \subset \bigcap_{i \in S \cap r} W_i
  = \bigcap_{i \in S \cap r} A (e_i)$. From this we conclude $r \in Z \cap
  \kappa$.
  
  $'' \subset''$. Let $r \in Z \cap \kappa$. Then $M < r < \kappa$ and $r \in
  \bigcap_{i \in S \cap r} A (e_i) = \bigcap_{i \in S \cap r} W_i = \bigcap_{i
  \in S \cap r} W_i \cap W_1 =$\\
  $= \bigcap_{i \in S \cap r} W_i \cap \bigcap_{i \in (r \backslash S)} W_i =
  \bigcap_{i < r} W_i$. So $r \in W \backslash (M + 1)$.

  Therefore, since $Z \cap \kappa = W \backslash (M + 1)$ and $W$ is club in
  $\kappa$, then $Z = \left\{ \begin{array}{l}
    Z \cap \kappa \cup \{\kappa\} \text{ iff } \kappa \in Z\\
    Z \cap \kappa \text{ otherwise}
  \end{array} \right.$ is club in $\kappa$.
\end{proof}

Consider an epsilon number $\alpha \in \mathbbm{E}$, $t \in [\alpha,
\alpha^+)$ and a non-countable regular ordinal $\kappa > \alpha$. The
``solutions to the $<_1$-inequality $x <_1 \eta t [\alpha \assign x] + 1$ in
interval $[0, \kappa]$'' are the same as the ``solutions to the
$<_1$-inequality $x <_1 \eta t [\alpha \assign \kappa] [\kappa \assign x] + 1$
in interval $[0, \kappa]$'', which are, of course, the elements of the set $G
(t [\alpha \assign \kappa])\underset{\text{theorem } \ref{teo.A(t)=G(t)}}{=}A (t
[\alpha \assign \kappa])$. This way, proposition \ref{A(t)_club_in_kapa} shows
us that such solutions are indeed many: $G (t [\alpha \assign \kappa])$ is
club in $\kappa$. So our hierarchy $A (l)_{l \in [\kappa, \kappa^+)}$ captures
all these solutions (in interval $[0, \kappa]$) and such solutions do exist.
Now we just want to make explicit that we get a similar result for arbitrary
``$<_1$-inequality $x <_1 t [\alpha \assign x]$''.

\begin{proposition}
  \label{k-club_x<less>_1t(x)}Let $\kappa$ be an uncountable regular ordinal
  and $\alpha \in \mathbbm{E} \cap \kappa$.\\
  Then for any $t \in [\alpha 2, \alpha^+)$, there are $\gamma \in \mathbbm{E}
  \cap \kappa$ and $s \in [\gamma 2, \gamma^+)$ such that \\
  $\{\beta \in \mathbbm{E} | \tmop{Ep} (t) \cap \alpha \subset \beta \leqslant
  \kappa \wedge \beta \leqslant_1 t [\alpha \assign \beta]\} = \left\{
  \begin{array}{l}
    \mathbbm{E} \cap (\kappa + 1) \text{ iff } t = \alpha 2\\
    \\
    {}[\gamma, \kappa] \cap \bigcap_{\xi \in [\gamma 2, s)} A (\xi [\gamma
    \assign \kappa])  \text{ iff } t \in (\alpha 2, \alpha^+)
  \end{array} \right.$ and the set $\{\beta \in \mathbbm{E} | \tmop{Ep} (t)
  \cap \alpha \subset \beta \leqslant \kappa \wedge \beta \leqslant_1 t
  [\alpha \assign \beta]\}$ is club in $\kappa$.
\end{proposition}

\begin{proof}
  Let $\kappa$ and $\alpha \in \mathbbm{E} \cap \kappa$ be as stated. Take $t
  \in [\alpha 2, \alpha^+)$ and consider \\
  $\gamma \assign \min \{e \in \mathbbm{E} \cap \kappa | \tmop{Ep} (t) \cap
  \alpha \subset e\}$ ($\gamma$ exists because $\tmop{Ep} (t) \cap \alpha
  \subset \alpha < \kappa$). Then $t [\alpha \assign \gamma] \in [\gamma 2,
  \gamma^+)$ and $C \assign \{\beta \in \mathbbm{E} | \tmop{Ep} (t) \cap
  \alpha \subset \beta \leqslant \kappa \wedge \beta \leqslant_1 t [\alpha
  \assign \beta]\} =$\\
  \ \ \ \ \ \ \ \ \ \ \ \ \ \ \ $= \{\beta \in \mathbbm{E} | \tmop{Ep} (t
  [\alpha \assign \gamma]) \cap \gamma \subset \beta \leqslant \kappa \wedge
  \beta \leqslant_1 t [\alpha \assign \gamma] [\gamma \assign \beta]\}$.

  We have two cases:

  {\noindent}$\bullet$ Case $t [\alpha \assign \gamma] = \gamma 2$. Then $C =
  \mathbbm{E} \cap (\kappa + 1)$ is (clearly) club in $\kappa$.
  
  {\noindent}$\bullet$ Case $t [\alpha \assign \gamma] \in (\gamma 2,
  \gamma^+)$. Let $s \assign \min \{z \in [\gamma 2, t [\alpha \assign
  \gamma]] | m (z) \geqslant t [\alpha \assign \gamma]\}$ (of course $s$
  exists because $m (t [\alpha \assign \gamma]) \geqslant t [\alpha \assign
  \gamma]$). Let $Z \assign \max (\tmop{Ep} (s) \cap \gamma)$. We show that

  \ \ \ \ \ \ \ \ \ \ $C = [\gamma, \kappa] \cap\underset{\xi \in [\gamma
  2, s)}{\bigcap}A (\xi [\gamma \assign \kappa]) = [\gamma, \kappa]
  \cap\underset{\xi \in [\gamma 2, s)}{\bigcap}G (\xi [\gamma \assign
  \kappa])$. \ \ \ \ \ \ \ {\tmstrong{($\circ$)}}

  {\noindent}``$\subset$''. Take $\beta \in C$. Then $\beta \geqslant \gamma$
  (because $\tmop{Ep} (t) \cap \alpha \subset \beta$) and \ \\
  $\tmop{Ep} (t [\alpha \assign \gamma]) \cap \gamma \subset \beta \in
  \mathbbm{E} \cap (\kappa + 1) \wedge \beta \leqslant_1 t [\alpha \assign
  \gamma] [\gamma \assign \beta]$. \ \ \ \ \ \ \ {\tmstrong{(*)}}
  
  {\noindent}On the other hand, let $\xi \in [\gamma 2, s)$ be arbitrary. Then
  $\tmop{Ep} (\eta \xi + 1) \cap \gamma \subset \gamma \leqslant \beta$, and
  then, using that $\eta \xi + 1\underset{\text{proposition }
  \ref{eta(t)_m(t)_and_<less>_1}}{=} \max \{m (\alpha) | a \in (\alpha,
  \xi]\} + 1 \leqslant t [\alpha \assign \gamma]$ we get \\
  $\beta \leqslant (\eta \xi + 1) [\gamma \assign \beta] \leqslant t [\alpha
  \assign \gamma] [\gamma \assign \beta]$. This, (*) and
  $\leqslant_1$-connectedness imply that \\
  $\beta \leqslant_1 (\eta \xi + 1) [\gamma \assign \beta] = \eta \xi [\gamma
  \assign \beta] + 1$.
  
  Our previous work has shown that $\beta \geqslant \gamma$ and that\\
  $\beta \in \{e \in \mathbbm{E} | \tmop{Ep} (\xi) \cap \gamma \subset e
  \leqslant \kappa \text{\tmop{and}} e \leqslant_1 \eta \xi [\gamma \assign e]
  + 1\} =$\\
  \ \ \ \ $\{e \in \mathbbm{E} | \tmop{Ep} (\xi [\gamma \assign \kappa]) \cap
  \kappa \subset e \leqslant \kappa \text{\tmop{and}} e \leqslant_1 (\eta \xi
  [\gamma \assign \kappa]) [\kappa \assign e] + 1\}\underset{\text{theorem }
  \ref{teo.A(t)=G(t)}}{=}A (\xi [\gamma \assign \kappa])$. Since this was
  done for arbitrary $\xi \in [\gamma 2, s)$, we have shown $\beta \in
  [\gamma, \kappa] \cap\underset{\xi \in [\gamma 2, s)}{\bigcap}A (\xi
  [\gamma \assign \kappa])$.

  {\noindent}``$\supset$''. Take $\beta \in [\gamma, \kappa]
  \cap\underset{\xi \in [\gamma 2, s)}{\bigcap}A (\xi [\gamma \assign
  \kappa])$. Then for any $\xi \in [\gamma 2, s)$, \\
  $\tmop{Ep} (\xi [\gamma \assign \alpha]) \cap \alpha = \tmop{Ep} (\xi) \cap
  \gamma = \tmop{Ep} (\xi [\gamma \assign \kappa]) \cap \kappa \subset \beta
  \leqslant \kappa$ and \\
  $\beta \leqslant_1 (\eta \xi [\gamma \assign \kappa]) [\kappa \assign \beta]
  + 1 = \eta \xi [\gamma \assign \beta] + 1 = \eta \xi [\gamma \assign \alpha]
  [\alpha \assign \beta] + 1$. \ \ \ \ \ \ \ {\tmstrong{(**)}}

  Subcase $s \in \tmop{Lim}$. \ Since by (**) we have that $\forall [\gamma 2,
  s) . \beta \leqslant_1 \eta \xi [\gamma \assign \beta] + 1$, then by
  $\leqslant_1$-connectedness $\forall \xi \in [\gamma 2, s) . \beta
  \leqslant_1 \xi [\gamma \assign \beta]$; but $(\xi [\gamma \assign
  \beta])_{[\gamma 2, s)}\underset{\text{cof}}{\longhookrightarrow}s [\gamma
  \assign \beta]$, thus, by $\leqslant_1$-continuity, $\beta \leqslant_1 s
  [\gamma \assign \beta]$. \ \ \ \ \ \ \ {\tmstrong{(***)}}. On the other
  hand, the inequalities $s [\gamma \assign \beta] \leqslant t [\alpha \assign
  \gamma] [\gamma \assign \beta] \leqslant m (s) [\gamma \assign
  \beta]\underset{\text{corollary } \ref{A_[alpha:=e]_isomorphisms}}{=}m (s
  [\gamma \assign \beta])$ imply, by $\leqslant_1$-connectedness, that $s
  [\gamma \assign \beta] \leqslant_1 t [\alpha \assign \gamma] [\gamma \assign
  \beta]$; from this, (***) and $\leqslant_1$-transitivity we conclude $\beta
  \leqslant_1 t [\alpha \assign \gamma] [\gamma \assign \beta] = t [\alpha
  \assign \beta]$. Hence $\beta \in C$.

  Subcase $s = l + 1$ for some $l \in \tmop{OR}$. Then $t [\alpha \assign
  \gamma] \geqslant s = m (s) \geqslant t [\alpha \assign \gamma]$, that is,
  \\
  $l + 1 = s = t [\alpha \assign \gamma]$. On the other hand, $l \leqslant
  \eta l = \max \{m (a) | a \in (\alpha, l]\} < m (s) = s$, so $\eta l = l$
  and from all this we conclude $t [\alpha \assign \gamma] = s = \eta l + 1$.
  But by hypothesis $\beta \in A (l [\gamma \assign \kappa])$, then \\
  $\beta \leqslant_1 \eta l [\gamma \assign \kappa] [\kappa \assign \beta] + 1
  = \eta l [\gamma \assign \beta] + 1 = (\eta l + 1) [\gamma \assign \beta] =
  s [\gamma \assign \beta] = t [\alpha \assign \gamma] [\gamma \assign
  \beta]$. Hence $\beta \in C$.

  The previous concludes the proof of ($\circ$).

  Finally, since $| [\gamma 2, s) | \leqslant |s| < \kappa$ then
  $\underset{\xi \in [\gamma 2, s)}{\bigcap}A (\xi [\gamma \assign
  \kappa])$ is club in $\kappa$ (by proposition
  \ref{Intersection_club_classes}) and therefore $C\underset{\text{by }
  (\circ)}{=}[\gamma, \kappa] \cap\underset{\xi \in [\gamma 2,
  s)}{\bigcap}A (\xi [\gamma \assign \kappa])$ is club in $\kappa$ too.
\end{proof}

\subsection{Epsilon numbers $\alpha$ satisfying $\alpha <_1 \alpha^+$.
Class(2).}

We comment on this subsection after the next

\begin{corollary}
  \label{kapa<less>^1kapa^+}Let $\kappa$ be an uncountable regular ordinal.
  Then
  \begin{enumeratealpha}
    \item $\kappa <^1 \kappa^+$
    
    \item $\kappa \in \underset{t \in [\kappa, \kappa^+)}{\bigcap A (t)}$
  \end{enumeratealpha}
\end{corollary}

\begin{proof}

  {\noindent}$a)$.\\
  By proposition \ref{A(t)_club_in_kapa}, for any $t \in [\kappa, \kappa^+)$,
  $A (t)$ is club in $\kappa$. This means there exist a sequence
  $(c_{\xi})_{\xi \in X}$ such that $c_{\xi} \in A (t)$ and $c_{\xi}
  \underset{\tmop{cof}}{\longhookrightarrow} \kappa$. But by theorem
  \ref{teo.A(t)=G(t)}, $A (t) = G (t) =$ $\{\beta \in \mathbbm{E} | \tmop{Ep} (t) \cap \kappa \subset \beta
  \leqslant \kappa \in \mathbbm{E} \text{\tmop{and}} \beta \leqslant^1 \eta t
  [\kappa \assign \beta] + 1\}$ which implies $\forall \xi \in X.c_{\xi}
  \leqslant_1 \eta t [\kappa \assign c_{\xi}]$. Now, from all this and
  proposition \ref{2nd_Fund_Cof_Property_<less>^1} we obtain $\kappa
  \leqslant^1 \eta t + 1$. The previous shows that \\
  $\forall t \in [\kappa, \kappa^+) . \kappa \leqslant^1 \eta t + 1$, and
  since the sequence $(\eta t + 1)_{\kappa \leqslant t < \kappa^+}$ is cofinal
  in $\kappa^+$, then $\kappa \leqslant^1 \kappa^+$ by $<^1$-continuity.

  {\noindent}$b)$.\\
  $\kappa <^1 \kappa^+$ is equivalent to $\kappa \in \underset{t \in [\kappa,
  \kappa^+)}{\bigcap A (t)}$ by next proposition
  \ref{Intersection_A(t)_equiv_alpha<less>^1alpha^+}.
\end{proof}

We had seen previously that for $\alpha \in \mathbbm{E}$ and $t \in [\alpha,
\alpha^+)$ arbitrary, the ``solutions of the $<_1$-inequality $x <_1 t [\alpha
\assign x]$ in interval $[0, \kappa]$'' can always be given in terms of our
hierarchy $A (l)_{l \in [\kappa, \kappa^+)}$. But we can tell even more:
Consider $B \assign \min \{\beta \in \mathbbm{E} | \beta <_1 \beta^+ \}$
(previous corollary \ref{kapa<less>^1kapa^+} guarantees the existence of $B$).
Then corollaries \ref{a_limit{e|m(e)=eta(t)[a:=e]}} and
\ref{m(a)=eta(t[kapa:=a])_in_Class(1)} provide the big picture of what happens
in $B \cap \mathbbm{E}$ (indeed, they provide the following characterization
of $B$):

\begin{enumerateromancap}
  \item For any ordinal $\alpha \in B \cap \mathbbm{E}$, $m (\alpha) \in
  [\alpha, \alpha^+)$ and therefore $m (\alpha) = s [B \assign \alpha]$ for
  some \\
  $s \in [B, B^+)$ with $s = \eta s$.
  
  \item For every $s = \eta s \in [B, B^+)$, there are cofinal many ordinals
  in $B$ with $m (\alpha) = s [B \assign \alpha]$.
  
  \item $B$ is the only one ordinal such that I and II hold.
\end{enumerateromancap}

This way, for $\alpha \in B \cap \mathbbm{E}$ (note $B \leqslant \kappa$ by
previous corollary \ref{kapa<less>^1kapa^+}) we have:
\begin{itemizedot}
  \item Case $m (\alpha) = \alpha 2$. Then $\alpha \in \mathbbm{E} \backslash
  A (m (\alpha) [\alpha \assign \kappa])$;
  
  \item Case $\alpha 2 < m (\alpha) \wedge \exists s \in [\alpha 2, m
  (\alpha)) . \eta s + 1 \geqslant m (\alpha)$. Let \\
  $z \assign \min \{s \in [\alpha 2, m (\alpha)) | \eta s + 1 \geqslant m
  (\alpha)\}$. Then $\eta z + 1 = m (\alpha)$ (otherwise the inequalities \\
  $\alpha < z < m (\alpha) < \eta z + 1 \geqslant m (\alpha) + 1$ would imply,
  by $\leqslant_1$-connectedness and proposition
  \ref{eta(t)_m(t)_and_<less>_1}, that $\alpha <_1 m (\alpha) + 1$, which is
  contradictory). Therefore $\alpha \in A (\eta z [\alpha \assign \kappa])
  \backslash A ( {\color{magenta} (\eta z + 1) [\alpha \assign \kappa]}) = A
  (\eta z [\alpha \assign \kappa]) \backslash A (m (\alpha) [\alpha \assign
  \kappa])$.
  
  \item Case $\alpha 2 < m (\alpha) \wedge \forall s \in [\alpha 2, m
  (\alpha)) . \eta s + 1 < m (\alpha)$. Then \\
  $\alpha \in [\underset{s \in [\alpha 2, m (\alpha))}{\bigcap}A (s
  [\alpha \assign \kappa])] \backslash A (m (\alpha) [\alpha \assign
  \kappa])$.
\end{itemizedot}

So our theorems explain us quite well what happens in the segment $[0, B)$,
but what about ordinals bigger that $B$?. Corollary \ref{kapa<less>^1kapa^+}
showed us, for the first time, that the class of ordinals
$\tmmathbf{\tmop{Class} (2) \assign \{\alpha \in \mathbbm{E} | \alpha <_1
\alpha^+ \}}$ is nonempty. We now focus our attention on those ordinals. Our
goals are propositions \ref{gama_in_intersection_then_gama<less>^1gama^+} and
\ref{Class(2)_club_in_kapa} which relate $\tmop{Class} (2)$ with our
hierarchies $(A (t))_{t \in [\kappa, \kappa^+)}$, for $\kappa$ an uncountable
regular ordinal.

\begin{proposition}
  \label{Intersection_A(t)_equiv_alpha<less>^1alpha^+}$\forall \alpha \in
  \tmop{OR} . \alpha <_1 \alpha^+ \Longleftrightarrow \alpha <^1 \alpha^+
  \Longleftrightarrow \alpha \in \bigcap_{t \in [\alpha, \alpha^+)} A (t)$
\end{proposition}

\begin{proof}
  Let $\alpha \in \tmop{OR}$. We already know $\alpha <_1 \alpha^+
  \Longleftrightarrow \alpha <^1 \alpha^+$. We now {
  
  }show $\alpha <^1 \alpha^+ \Longleftrightarrow \alpha \in \bigcap_{t \in
  [\alpha, \alpha^+)} A (t)$.
  
  {\noindent}$\Longrightarrow)$.
  
  Suppose $\alpha <^1 \alpha^+$. Let $t \in [\alpha, \alpha^+)$. Then $\alpha
  <^1 \eta t [\alpha \assign \alpha] + 1 = \eta t + 1$ by $<^1$-connectedness.
  So \ $\alpha \in G (t) = \{\beta \in \tmop{OR} | \tmop{Ep} (t) \cap \alpha
  \subset \beta \leqslant \alpha \wedge \beta \leqslant^1 \eta t [\alpha
  \assign \beta] + 1\} \underset{\text{\tmop{theorem} \ref{teo.A(t)=G(t)}}}{=}
  A (t)$. Since this holds for an arbitrary $t \in [\alpha, \alpha^+)$, we
  have then actually shown $\alpha \in \bigcap_{t \in [\alpha, \alpha^+)} A
  (t)$.
  
  {\noindent}$\Longleftarrow)$.
  
  Suppose $\alpha \in \bigcap_{t \in [\alpha, \alpha^+)} A (t)
  \underset{\text{\tmop{theorem} \ref{teo.A(t)=G(t)}}}{=} \bigcap_{t \in
  [\alpha, \alpha^+)} G (t)$. Then for any $t \in [\alpha, \alpha^+)$, $\alpha
  <^1 \eta t [\alpha \assign \alpha] + 1 = \eta t + 1$, and since $(\eta t +
  1)_{\alpha \leqslant t < \alpha^+}$ is cofinal in $\alpha^+$, then $\alpha
  \leqslant^1 \alpha^+$ by $<^1$-continuity.
\end{proof}

\begin{corollary}
  \label{m(a)=eta(t[kapa:=a])_in_Class(1)}Let $\rho$ be an epsilon number such
  that $\rho <_1 \rho^+$. Then\\
  $\forall t \in [\rho, \rho^+) .t = \eta t \Longrightarrow \{\alpha \in \rho
  \cap \mathbbm{E} | \tmop{Ep} (t) \cap \rho \subset \alpha \wedge m (\alpha)
  = t [\rho \assign \alpha]\}$ is confinal in $\rho$.
\end{corollary}

\begin{proof}
  {\color{orange} Not hard. Left to the reader.}
\end{proof}

\begin{proposition}
  \label{l_j-sequence-restriction}Let $\alpha, \beta \in \mathbbm{E}$, $t \in
  (\alpha, \alpha^+) \cap \tmop{Lim}$, $\beta < \alpha$ and $\tmop{Ep} (t)
  \cap \alpha \subset \beta$. {
  
  }Let $(l_j)_{j \in I}$ be obtained by lemma \ref{l_j-sequence} applied to
  $t$ and $\alpha$. Since $\tmop{Ep} (t) \cap \alpha \subset \beta$ and {
  
  }$t \in (\alpha, \alpha^+) \cap \tmop{Lim}$, then $t [\alpha \assign \beta]
  \in (\beta, \beta^+) \cap \tmop{Lim}$. Then $(l_j [\alpha \assign \beta])_{j
  \in I \cap (\beta + 1)}$ is the sequence obtained by lemma
  \ref{l_j-sequence} applied to $t [\alpha \assign \beta]$ and $\beta$.
\end{proposition}

\begin{proof}
  Long and boring. Left to the reader.
\end{proof}

\begin{proposition}
  \label{Lim_intersec}Let $A \subset \tmop{OR} \ni \alpha$. Then $\tmop{Lim}
  (A) \cap (\alpha + 1) = \tmop{Lim} (A \cap (\alpha + 1))$
\end{proposition}

\begin{proof}
  $'' \subset''$. Let $r \in \tmop{Lim} (A) \cap (\alpha + 1)$. Then there
  exist a sequence $(c_i)_{i \in I} \subset A$ such that {
  
  }$c_i \underset{\tmop{cof}}{\longhookrightarrow} r \leqslant \alpha$. Then
  $(c_i)_{i \in I} \subset (\alpha + 1)$. All this means $(c_i)_{i \in I}
  \subset A \cap (\alpha + 1)$ and $c_i
  \underset{\tmop{cof}}{\longhookrightarrow} r$, i.e., {
  
  }$r \in \tmop{Lim} (A \cap (\alpha + 1))$.
  
  $'' \supset''$. Let $r \in \tmop{Lim} (A \cap (\alpha + 1))$. Then there
  exist a sequence $(c_i)_{i \in I} \subset A \cap (\alpha + 1)$ such that
  $c_i \underset{\tmop{cof}}{\longhookrightarrow} r$; note that since
  $(c_i)_{i \in I} \subset (\alpha + 1)$, then $r \leqslant \alpha$. So $r \in
  \tmop{Lim} (A) \cap (\alpha + 1)$.
\end{proof}

The next proposition is (much) easier to prove using theorem
\ref{teo.A(t)=G(t)} and the properties we already know about the substitutions
$t \longmapsto t [\beta \assign \alpha]$. The reader can do that as an easy
exercise. We provide here our original proof.

\begin{proposition}
  \label{A(t)_arriba_abajo}Let $\beta \in \mathbbm{E}$, $\alpha \in
  \mathbbm{E} \cap \beta$ and $t \in [\beta, \beta^+)$ be such that $\tmop{Ep}
  (t) \cap \beta \subset \alpha$.{
  
  }So $t [\beta \assign \alpha] \in [\alpha, \alpha^+)$. Then $A (t) \cap
  (\alpha + 1) = A (t [\beta \assign \alpha])$.
\end{proposition}

\begin{proof}
  Let $\beta$ and $\alpha \in \mathbbm{E} \cap \beta$ be as stated. We will
  prove by induction on $[\beta, \beta^+)$ the statement $\forall t \in
  [\beta, \beta^+) . \tmop{Ep} (t) \cap \beta \subset \alpha \Longrightarrow A
  (t) \cap (\alpha + 1) = A (t [\beta \assign \alpha])$.

  Let $t \in [\beta, \beta^+)$. Our induction hypothesis is
  
  $\forall r \in [\beta, \beta^+) \cap t. \tmop{Ep} (r) \cap \beta \subset
  \alpha \Longrightarrow A (r) \cap (\alpha + 1) = A (r [\beta \assign
  \alpha])$ \ \ \ \ \ \ \ {\tmstrong{(IH)}}

  Suppose $\tmop{Ep} (t) \cap \beta \subset \alpha$. Let $M \assign \max
  (\tmop{Ep} (t) \cap \alpha)$.

  {\tmstrong{Case}} $\tmmathbf{t \in [\beta, \beta 2] \cap \tmop{Lim}.}$ Then
  $A (t) = (\tmop{Lim} \mathbbm{E}) \cap (M, \beta + 1)$ and so \\
  $A (t) \cap (\alpha + 1) = (\tmop{Lim} \mathbbm{E}) \cap (M, \beta + 1) \cap
  (\alpha + 1) = (\tmop{Lim} \mathbbm{E}) \cap (M, \alpha + 1) = A (t [\beta
  \assign \alpha])$, where the last equality is because $t [\beta \assign
  \alpha] \in [\alpha, \alpha 2] \cap \tmop{Lim}$.

  {\tmstrong{Case}} $\tmmathbf{t = l + 1 \in [\beta, \beta^+).}$ Then clearly
  $\tmop{Ep} (l) \cap \beta \subset \alpha$ and so
  
  {\noindent}$A (l + 1) \cap (\alpha + 1) = \left\{ \begin{array}{l}
    A (l) \cap (\alpha + 1) \text{ iff } l < \pi l + d \pi l\\
    (\tmop{LimA} (l)) \cap (\alpha + 1) \text{ otherwise}
  \end{array} \right.$\\
  \ \ \ \ \ \ \ \ \ \ \ \ \ \ \ \ \ \ \ \ $\underset{\text{\tmop{prop} .
  \ref{Lim_intersec}}}{=} \left\{ \begin{array}{l}
    A (l) \cap (\alpha + 1) \text{ iff } l < \pi l + d \pi l\\
    \tmop{Lim} (A (l) \cap (\alpha + 1)) \text{ otherwise}
  \end{array} \right.$\\
  \ \ \ \ \ \ \ \ \ \ \ \ \ \ \ \ \ \ \ \ \ \ $=$, by our IH and because of
  proposition \ref{Lim_intersec},\\
  \ \ \ \ \ \ \ \ \ \ \ \ \ \ \ \ \ \ \ \ \ \ $= \left\{ \begin{array}{l}
    A (l [\beta \assign \alpha]) \text{ iff } l [\beta \assign \alpha] <
    \pi (l [\beta \assign \alpha]) + d \pi (l [\beta \assign \alpha])\\
    (\tmop{Lim} A (l [\beta \assign \alpha])) \text{ otherwise}
  \end{array} \right.$\\
  \ \ \ \ \ \ \ \ \ \ \ \ \ \ \ \ \ \ \ \ \ \ $= A (l [\beta \assign \alpha]
  + 1) = A ((l + 1) [\beta \assign \alpha]) = A (t [\beta \assign \alpha])$.

  {\tmstrong{Case}} $\tmmathbf{t \in (\beta 2, \beta^+) \cap \tmop{Lim}.}$ Let
  $(l_j)_{j \in I}$ be obtained by lemma \ref{l_j-sequence} applied to $t$ and
  $\beta$; moreover, let $(e_j)_{j \in S}$ be obtained by lemma
  \ref{l_j-sequence} applied to $\pi t$ and $\beta$.

  Subcase $t > \pi t + d \pi t$.
  
  {\noindent}$A (t) \cap (\alpha + 1) = (\tmop{Lim} \{r \leqslant \beta |M < r
  \in \bigcap_{j \in I \cap r} A (l_j)\}) \cap (\alpha + 1)
  \underset{\text{prop.  \ref{Lim_intersec}}}{=}$\\
  \ \ \ \ \ \ \ \ \ \ \ \ \ \ \ \ \ $= \tmop{Lim} (\{r \leqslant \beta |M < r
  \in \bigcap_{j \in I \cap r} A (l_j)\} \cap (\alpha + 1)) = \tmop{Lim} \{r
  \leqslant \alpha |M < r \in \bigcap_{j \in I \cap r} A (l_j)\} =$\\
  \ \ \ \ \ \ \ \ \ \ \ \ \ \ \ \ \ $= \tmop{Lim} \{r \leqslant \alpha |M < r
  \in \bigcap_{j \in I \cap r} (A (l_j) \cap (\alpha + 1))\}
  \underset{\tmop{IH}}{=}$\\
  \ \ \ \ \ \ \ \ \ \ \ \ \ \ \ \ \ $= \tmop{Lim} \{r \leqslant \alpha |M < r
  \in \bigcap_{j \in I \cap r} A (l_j [\beta \assign \alpha])\} =$\\
  \ \ \ \ \ \ \ \ \ \ \ \ \ \ \ \ \ $= \tmop{Lim} \{r \leqslant \alpha |M < r
  \in \bigcap_{j \in (I \cap (\alpha + 1)) \cap r} A (l_j [\beta \assign
  \alpha])\} = A (t [\beta \assign \alpha])$, \\
  where the last equality holds by proposition \ref{l_j-sequence-restriction}
  and because $\tmop{Ep} (t) \cap \beta \subset \alpha$ implies \\
  $\tmop{Ep} (\pi t + d \pi t) \cap \beta \subset \alpha$ and so $t > \pi t +
  d \pi t \Longleftrightarrow t [\beta \assign \alpha] > (\pi t + d \pi t)
  [\beta \assign \alpha] =$\\
  $\pi (t [\beta \assign \alpha]) + d \pi (t [\beta \assign \alpha])$.

  Subcase $t \leqslant \pi t + d \pi t$.
  
  {\noindent}$A (t) \cap (\alpha + 1) = (\tmop{Lim} \{r \leqslant \beta |M < r
  \in \bigcap_{j \in S \cap r} A (e_j)\}) \cap (\alpha + 1)
  \underset{\text{\tmop{prop} . \ref{Lim_intersec}}}{=}$\\
  \ \ \ \ \ \ \ \ \ \ \ \ \ \ \ \ \ $= \tmop{Lim} (\{r \leqslant \beta |M < r
  \in \bigcap_{j \in S \cap r} A (e_j)\} \cap (\alpha + 1)) =$\\
  \ \ \ \ \ \ \ \ \ \ \ \ \ \ \ \ \ $= \tmop{Lim} \{r \leqslant \alpha |M < r
  \in \bigcap_{j \in S \cap r} A (e_j)\} =$\\
  \ \ \ \ \ \ \ \ \ \ \ \ \ \ \ \ \ $= \tmop{Lim} \{r \leqslant \alpha |M < r
  \in \bigcap_{j \in S \cap r} (A (e_j) \cap (\alpha + 1))\}
  \underset{\tmop{IH}}{=}$\\
  \ \ \ \ \ \ \ \ \ \ \ \ \ \ \ \ \ $= \tmop{Lim} \{r \leqslant \alpha |M < r
  \in \bigcap_{j \in S \cap r} A (e_j [\beta \assign \alpha])\} =$\\
  \ \ \ \ \ \ \ \ \ \ \ \ \ \ \ \ \ $= \tmop{Lim} \{r \leqslant \alpha |M < r
  \in \bigcap_{j \in (S \cap (\alpha + 1)) \cap r} A (e_j [\beta \assign
  \alpha])\} = A (t [\beta \assign \alpha])$,\\
  where the last equality holds by proposition \ref{l_j-sequence-restriction}
  (more precisely, since $(e_j)_{j \in S}$ approximates \\
  $\pi t \in [\beta, \beta^+)$, then $(e_j [\beta \assign \alpha])_{S \cap
  (\alpha + 1)}$ approximates $(\pi t) [\beta \assign \alpha] = \pi (t [\beta
  \assign \alpha]) \in [\alpha, \alpha^+)$)), and because $t \leqslant \pi t +
  d \pi t \Longleftrightarrow t [\beta \assign \alpha] \leqslant (\pi t + d
  \pi t) [\beta \assign \alpha] = \pi (t [\beta \assign \alpha]) + d \pi (t
  [\beta \assign \alpha])$.
\end{proof}

\begin{proposition}
  \label{Intersection_A(t)_down_intersection_A(t)_up}Let $\beta \in
  \mathbbm{E}$ and $\alpha \in \mathbbm{E} \cap \beta$. Consider the set of
  ordinals\\
  $M (\beta, \alpha) = \{q \in \beta^+ | \tmop{Ep} (q) \cap \beta \subset
  \alpha\}$. Then $\bigcap_{t \in [\alpha, \alpha^+)} A (t) = (\alpha + 1)
  \cap \bigcap_{t \in M (\beta, \alpha) \cap [\beta, \beta^+)} A (t)$.
\end{proposition}

\begin{proof}
  {\color{orange} Not hard. Left to the reader.}
\end{proof}

\begin{proposition}
  \label{Intersection_alpha_A(t)_may_be_only_alpha}Let $\beta \in
  \mathbbm{E}$, $\alpha \in \mathbbm{E} \cap \beta$ and $M (\beta, \alpha) =
  \{q \in \beta^+ | \tmop{Ep} (q) \cap \beta \subset \alpha\}$. Then
  \begin{enumeratealpha}
    \item $\forall \gamma \in \bigcap_{t \in M (\beta, \alpha) \cap [\beta,
    \beta^+)} A (t) . \gamma \geqslant \alpha$.
    
    \item $\bigcap_{t \in [\alpha, \alpha^+)} A (t) = (\alpha + 1) \cap
    \bigcap_{t \in M (\beta, \alpha) \cap [\beta, \beta^+)} A (t)$; moreover,
    if $\bigcap_{t \in [\alpha, \alpha^+)} A (t) \neq \emptyset$ then\\
    $\bigcap_{t \in [\alpha, \alpha^+)} A (t) = \{\alpha\}$.
  \end{enumeratealpha}
  
\end{proposition}

\begin{proof}
  {\color{orange} Not hard. Left to the reader.}
\end{proof}

\begin{proposition}
  \label{gama_in_intersection_then_gama<less>^1gama^+}Let $\kappa \in
  \mathbbm{E}$ be an uncountable regular ordinal and $\alpha \in \mathbbm{E}
  \cap \kappa$.\\
  Suppose $\gamma \in \bigcap_{t \in M (\kappa, \alpha) \cap [\kappa,
  \kappa^+)} A (t)$. Then $\gamma <^1 \gamma^+$.
\end{proposition}

\begin{proof}
  Let $a_0 \assign \kappa$, $a_{n + 1} \assign \kappa^{a_n}$.\\
  Observe $\forall e \in \mathbbm{E} . \{a_n |n \in \omega\} \subset [\kappa,
  \kappa^+) \cap M (\kappa, e)$. This way, since $\gamma \in \bigcap_{t \in M
  (\kappa, \alpha) \cap [\kappa, \kappa^+)} A (t)$, then $\gamma \in A (a_n) =
  G (a_n) = \{\beta \in \mathbbm{E} | \tmop{Ep} (a_n) \cap \kappa \subset
  \beta \leqslant \kappa \wedge \beta \leqslant^1 \eta a_n [\kappa \assign
  \beta] + 1\}$ for any $n \in \omega$. Therefore for any $n \in \omega$,
  $\gamma \leqslant^1 \eta a_n [\kappa \assign \gamma] + 1$ and since $\gamma
  \leqslant a_n [\kappa \assign \gamma] \leqslant \eta a_n [\kappa \assign
  \gamma] + 1$, then we conclude by $\leqslant^1$-connectedness $\forall n \in
  \omega . \gamma \leqslant^1 a_n [\kappa \assign \gamma)]$. But the sequence
  $(a_n [\kappa \assign \gamma])_{n \in \omega}$ is confinal in $\gamma^+$,
  therefore by $\leqslant^1$-continuity, $\gamma \leqslant^1 \gamma^+$.
\end{proof}

\begin{proposition}
  \label{Class(2)_club_in_kapa}Let $\kappa \in \tmop{OR}$ be an uncountable
  regular ordinal and $\sigma \in \kappa \cap [\varepsilon_{\omega}, \infty]
  \cap \mathbbm{E}$. Then
  \begin{enumerateroman}
    \item $\bigcap_{t \in M (\kappa, \sigma) \cap [\kappa, \kappa^+)} A (t)$
    is club in $\kappa$.
    
    \item $\tmop{Class} (2) \assign \{\alpha \in \mathbbm{E} | \alpha <_1
    \alpha^+ \}$ is club in $\kappa$.
  \end{enumerateroman}
\end{proposition}

\begin{proof}

  {\noindent}$i$.\\
  Take $\kappa \in \tmop{OR}$ as stated and $\sigma \in \kappa \cap
  [\varepsilon_{\omega}, \infty] \cap \mathbbm{E}$. Then, directly from
  corollary \ref{main_[a:=e]_Isomorphism1}, the function \begin{tabular}{l}
    $f : M (\kappa, \sigma) \cap [\kappa, \kappa^+) \longrightarrow [\sigma,
    \sigma^+)$\\
    \ \ \ \ \ \ \ \ \ \ \ \ \ \ \ \ \ \ \ $t \longmapsto t [\kappa \assign
    \sigma]$
  \end{tabular} is an ($<, +, \cdot, \lambda x. \omega^x$)-isomorphism and
  therefore a bijection. This way, $|M (\kappa, \sigma) \cap [\kappa,
  \kappa^+) | = | [\sigma, \sigma^+) | < \kappa$. \ \ \ \ \ \ \ (1)

  On the other hand, by proposition \ref{A(t)_club_in_kapa}, for all $t \in M
  (\kappa, \sigma) \cap [\kappa, \kappa^+)$, $A (t)$ is club in $\kappa$. \ \
  \ \ \ (2)

  From (1), (2) and proposition \ref{Intersection_club_classes}, we conclude
  $\bigcap_{t \in M (\kappa, \sigma) \cap [\kappa, \kappa^+)} A (t)$ is club
  in $\kappa$.

  {\noindent}$i i$.\\
  Direct from previous proposition
  \ref{gama_in_intersection_then_gama<less>^1gama^+} and $i$.
\end{proof}

{\nocite{Bachmann}} \ {\nocite{Bridge}} \ {\nocite{Buchholz1}} \
{\nocite{Buchholz2}} \ {\nocite{Buchholz3}} \ {\nocite{Buchholz4}} \
{\nocite{BuchholzSch{"u}tte1}} \ {\nocite{Carlson1}} \ {\nocite{Carlson2}} \
{\nocite{Pohlers1}} \ {\nocite{Pohlers2}} \ {\nocite{Rathjen1}} \
{\nocite{Sch{"u}tte2}} \ {\nocite{Sch{"u}tteSimpson}} \
{\nocite{Schwichtenberg1}} \ {\nocite{Setzer}} \ {\nocite{Wilken1}} \
{\nocite{Wilken2}} \ {\nocite{Wilken3}} \ {\nocite{GarciaCornejo0}}

\end{document}